\documentclass[10pt]{amsart}
\usepackage{amsmath, amsthm, setspace, graphicx, subfig}  
\usepackage{latexsym, amsfonts, amssymb}
\usepackage[all,cmtip]{xy} 
\makeatletter
\def\subsubsection{\@startsection{subsubsection}{3}%
\z@{.5\linespacing\@plus.7\linespacing}{-.5em}%
{\normalfont\bfseries\itshape}}
\makeatother

\newtheorem{thm}{Theorem}
\newtheorem{cor}[thm]{Corollary} 
\newtheorem{lem}[thm]{Lemma} 
\newtheorem{prop}[thm]{Proposition}
\newtheorem{defn}[thm]{Definition}
\newtheorem{remark}[thm]{Remark} 
\newtheorem{remarks}[thm]{Remarks}

\newtheorem{example}[thm]{Example}
\newtheorem{examples}[thm]{Examples}

\def\a{\alpha}
\def\b{\beta} 
 
\def\d{\delta}
\def\e{\epsilon} 

\def\k{\kappa} 
\def\l{\lambda} 
\def\s{\sigma} 
\def\t{\tau} 
\def\o{\omega}

\def\G{\Gamma}
\def\D{\Delta}
\def\L{\Lambda}

\def\R{\mathbb{R}}

\def\Z{\mathbb{Z}}

\def\p{\partial} 
\def\i{\infty}

\def\cal{\mathcal}
\def\B{\cal{B}}

\def\hB{\hat{\cal{B}}}
\def\pB{{^\prime\cal{B}}}

\def\pX{{^\prime\cal{X}}}

\def\A{\cal{A}}

\def\det{\it{det}} 
\def\supp{\it{supp}}

\def\<{\langle}
\def\>{\rangle}

\def\rotateminus{\reflectbox{\rotatebox[origin=c]{155}{\hspace{.6pt}-}}}
\def\Xint#1{\mathchoice
{\XXint\displaystyle\textstyle{#1}}%
{\XXint\textstyle\scriptstyle{#1}}%
{\XXint\scriptstyle\scriptscriptstyle{#1}}%
{\XXint\scriptscriptstyle\scriptscriptstyle{#1}}%
\!\int}
\def\XXint#1#2#3{{\setbox0=\hbox{$#1{#2#3}{\int}$}
\vcenter{\hbox{$#2#3$}}\kern-.5\wd0}}

\def\cint{\Xint \rotateminus }

\numberwithin{thm}{subsection} 
\numberwithin{equation}{section} 

\begin{document} 
	\title{Operator calculus and the exterior differential complex } 			
	\author[J. Harrison]{J. Harrison 
	\\Department of Mathematics 
	\\University of California, Berkeley} 
	\thanks{The author was partially supported by the Miller Institute for Basic
Research in Science and the Foundational Questions in Physics Institute} 
	\maketitle
 
\begin{abstract}     
	We describe a topological predual to differential forms constructed as an inductive limit of a sequence of Banach spaces. This subspace of currents has nice properties, in that Dirac chains and polyhedral chains are dense, and its operator algebra contains operators predual to exterior derivative, Hodge star, Lie derivative, and interior product.  Using these operators, we establish higher order divergence theorems for net flux of \( k \)-vector fields across nonsmooth boundaries, Stokes' theorem for domains in open sets which are not necessarily regular, and a new fundamental theorem for nonsmooth domains and their boundaries moving in a smooth flow.  We close with broad generalizations of the Leibniz integral rule and Reynold's transport theorem.  
\end{abstract}

\section{Introduction}

Infinite dimensional linear spaces can offer insights and simplifications of analytical problems, since continuous linear operators acting on the space can appear as nonlinear operations on an associated object, such as a finite dimensional manifold. A simple example is pushforward of the time-\( t \) map of the flow of a Lipschitz vector field as seen, for example, in the linear space of Whitney's sharp chains, or in the space of de Rham currents if the vector field is smooth with compact support.  There are numerous ways for creating such spaces, and each has different properties. Although Whitney's Banach spaces of sharp and flat chains each have useful operators, no single Banach space carries an operator algebra sufficient for the most interesting applications. Federer and Fleming's normal and integral currents afford a stronger calculus in many ways, but their use of the flat norm carries the typical flat problems such as lack of continuity of the Hodge star operator.  The hallmark application of their seminal paper is the celebrated  Plateau's Problem of the calculus of variations, but their solutions do not permit M\"{o}bius strips or any surfaces with triple junctions. The author has treated Plateau's Problem (in its more general setting of soap films arising in nature, permitting non-orientability and branchings) as a ``test problem'' for the depth and flexibility of any new calculus of variations.  Long was the search for a linear space which possessed the right balance of size (the smaller the better so as to avoid pathologies,) and operator algebra richness (containing the operators relevant to calculus such as boundary).   This paper covers some of the main properties of the space which ultimately succeeded, and the aforementioned application is contained in a sequel \cite{plateau}.  In this paper, we set up the theory for chains defined in an open set and extend Stokes' Theorem \ref{cor:stokesopen} to this setting. We prove higher order divergence theorems for net flux of $k$-vector fields across nonsmooth boundaries in Corollary \ref{cor:lapl}.   Finally, we present a generalization of the classical ``differentiation of the integral'' for nonsmooth evolving differential chains and forms, and  state and prove a new fundamental theorem of calculus for evolving chains and their boundaries Theorem \ref{thm:Lieder}.

\subsection{Dirac chains and the Mackey topology} \label{sub:history}

Mackey set the stage for the study of topological predual spaces with his seminal papers \cite{mackey1, mackey2} when he and Whitney were both faculty at Harvard.  Most relevant to us is the following result: Let \( (X,Y,\<\cdot,\cdot\>) \) be a dual pair\footnote{A \emph{dual pair} \( (X,Y,\<\cdot,\cdot\>) \) is a pair of vector spaces over \( \R \) together with a bilinear form \( \<\cdot,\cdot\>: X\times Y \to \R\).  The dual pair is separated in \( Y \) (resp. \( X \) ) if the map \( i_Y: Y\to X^* \) induced by \( \<\cdot,\cdot\> \) (resp. \( i_X: X\to Y^* \)) is injective.} separated in \( Y \). According to the Mackey-Arens Theorem \cite{arens}, there exists a finest locally convex topology \( \t \) on \( X \), called the \emph{Mackey topology}, such that the continuous dual \( X' \) is the image of \( Y \) under the injection \( i_Y: Y\to X^* \).  The weak topology  on \( X \) is the coarsest such topology.  In particular, the Mackey topology does not make linear functions continuous which were discontinuous in the weak topology.  
		
The Mackey topology is canonical and rather beautiful in its abstract conception, but it is often useful to find explicit formulations of it. For example, the space of polyhedral \( k \)-chains can be paired with any space of locally integrable differential \( k \)-forms via integration, yielding a dual pair separated in its second component.  The simplest example is given by the Banach space of continuous and bounded differential \( k \)-forms on \( \R^n \). The Mackey topology on polyhedral \( k \)-chains turns out to be the classical mass norm. When paired with Lipschitz forms, the Mackey topology on polyhedral chains is Whitney's sharp norm, originally call the ``tight norm'' topology in his 1950 ICM lecture \cite{whitneyicm}.  However, the boundary operator is not continuous in the sharp norm topology, which is something Whitney needed for his study of sphere bundles\footnote{According to his own account in \cite{whitneytopology}, Whitney wanted to solve a foundational question about sphere bundles and this is why he developed Geometric Integration Theory.  Steenrod \cite{steenrod} solved the problem first, though, and Whitney stopped working on GIT (see \cite{whitneytopology}). However, Whitney's student Eells was interested in other applications of GIT in analysis, and for a time he thought he had a workaround (see \cite{whitney} and \cite{eells}), but the lack of a continuous boundary operator halted progress.}. He then defined the flat norm which does have a continuous boundary operator, and his student Wolfe \cite{wolfethesis} identified the topological dual to flat chains. The flat topology is limited because the Hodge star operator is not closed.
  
Given these problems with the above established spaces, it became clear to the author that a new space was needed.  Our main requirements for such a space are as follows:  It should have the Mackey topology given a space of differential \( k \)-forms, and operators predual to classical operators on forms such as exterior derivative, Lie derivative, Hodge star, pullback, and interior product should be well-defined and continuous. Furthermore, the space should be as small as possible, and possess good topological properties.  In particular it should be Hausdorff, separable and complete.  The main goal of this paper is to define such a space and exhibit some of its properties.  

Let \( U \) be an open subset of a Riemannian \( n \)-manifold \( M \). For \( 0\leq k\leq n \), define a ``Dirac k-chain''\footnote{Previously called  a ``pointed \( k \)-chain''. We prefer not to use the term ``Dirac current'' since it has various definitions in the literature, and the whole point of this work is to define calculus starting with chains, setting aside the larger space of currents.} in \( U \) to be a finitely supported section of \( \L_k(TU) \), the \( k \)-th exterior power of the tangent bundle of \(U\).  
We can write a Dirac \( k \)-chain as a formal sum \( A = \sum_{i=1}^s (p_i; \a_i) \) where \( p_i \in U \) and \( \a_i \in \L_k(T_{p_i}(U)) \). We call \( (p;\a) \) a \emph{simple} \( k \)-\emph{element} if \( \a \) is a simple \( k \)-vector, otherwise \( (p;\a)\) is a \( k \)-\emph{element}. Let \( \A_k = \A_k(U) \) be the vector space of Dirac \( k \)-chains in \( U \). Its algebraic dual \( \A_k^* \) is naturally identified with exterior \( k \)-forms, via evaluation.  Therefore, any vector subspace \( X_k  \subseteq \A_k^* \) can be paired with \( \A_k \), giving a dual pair \( (\A_k, X_k) \) separated in \( X_k \).   Let \( \t_k = \t_k(\A_k, X_k) \) be the Mackey topology on \( \A_k \).  If \( (\A_k, X_k) \) is separated in \( \A_k \) (equivalently if \( \t_k \) is Hausdorff), then then the completion \( \pX_k \) of \( (\A_k, \t_k) \) is a well-defined Hausdorff locally convex topological vector space.  Note that \( (\pX_k)'=X_k \), and the topology \( \pX_k \) is the Mackey topology of the dual pair \( (\pX_k, X_k) \) (Robinson, p.104.)  We call such spaces \( \pX_k \) ``differential \( k \)-chains of class \( X \).''    

In this paper, we pay especial attention to the space \( \B_k  = \B_k^\i(U) \) which is the Fr\'echet space of differential \( k \)-forms defined on \(U \subseteq M\) each with uniform bounds on each of its directional derivatives.  (Geodesic coordinates can be used to define the norms.) The Mackey topology \( \t_k(\A_k, \B_k) \) is Hausdorff, and so elements of the resulting completion \( \pB_k \) are then \emph{differential \( k \)-chains of class \( \B \)}.

\begin{remarks}\mbox{}
\begin{itemize}
	
	\item The \emph{support} of a differential chain \( J \in  \pB_k(U) \) is a well-defined closed subset \( \supp(J) \) of \( U \) (see \S\ref{thm:supportwell}).  

	\item \( \pB_k(\R^n) \) is a proper subspace of \( \B_k'(\R^n) \) (see \cite{topological}).
	
	\item Boundary \( \p \) and \( d \) are continuous on \(\pB_k\) and \(\B_k\), respectively, and are in duality (see Corollary \ref{thm:bod}), so the linear isomorphism \( \B_k \cong \pB_k' \) passes to the de Rham isomorphism of cohomology.  
	
	\item The strong dual topology \( \beta(\B_k, \pB_k) \) on \( \B_k \) coincides with its Fr\'echet space topology according to \cite{topological} (see  Theorem \ref{thm:isot} for a recap.).  
	
	\item In \S\ref{sub:top} and \S\ref{sub:the_top} we shall provide an explicit construction of the topology on \( \pB_k \) (see pp. \pageref{lem:varinj} and \pageref{def:norminopen}).

	\item  In a sequel, we show that coproduct and convolution product (if \( M \) is a Lie group) are closed and continuous in \( \pB \) \cite{algebraic}.     
	
 \end{itemize}

\end{remarks} 

\subsection{Integration} % (fold)
\label{sub:integration}
% subsection integration (end) 
We denote the bilinear form \( \pB_k \times \B_k \to \R \) as the \emph{integral} \( \cint_J \o : = \o(J) \).  We say that a differential \( k \)-chain \( J \) \emph{represents} a classical domain \( D \) of integration if the integrals agree \( \cint_J \o = \int_D \o \) for all \( \o \in \B_k\), where the integral on the right hand side is the Riemann integral.  Examples of domains that permit differential chain representations include open sets (Theorem \ref{thm:opensets}),   polyhedral \( k \)-chains and submanifolds (Proposition \ref{cells}),  vector fields and foliations (Theorem \ref{thm:jX}), and fractals \cite{harrison9, earlyhodge, continuity}.  Roughly speaking, \emph{polyhedral \( k \)-chains} are finite sums of weighted oriented simplices with algebraic cancellation and addition of the weights wherever the simplices overlap (see \S\ref{sub:cells} for a definition). ``Curvy'' versions of polyhedral \( k \)-chains,  called ``algebraic \( k \)-chains'' include representatives of all compact \( k \)-dimensional Lipschitz immersed submanifolds, and Whitney stratified sets (see \S\ref{sub:algebraic_chains}).  We can also represent branched \( k \)-dimensional surfaces such as soap films, bubbles, crystals, lightening (see \cite{plateau}), given a suitable notion of classical integration over such objects.

\subsection{Operator algebras} \label{sub:operators} 
If \( X \) and \( Y \) are locally convex, let \( {\cal L}(X,Y) \) be the space of continuous linear maps from \( X \) to \( Y \), and \( {\cal L}(X) \) the space of continuous operators on \( X \).  Three ``primitive operators'' in \( {\cal L}(\pB) \) are introduced in \S\ref{ssub:creation_and_annihilation_operators}. Each is determined by a suitable\footnote{Sufficiently smooth, with bounded derivatives} vector field \( V \) (\S\ref{sub:vector_fields_ops}).  On simple \( k \)-elements, these operators are defined as follows:  \emph{extrusion} \( E_V (p;\a) := (p; V(p)\wedge \a) \), \emph{retraction} \( E_V^\dagger(p; v_1 \wedge \cdots \wedge v_k) := \sum_{i=1}^k (-1)^{i+1} \<V(p),v_i\> v_1 \wedge \cdots \hat{v_i} \cdots \wedge v_k\) and \emph{prederivative} (first defined in \cite{soap} for constant vector fields) \( P_V(p;\a) = \lim_{t \to 0} ((\phi_t(p); \phi_*\a)- (p;\a))/t \), where \( \phi_t \) is the time \( t \) flow of \( V \), and \( \phi_* \) is the appropriate pushforward.  We call \( P_V (p;\a) \) a \emph{\( k \)-element of order one}. Whereas \( (p;\a) \) is a higher dimensional version of a Dirac delta distribution, \( P_V (p;\a) \) corresponds to the (negative of the\footnote{Weak derivative is the ``wrong way around'' because its definition is motivated by integration by parts, which treats distributions as ``generalized functions,'' whereas here we treat them as ``generalized domains.''}) weak derivative of a Dirac delta (see \S\ref{ssub:prederivative} for a full development).  We call such elements \emph{\( k \)-elements of order $1$ at \( p \)}.  We may recursively define \emph{ \( k \)-elements of order \( s \) at \( p \)} by applying \( P_V \) to \( k \)-elements of order \( s-1 \). 
     
Each of these primitive operators on differential chains dualizes to a classical operator on forms. In particular, \(i_V \o = \o E_V, \quad V^\flat \wedge \o = \o E_V^\dagger, \quad \text{ and }  L_V \o =  \o P_V \) where \( i_V \) is interior product, \( V^\flat \) is the \( 1 \)-form obtained by musical isomorphism, and \( L_V \) is the Lie derivative.  

The resulting integral equations 
\begin{align} 
	\cint_{E_V J} \o &= \cint_J i_V \o  &&\mbox{[Change of dimension I]} \label{ev} \\
	\cint_{E_V^\dagger J} \o &= \cint_J V^\flat \wedge \o &&\mbox{[Change of dimension II]} \label{evt}\\ 
	\cint_{P_V J} \o &= \cint_J L_V \o &&\mbox{[Change of order I]}\label{pv} 
\end{align}
are explored for constant vector fields \( V\) in \( \R^n \) in \S\ref{ssub:creation_and_annihilation_operators} and for smooth vector fields in \S\ref{sub:vector_fields_ops}. 

Note that it is far from obvious that the primitive operators are continuous and closed in the spaces of differential chains \( \pB \), but once this is established, integral relations are simple consequences of duality.  The veracity of such equalities depends on subtle relationships between the differentiability class of the vector field \( V \), the differential form \( \o \), and the chain \( J \).

The primitive operators generate other operators on chains via composition.  Boundary can be written \( \p = \sum_{i=1}^n P_{e_i} E_{e_i}^\dagger  \) where \( \{e_i\} \) is an orthonormal basis of \( \R^n \).  One of the most surprising aspects of a nonzero $k$-dimensional Dirac chain is that its boundary is well-defined and nonzero if \( k \ge 1 \).  In particular, a vector \( v \in \R^n \) has a well-defined infinitesimal boundary in \( \pB_0 \). The resulting Stokes' Theorem \ref{cor:stokes} \[ \cint_{\p J} \o = \cint_J d \o \] for \( J \in \pB_k \) and \( \o \in \B_{k-1} \) therefore yields an infinitesimal version on Dirac chains. ``Perpendicular complement'' \(\perp = \Pi_{i=1}^n (E_{e_i} + E_{e_i}^\dagger)\) dualizes to Hodge-star on \( \B \), and so we can write down (see Theorem \ref{thm:perpcont} below, which first appeared in \cite{hodge}) an analogue to Stokes' theorem for Hodge-star: \[ \cint_{\perp J} \o = \cint_J \star \o. \]   
  
In \S \ref{ssub:pushforward} and \S\ref{sub:multiplication_by_a_function}, we find two more fundamental operators, namely, \emph{pushforward} \( F_* \) and \emph{multiplication by a function} \( m_f \), where \( F:U \to U' \) is a smooth map and \( f:U \to \R \) is a smooth function with bounded derivatives. These dualize to pullback and multiplication by a function, respectively, giving us change of variables (Corollary \ref{cor:pull}) and change of density  (Theorem \ref{thm:funcon}). Notably, bump functions are permitted functions for the operator \( m_f \), as is the unit \( f \equiv 1 \), giving us partitions of unity.  Given ``uniform enough'' bump functions, the partition of unity sum \( 1 = \sum f_i \) carries over to differential chains in Theorem  \ref{thm:pou}:  \( J = \sum (m_{f_i} J) \).  
 
Note that vector fields can be represented as differential chains, and we may freely apply pushforward to their representatives, even when the map is not a diffeomorphism. The end result will still be a differential chain, although it may no longer represent a vector field.     

In \S\ref{sub:a_stokes} we prove two new fundamental theorems for differential chains in a flow. Given a differential \( k \)-chain \( J_0 \) and the flow \( \phi_t \) of a smooth vector field \( V \), we define a differential \( k \)-chain \( \{J_t\}_a^b \) satisfying \( \p\{J_t\}_a^b = \{\p J_t\}_a^b \) (See Figure  \ref{fig:MovingChains}  and \S\ref{sub:a_stokes}.)  Denoting \( J_t:={\phi_t}_* J_0 \), we have the following relations: \[ \cint_{J_b} \o - \cint_{J_a} \o = \cint_{\{J_t\}_a^b} L_V \o    \text{  (Fundamental theorem for differential chains in a flow)}\] and its corollary \[ \cint_{\{J_t\}_a^b} d L_V \o = \cint_{\p J_b} \o - \cint_{\p J_a} \o \text{  (Stokes' theorem for differential chains in a flow,)} \] given a form \( \o \) of corresponding dimension.

 cknowledgements: I am grateful to Morris Hirsch for his support of this work since its inception, and his many helpful remarks and historical insights as the theory evolved. He has been a true mentor and colleague, and he and his wife Charity have been dear, lifelong friends.  I thank Harrison Pugh for many helpful conversations, and for proofreading and editing this manuscript.  I also thank James Yorke and Steven Krantz who have contributed to important and useful discussions throughout the development of this work, while Robert Kotiuga, Alain Bossavit, and Alan Weinstein are thanked for their early interest and helpful comments along the way. I thank my students, especially Patrick Barrow, as well as the undergraduates in the experimental ``chainlets'' course at Berkeley in the spring of 1996, for their feedback, patience, and enthusiasm.

\section{The topological vector space of differential chains of class \( \B \)}\label{sub:top} 

We begin our work with differential chains of class \( \B \) supported in \( \R^n \) until \S\ref{sec:chains_in_open_set} where we begin to consider chains supported in open subsets   \( U \subseteq \R^n \). In \S\ref{sec:chains_in_manifolds} we sketch the basic steps for extending the theory to Riemannian manifolds.  One could begin with manifolds directly, but one finds themselves quickly awash\footnote{cf Sobolev spaces on manifolds.  It can be done, but it gets a bit messy.} in exponential maps and parallel transports.  The difficulty in the proofs becomes keeping track of these, rather than the underlying analysis.  So for the sake of clarity, we give the proofs in the Euclidian case.

Recall the definition of the space \(\A_k\) of Dirac $k$-chains  above.  What follows is our promised explicit definition of the Mackey topology \( \t_k \) on \( \A_k \).

\subsection{Mass norm}\label{ssub:differential_chains}

Choose an inner product \( \<\cdot,\cdot\> \) on \( \R^n \).  The inner product extends to \( \L_k(T\R^n) \) as follows: Let \( \<u_1 \wedge \cdots \wedge u_k, v_1 \wedge \cdots \wedge v_k \> := \det( \<u_i,v_j\>) \). (We sometimes use the notation \( \<\a,\b\>_\wedge = \<\a,\b\> \).) The \emph{mass} of a simple \( k \)-vector \( \a = v_1 \wedge \cdots \wedge v_k \) is given by \( \|\a\| := \sqrt{\<\a,\a\>} \). The mass of a \( k \)-vector \( \a \) is \[\|\a\|_{B^0}=   \|\a\|  := \inf\left\{\sum_{j=1}^{N} \|(\a_i)\| : \a_i \mbox{ are simple, } \a = \sum_{j=1}^N \a_i \right\}. \] Define the \emph{mass} of a \( k \)-element \( (p;\a) \) by \(\|(p;\a)\|_{B^0} := \|\a\|_{B^0} \).  The \emph{mass} of a Dirac \( k \)-chain \( A=\sum_{i=1}^{m} (p_i; \a_i) \) is defined by \[ \|A\|_{B^0} := \sum_{i=1}^{m} \|(p_i; \a_i)\|_{B^0}. \]  Another choice of inner product leads to an equivalent mass norm for each \( k \).

\subsection{Difference chains}\label{sub:difference_chains}
\begin{defn}\label{def:morenotation}
Given a \( k \)-element \( (p;\a) \) and \( u \in \R^n \), let \( T_u(p;\a) := (p+u;\a) \) be translation through \( u \), and \(\D_u(p;\a) := (T_u -I)(p; \a) \). Let \( S^j =  \) be the \( j \)-th \emph{symmetric power} of \(  \R^n \). Denote the symmetric product in \( S \) by \( \circ \). Let \( \s = \s^j = u_1 \circ \dots \circ u_j \in S^j \) with \( u_i \in \R^n, i = 1, \dots, j \). Recursively define \( \D_{u\circ \s^j}(p; \a) := (T_u - I)(\D_{\s^j}(p; \a)) \). Let  \( \|\s^j\| := \|u_1\| \cdots \|u_j\| \) and \( |\D_{\s^j}(p; \a)|_{B^j} := \|\s^j\| \|\a\| \).  Define \( \D_{\s^0} (p;\a) := (p;\a) \), to keep the notation consistent.  Finally, when there is no ambiguity, we sometimes use the abbreviated notation \( \D_{\s^j} = \D_{\s^j}(p;\a) \).  We call \( \D_{\s^j} (p; \a)\) a \emph{\( j \)-difference \( k \)-chain}. 
\end{defn} 
                             
The geometric interpretation of \( \D_{\s^j} (p;\a) \) is as follows: \(\D_{\s^j}(p;\a)\) is the Dirac chain supported on the vertices of a \( j \)-dimensional parallelepiped (possibly degenerate, if the \( u_i \) are linearly dependent), with a copy of \( \a \) at each vertex.  The sign of \( \a \) alternates, but note that \( \D_{\s^j} (p;\a) \ne 0 \), as long as \( \a \ne 0 \) and \(\s \ne 0 \).

\subsection{The \( B^r \) norm on Dirac chains}

\begin{defn}\label{def:norms} 
	For \( A \in \A_k \) and \( r \ge 0 \), define the norm \[ \|A\|_{B^r} := \inf \left \{ \sum_{i=1}^{m} \|\s^{j(i)}_i\|\|\a_i\|: A = \sum_{i=1}^{m} \D_{\s_i^{j(i)}}(p_i;\a_i) ,\,\, 0\leq j(i)\leq r \right\}\] where each \( \s_i^{j(i)} \in S^{j(i)}(\R^n) \), \( p_i \in \R^n \), \( \a_i \in \L_k(\R^n) \), and \(m\) is arbitrary.  That is, we are taking the infimum over all ways to write \( A \) as a sum of difference chains, of ``order'' up to \( r \).  
\end{defn} 

\begin{remarks} \mbox{} 
	\begin{itemize} 
		\item It is not obvious that \( \|\cdot\|_{B^r} \) is a norm on \( \A_k \), this is proved in Theorem \ref{thm:normspace}. 
		\item However, it is immediate that \( \| A \|_{B^r} < \i \) for all \( A \in \A_k \) since \( \| A \|_{B^r} \le \|A\|_{B^0} < \i \).  It is also immediate that \(\|\cdot\|_{B^r}\) is a semi-norm.  
		\item It is not important to know the actual value of \( \|A\|_{B^r} \) for a given chain \( A \).  Well-behaved upper bounds suffice in our proofs and examples, and these are not usually hard to find.  
	\end{itemize} 
\end{remarks}

\subsection{Differential forms of class \( B^r \)}\label{sub:cochains}

An element \( \o \in \A_k^* \) acts on Dirac chains, and thus can be treated as an exterior \( k \)-form.  In this section we assume that \( \o \) is bounded and measurable. 

\begin{defn}\label{eq:Bj} 
	Define \[  |\o|_{B^j} := \sup \{ |\o(\D_{\s^j} (p;\a))|: \|\s\|\|\a\| = 1 \}   \mbox{ and } \|\o\|_{B^r} = \max\{|\o|_{B^0}, \dots, |\o|_{B^r} \}.  \]
\end{defn}  

We say that \( \o \in \A_k^* \) is a differential form of \emph{class} \( B^r \) if \( \|\o\|_{B^r} < \i \).   Denote the space of differential \( k \)-forms of class \( B^r \) by \( \B_k^r \). Then \( \|\cdot\|_{B^r} \) is a norm on \( \B_k^r \). (It is straightforward to see that \( |\o|_{B^j} \) satisfies the triangle inequality and homogeneity for each \( j \ge 0 \).  If \( \o \ne 0 \), there exists \( (p;\a) \) such that \( \o(p;\a) \ne 0 \). Therefore \( |\o|_{B^0} > 0 \), and hence \( \|\o\|_{B^r} > 0 \).) In \S\ref{sec:isomorphism_theorem} we show that \( \B_k^r \) is topologically isomorphic to \( (\A_k, \|\cdot\|_{B^r})' \).

\begin{lem}\label{lem:ineq} 
	Let \( A \in \A_k \) and \( \o \in \A_k^* \) a differential \( k \)-form. Then \( |\o(A)| \le \|\o\|_{B^r} \|A\|_{B^r}  \) for all \( r \ge 0 \).
\end{lem}

\begin{proof}
	Let \( A \in \A_k \) and \( \e > 0 \). By the definition of \( \|A\|_{B^r} \), there exist \( \s_i^{j(i)}\) with \( j(i)\in \{0,\dots, r\} \) and \( k \)-elements \( (p_i; \a_i) \), \(i = 1, \dots, m\), such that \( A = \sum_{i=1}^{m} \D_{\s_i^{j(i)}}(p_i;\a_i) \) and \( \|A\|_{B^r} > \sum_{i=1}^{m} \|\s_i^{j(i)}\|\|\a_i\| - \e. \) Thus
	\begin{align*}
 	|\o(A)| \le \sum_{i=1}^{m} |\o(\D_{\s_i^{j(i)}}) | \le \sum_{i=1}^{m} |\o|_{B^{j(i)}} \|\s_i^{j(i)}\|\|\a_i\| \le \|\o\|_{B^r} (\|A\|_{B^r} + \e).
	\end{align*}  
\end{proof}

\subsection{Differential forms of class \( C^r \) and \( C^{r-1+Lip} \)}\label{sub:differential_forms}

\begin{defn}
	If \( \o \) is \( r \)-times differentiable, let \[ |\o|_{C^j} := \sup \left\{\frac{|L_{\s^j}\o(p;\a)|}{\|\s\|\|\a\|} \right\}, \] and \( \|\o\|_{C^r} := \max\left\{|\o|_{C^0}, |\o|_{C^1}, \dots, |\o|_{C^r} \right \} \), where \(L_{\s^j}\) denotes the \(j\)-th directional derivative in the directions \( u_1,\dots,u_j\). 
\end{defn} 

\begin{lem}\label{est} 
	If \( \o \) is an exterior form with \( |\o|_{C^r} < \i \), then \( \left| \o(\D_{{\s^r}}(p;\a)) \right| \le \|\s\|\|\a\||\o|_{C^r} \) for all \( r \)-difference \( k \)-chains \( \D_{{\s^r}} (p;\a) \) where \( \a \) is simple.
\end{lem}

\begin{proof} 
	By the Mean Value Theorem, there exists \( q = p + su \) such that \[\frac{\o(p +tu;\a) - \o (p;\a)}{\|u\|} = L_u \o (q;\a). \] It follows that \( \left|\o(\D_u(p;\a)) \right| \le \|u\|\|\a\||\o|_{C^1} \).

	The proof proceeds by induction. Assume the result holds for \( r-1 \), and \(|\o|_{C^r} < \i \). Suppose \( \s = \{u_1, \dots, u_r\} \), and let \( \hat{\s} = \{u_1,\dots, u_{r-1}\} \). Since \( |\o|_{C^r} < \i \), we may apply the mean value theorem again to see that
	\begin{align*} 
		\left|\o(\D_{{\s^r}}(p;\a)) \right| &= \left| ( T_{u_r}^* \o - \o)(\D_{\hat{\s}^{r-1}}(p;1)) \right| \\&\le \|u_1\| \cdots \|u_{r-1}\| |T_{u_r}^* \o - \o |_{C^{r-1}} \\& \le \|\s\||\o|_{C^r}. 
	\end{align*} 
\end{proof}

Let \[ |\o|_{Lip} := \sup_{u \ne 0} \left\{ \frac{|\o(p+u;\a) - \o(p;\a)|}{\|u\|}: \|\a\|= 1 \right\}. \] If \( \o \) is \( (r-1) \)-times differentiable, let \[ |\o|_{C^{r-1+Lip}} := \sup_{\|u_i\|=1} \{ |L_{u_{r-1}} \circ \cdots \circ L_{u_1}\o|_{Lip} \}, \] and \( \|\o\|_{C^{r-1+Lip}} := \max\{|\o|_{C^0}, |\o|_{C^1}, \dots, |\o|_{C^{r-1}}, |\o|_{C^{r-1+Lip}} \} \).

\subsection{The \( B^r \) norm is indeed a norm on Dirac chains} \label{sub:the_(_)}

\begin{thm}\label{thm:normspace} 
	\( \| \cdot \|_{B^r} \) is a norm on Dirac chains \(\A_k(\R^n) \) for each \( r \ge 0 \).
\end{thm}
                                           
\begin{proof} 
	Suppose \( A \ne 0 \) where \( A \in \A_k \). It suffices to find nonzero \( \o \in \A_k^* \) with \( \o(A) \ne 0 \) and \( \|\o\|_{B^r} < \i \). By Lemma \ref{lem:ineq} we will then have \[ 0 < |\o(A)| \le \|\o\|_{B^r}\|A\|_{B^r} \implies \|A\|_{B^r} > 0. \]

	Now \( \supp(A) = \{p_0, \dots, p_N \} \). Without loss of generality, we may assume that \( A(p_0) = e_I \) for some multi index \( I \). Choose a smooth function \( f \) with compact support that is nonzero at \( p_0 \) and vanishes at \( p_i, 1 \le i \le N \). Let \[ \o(p; e_J) := \begin{cases} f(p), & J = I \\ 0, &J \ne I \end{cases} \] and extend to Dirac chains \( \A_k \) by linearity.

	Then \( \o(A) = \o(p_0; e_I) = 1 \ne 0 \). Since \( f \) has compact support, \( \|f\|_{C^r} < \i \). We use this to deduce \( \|\o\|_{B^r} < \i \). This reduces to showing \( |\o|_{B^j} < \i \) for each \( 0 \le j \le r \). This  reduces to showing \( |\o(\D_{\s^j} (p;e_I))| \le \|\s\| \|f\|_{C^j} \). By Lemma \ref{est} \(|\o(\D_{\s^j} (p;e_I))| = \left | f(\D_{\s^j}(p;1)) \right| \le \|\s\| |f|_{C^j} \le \|\s\|\|f\|_{C^j} < \i \).
\end{proof}

\begin{defn}    
	Let \( \hB_k^r \) be the Banach space obtained upon completion of \( \A_k \) with the \( B^r \) norm. Elements of \( \hB_k^r \) are called \emph{\textbf{differential \( k \)-chains of class \( B^r \)}}.  	
\end{defn}

\subsection{Characterizations of the \( B^r \) norms}\label{sub:characterization_of_the_norms} 

\begin{lem}\label{lem:trnsl} 
	\( \|\D_v J\|_{B^{r+1}} \le \|v\|\|J\|_{B^r} \) for all differential chains \( J \in \hB_k^r \) and \( v \in \R^n \). 
\end{lem}
                                 
\begin{proof} 
	Since Dirac chains \( \A_k \) are dense in \( \hB_k^r \), it suffices to prove this for \( A \in \A_k \). Let \( \e > 0 \). We can write \( A = \sum_{i=1}^m \D_{\s_i^{j(i)}}(p_i; \a_i) \) as in the proof of Lemma \ref{lem:ineq}, such that \[ \|A\|_{B^r} > \sum_{i=1}^m \|\s_i^{j(i)}\|\|\a_i\| - \e.\] Then
	\begin{align*} 
		\|\D_v A\|_{B^{r+1}}= \|T_v A - A\|_{B^{r+1}} \le \sum_{i=1}^m \|T_v \D_{\s_i^{j(i)}} - \D_{\s_i^{j(i)}} \|_{B^{r+1}} &\le \sum_{i=1}^m | T_v \D_{\s_i^{j(i)}} - \D_{\s_i^{j(i)}} |_{B^{j+1}}   \\
		&\le \|v\|\sum_{i=1}^m \|\s_i^{j(i)}\|\|\a_i\| \\
		&< \|v\|(\|A\|_{B^r} + \e ). 
	\end{align*} 
\end{proof} 

\begin{lem}\label{lar} 
	The norm \(\|\cdot\|_{B^r} \) is the largest seminorm \( |\cdot|' \) on Dirac chains \( \A_k \) such that \( |\D_{\s^j} (p;\a)|' \le \|\s\|\|\a\| \) for all \( j \)-difference \( k \)-chains \( \D_{\s^j} (p;\a) \), \( 0 \le j \le r \). 
\end{lem} 
                     
\begin{proof}
	First observe that the \( B^r \) norm itself satisfies this inequality by its definition. On the other hand, suppose \( |\quad|' \) is a seminorm satisfying \( |\D_{\s^j} (p;\a)|' \le \|\s\|\|\a\| \). Let \( A \in \A_k \) be a Dirac chain and \( \e > 0 \). We can write \( A = \sum_{i=1}^m \D_{\s_i^{j(i)}}(p_i;\a_i) \) as in the proof of Lemma \ref{lem:ineq}, with \( \|A\|_{B^r} > \sum_{i=1}^m \|\s_i^{j(i)}\|\|\a_i\| - \e \). Therefore, by the triangle inequality, \( |A|' \le \sum_{i=1}^m |\D_{\s_i^{j(i)}}(p_i;\a_i)|' \le \sum_{i=1}^m \|\s_i^{j(i)}\|\|\a_i\| < \|A\|_{B^r} + \e \). Since this estimate holds for all \( \e > 0 \), the result follows. 
\end{proof} 

\begin{thm}\label{lartheorem} 
	The norm \( \|\cdot\|_{B^r} \) is the largest seminorm  \( |\cdot|' \) on Dirac chains \( \A_k \) such that 
	\begin{itemize} 
		\item \( |A|' \le \|A\|_{B^0} \) 
		\item \( |\D_u A |' \le \|u\| \|A\|_{B^{r-1} } \). 
  \end{itemize} 
  for all \( r \ge 1 \) and \( A \in \A \). 
\end{thm}

\begin{proof} 
	\( \|\cdot\|_{B^r} \) satisfies the first part since the \( B^r \) norms are decreasing on chains. It satisfies the second inequality by Lemma \ref{lem:trnsl}. On the other hand, suppose \( |\cdot|' \) is a seminorm satisfying the two conditions. For \( j=0 \), we use the first inequality \( | (p;\a)|' \le \| (p;\a) \|_{B^0} = \| \a\| \).  Fix \( 0 < j \le r \). Using induction and recalling the notation \( \hat{\s} \) from Lemma \ref{est}, it follows that 
	\begin{align*} 
		|\D_{\s^j} (p;\a)|' = |\D_{\hat{\s}^{j-1},u_1} (p;\a)|' \le \|u_1\|\|\D_{\hat{\s}^{j-1}} (p;\a)\|_{B^{j-1} } \le \|\s\|\|\a\|. 
	\end{align*} 
	Therefore, the conditions of Lemma \ref{lar} are met and we deduce \( |A|' \le \|A\|_{B^r} \) for all Dirac chains \( A \). 
\end{proof}

Let \( \A = \oplus_{k=0}^n \A_k \). Recall \( \D_u A = T_u A - A \) where \( T_u (p;\a) = (p+u; \a) \).
     
\begin{cor}\label{cor:opr} 
	If \( T: \A \to \A \) is an operator satisfying \( \|T(\D_{\s^j} (p;\a))\|_{B^r} \le C \|\s\|\a\| \) for some constant \( C > 0 \) and all \( j \)-difference \( k \)-chains \( (\D_{\s^j} (p;\a) \) with \( 0 \le j \le s \) and \( 0 \le k \le n \), then \( \|T(A)\|_{B^r} \le C\|A\|_{B^s} \) for all \( A \in \A_k \).
\end{cor}

\begin{proof} 
	Let \( |A|' = \frac{1}{C} \|T (A)\|_{B^r} \). Then \( |A|' \) is a seminorm, and the result follows by Theorem \ref{lartheorem}.
\end{proof}

\subsection{Isomorphism of differential \( k \)-forms and differential \( k \)-cochains in the \( B^r \) norm } \label{sec:isomorphism_theorem} 

In this section we prove that the Banach space \( \B_k^r \) of differential forms is topologically isomorphic to the Banach space \( (\hB_k^r)' \) of \emph{differential cochains}. 
\begin{thm}\label{lem:seminorm} 
	If \( \o \in \B_k^r \) is a differential form, then \( \|\o\|_{B^r} = \sup_{ 0 \ne A \in \A_k} \frac{ \left| \o(A) \right|}{\|A\|_{B^r}} \).
\end{thm}

\begin{proof} 
	We know \( \left| \o(A)\right| \le \|\o\|_{B^r} \|A\|_{B^r}\) by Lemma \ref{lem:ineq}. On the other hand,  
	\begin{align*} 
		|\o|_{B^j} = \sup \frac{|\o(\D_{\s^j}(p;\a))|}{|\D_{\s^j}(p;\a)|_{B^j}} \le \sup \frac{|\o(\D_{\s^j}(p;\a))|}{\|\D_{\s^j}(p;\a)\|_{B^r}} \le \sup \frac{ | \o(A)|}{\|A\|_{B^r}}.
	\end{align*} 
 	It follows that \( \|\o\|_{B^r} \le \sup \frac{ |\o(A)|}{\|A\|_{B^r}}. \)
\end{proof}

\begin{thm}\label{lem:oncemore} 
	If \( \o \in \B_k^{r+1} \) is a differential \( k \)-form and \( r \ge 1 \), then \( \o \) is \( r \)-times differentiable and its \( r \)-th order directional derivatives are Lipschitz continuous with 	\( \|\o\|_{B^r} = \|\o\|_{C^{r-1+Lip}} \).
\end{thm}

This is a straightforward result in analysis.  Details may be found in \cite{thesis} or earlier versions of this paper on the arxiv.

\subsection{Integration}\label{sub:a_new_integral}

Although the space \(  \B_k^r \) of differential forms is dual to \( \hB_k^r \), and thus we are perfectly justified in using notation \( \o(J) \), this can become confusing since we are no longer evaluating \( \o \) at a set of points, as we were with elements of \( \A_k \).  Instead we shall think of \( \o(J) \) as ``integration'' over \( J \) and denote it as \( \cint_J \o \).   Indeed, the integral notation is justified by the approximation of \( \o(J) \) by its analogue to Riemann sums.  That is, \[ \cint_J \o = \lim_{i \to \i} \o(A_i), \] where \( A_i \in \A_k \) and \( A_i \to J \) in \( \hB_k^r \).  The integral over differential chains is therefore a nondegenerate bilinear pairing, and it is continuous in both variables.  We shall see later in Theorems \ref{thm:opensets} \ref{cells} \ref{thm:jX}, and \ref{thm:subm}  that when \( J \) represents a classical domain \(A\) of integration such as an open set, polyhedral chain, vector field or submanifold, respectively, the integrals \(\int_A \o \) and \( \cint_J \o\) agree\footnote{This is in fact our definition of the word ``represent,'' but we construct these chains first, then show the integrals agree.  That is, they are not found implicitly.  Such chains are also unique when they exist.}.  

In this new notation, Theorem \ref{lem:seminorm}, for example, becomes 

\begin{equation}\label{basic}
	\left|\cint_J \o \right| \le \|J\|_{B^r}\|\o\|_{B^r}. 
\end{equation}

\subsection{Representatives of open sets} 

We say that an \( n \)-chain \( \widetilde{U} \) \emph{represents} an open set \( U \), if \( \cint_{\, \widetilde{U}} \o = \int_U \o \) for all forms \( \o \in \B_n^1 \), where the integral on the right hand side is the Riemann integral.) We show there exists a unique element \( \widetilde{U} \) in \( \hB_n^1 \) representing \( U \)  for each bounded, open subset \( U \subset  \R^n \). 

Let \( Q \) be an \( n \)-cube in \( \R^n \) with unit side length. For each \( j \ge 1 \), subdivide \( Q \) into \( 2^{nj} \) non-overlapping sub-cubes \( Q_{j_i} \), using a binary subdivision. Let \( q_{j_i} \) be the midpoint of \( Q_{j_i} \) and \( P_j = \sum_{i=1}^{2^{nj}} (q_{j_i}; 2^{-nj} \mathbb 1) \) where \( \mathbb 1 = e_1 \wedge \cdots \wedge e_n \). Then \( \|P_j\|_{B^0} = 1 \).

\begin{prop}\label{lem:cuberep} 
	The sequence \( \{P_j\} \) is Cauchy in the \( B^1 \) norm. 
\end{prop}

\begin{proof}
	Rather than write out some inequalities with messy subscripts, we shall describe what is going on and why the result is true. We would like to estimate \( \|P_j - P_{j+1} \|_{B^1} \). Both Dirac \( n \)-chains \( P_j \) and \( P_{j+1} \) have mass one. They are supported in sets of points placed at midpoints of the binary grids with sidelength \( 2^{-j} \) and \( 2^{-(j+1)} \), respectively. Each of the \( n \)-elements of \( P_j \) has mass \( 2^{-nj} \) and those of \( P_{j+1} \) have mass \( 2^{-n(j+1)} \). The key idea is to think of each \( n \)-element of \( P_j \) as \( 2^n \) duplicate copies of \( n \)-elements of mass \( 2^{-n(j+1)} \). This gives us a \( 1-1  \) correspondence between the \( n \)-elements of \( P_j \) and those of \( P_{j+1} \). We can choose a bijection so that the distance between points paired is less than \( 2^{-j+1} \). Use the triangle inequality with respect to \( P_j -P_{j+1} \) written as a sum of differences of these paired \( n \)-elements and Lemma \ref{lem:trnsl} to obtain \( \|P_j - P_{j+1} \|_{B^1} \le 2^{-j+1} \). It follows that \( \{P_j\} \) is Cauchy in the \( B^1 \) norm since the series \( \sum 2^{-j+1} \) converges. 
\end{proof}

Denote the limit  \( \widetilde{Q} := \lim_{j \to \i} P_j \) in the \( B^1 \) norm.  Then \( \widetilde{Q} \in \hB_n^1 \) is a well-defined differential \( n \)-chain.  By continuity of the differential chain integral and the definition of the Riemann integral as a limit of Riemann sums, we have \[ \cint_{\widetilde{Q}} \o = \lim_{j \to \i} \cint_{\sum (p_{j_i}; \a_{j_i})} \o = \int_Q \o.\] That is, \( \widetilde{Q} \) represents \( Q \).
 
\begin{thm}\label{thm:opensets} 
	Let \( U \subset \R^n \) be bounded and open.  There exists a unique differential \( n \)-chains \( \widetilde{U} \in \hB_n^1 \) such that \( \cint_{\widetilde{U}} \o = \int_U \o \) for all \( \o \in \B_n^1 \), where the integral on the right hand side is the Riemann integral. 
\end{thm}

\begin{proof}
	Let \( U = \cup Q_i \) be  a Whitney decomposition of \( U \) into a union of non-overlapping \( n \)-cubes.  We first show that the differential chains \( \widetilde{Q_i} \) form a Cauchy sequence in \( \hB_n^1 \).  Now \( \|\widetilde{Q_i} - \widetilde{Q_j}\|_{B^1}  \) tends to zero as \( i, j \to \i \) since each term is bounded by the volume of the cubes in a small neighborhood of the boundary of \( U \), a bounded open set.  Therefore \( \sum_{i=1}^{\i} \widetilde{Q_i} \) converges to a well-defined chain in \( \hB_n^1 \).  We next show that \( \widetilde{U} := \sum \widetilde{Q_i} \). 
	 Suppose \( \o \in \B_n^1(\R^n) \). Then \( \cint_{\widetilde{U}} \o = \lim_{N \to \i} \cint_{\sum_{i=1}^N \widetilde{Q_i}} \o  = \int_U \o \) by the definition of the Riemann integral. Uniqueness follows immediately from Hahn-Banach.
\end{proof}

\subsection{Polyhedral chains}\label{sub:cells}

We can also represent \( k \)-cells, and using these, we define \emph{polyhedral chains}:

\begin{defn}
	Recall that an \emph{affine \( n \)-cell} in \( \R^n \) is the intersection of finitely many affine half spaces in \( \R^n \) whose closure is compact.  An \emph{affine \( k \)-cell} in \( \R^n \) is an affine \( k \)-cell in a \( k \)-dimensional affine subspace of \( \R^n \).   It can be open or closed.  
\end{defn}

\begin{prop} \label{cells} 
	If \( \s \) is an oriented affine \( k \)-cell in \( \R^n \), there is a unique differential \( k \)-chain \( \widetilde{\s} \in \hB_k^1 \) such that \[ \cint_{\widetilde{\s}} \o = \int_\s \o \]  for all \( \o\in \B_k^1 \), where the right hand integral is the Riemann integral.
\end{prop}

\begin{proof}
	The result follows from  Theorem \ref{thm:opensets} applied to the $k$-dimensional subspace of \( \R^n \) containing \( \s \).
\end{proof}

\begin{defn}
	A \emph{polyhedral\footnote{An equivalent definition uses simplices instead of cells. Every simplex is a cell and every cell can be subdivided into finitely many simplices} \(k\)-chain in \( U \)} is a finite sum \( \sum_{i=1}^s a_i \widetilde{Q_i} \) where \( a_i \in \R \) and \( \widetilde{Q_i} \in \hB_k^1 \) represents an oriented affine \( k \)-cell \( Q_i \) in \( U \). 
\end{defn}

\subsubsection{Polyhedral chains are dense in differential chains}\label{sub:poly} 

Let \( Q(V^k) \)  denotes the parallelepiped, possibly degenerate, determined by the vectors in list \( V^k = (v_1, \dots, v_k) \).

\begin{lem}\label{pointmass} 
	The sequence \( \{2^{ki} Q(2^{-i}V^k)\}_{i \ge 1} \) is Cauchy in the \( B^1 \) norm. 
\end{lem}

\begin{proof} 
	First observe that the mass \( M(2^{ki} Q(2^{-i}V^k)) = M(Q(V^k)) \) for all \( i \ge 0 \). Let \( Q_i = Q(2^{-i}V^k) \). We use \( 1 \)-difference cells to estimate the difference between the \( i \) and \( i+\ell \) terms \[ \|2^{ki}Q_i - 2^{k(i+\ell)}Q_{i + \ell}\|_{B^1}. \]

	Notice that \( Q_i \) can be subdivided into \( 2^{k\ell} \) cubes \( Q_i = \sum R_{i,j} \) each a homothetic replica of \( Q_{i+ \ell} \). When we also consider the multiplicities we can form pairs \[ 2^{ki}Q_i - 2^{k(i+\ell)}Q_{i + \ell} = \sum_{j=1}^{2^{k(i+\ell)}} R_{i,j} - Q_{i+\ell}. \] Therefore \[ \|2^{ki}Q_i - 2^{k(i+\ell)}Q_{i + \ell}\|_{B^1} \le \sum_{j=1}^{2^{k(i+\ell)}} \|R_{i,j} - Q_{i+\ell}\|_{B^1}. \] The distance of translation between each pair is bounded by the diameter of \( Q_i \) which is bounded by a constant \( C \) times \( 2^{-i} \). By Lemma \ref{lem:trnsl} the right hand side is bounded by \[ C 2^{k(i+\ell)} 2^{-i} M( Q_{i+\ell}) \le C 2^{-i}. \] This tends to \( 0 \) as \( i \to \i \). 
\end{proof}

Fix a simple \( k \)-element \( (p;\a) \), and \( \a \in \L_k \). Let \( Q_k \) be a \( k \)-cube centered at \( p \), with side length \( 2^{-k} \), and supported in the \( k \)-direction of \( \a \). We assume \( Q_k \) is oriented to match the orientation of \( \a \). 

\begin{lem}\label{lem:limitpoint} 
	\( (p;\a) = \lim_{k \to \i} 2^{nk}\widetilde{Q}_k \) in the \( B^1 \) norm.
\end{lem}

\begin{proof} 
	The sequence \( \{2^{nk}\widetilde{Q}\} \) is Cauchy by Lemma \ref{pointmass}. Let \( J = 2^{nk}\widetilde{Q} \) denote its limit in \( \hB_k^1 \). Let \( dx_I \) be the unit \( k \)-form in the direction of the \( k \)-direction of \( \a \). The result follows since \( \cint_J f(x) dx_I = \lim_{k \to \i} \cint_{2^{nk}\widetilde{Q}} f(x) dx_I = f(p) \) and \( \cint_{(p;\a)} f(x) dx_I = f(p) \) for all functions \( f \in \B_0^1 \). It follows that \( J = (p;\a) \) as elements of the space of currents \( (\B_k^1)' \). Since both \( J \) and \( (p;\a) \) are elements of \( \hB_k^1 \), which is naturally included in \( (\B_k^1)' \), we conclude that \( J = (p;\a) \) as chains.
\end{proof}

\begin{cor}\label{cor:cellsmore} 
	Polyhedral \( k \)-chains are dense in the Banach space of differential \( k \)-chains \( \hB_k^r \) for all \( r \ge 1 \). 
\end{cor}

\begin{remark}
	Lemma \( \ref{lem:limitpoint} \) does not rely on any particular shape, or open set to approximate \( (p;\a) \). It certainly does not need to be a cube, nor do we have to use a sequence of homothetic replicas. We may use a sequence of chains whose supports tend to \( p \) and whose \( k \)-vectors tend to \( \a \).
\end{remark}

\begin{example}\label{cantor} 
	We show how to represent the middle third Cantor set using a sequence of polyhedral \( 1 \)-chains. Let \( E_1 \) be the chain representing the oriented interval \( (0,1) \). Let \( C_1 \) represent \( (1/3, 2/3) \), and let \( E_2 = E_1 + (- C_1) \). We have replaced the word ``erase'' with the algebraically precise ``subtract'' (see p.\pageref{cor:opr}). Recursively define \( E_n \) by subtracting the middle third of \( E_{n-1} \). The mass of \( E_n \) is \( (\frac{2}{3})^n \). It is not hard to show that the sequence \( \{ (\frac{3}{2})^n E_n \} \) forms a Cauchy sequence in \( \hB_1^1 \). Therefore its limit is a differential \( 1 \)-chain \( \G \) in \( \hB_1^1 \). (See \S\ref{sub:boundary} where the boundary operator is applied to \( \G \).)
\end{example}

\subsection{The inductive limit topology} \label{sub:the_top}

Since the \( B^r \) norms are decreasing, the identity map from \( (\A_k, \|\cdot\|_{B^r})\) to \( \hB_k^s \) is continuous and linear whenever \( r\le s \), and therefore extends to the completion, \( \hB_k^r \).  The resulting continuous linear maps \( u_k^{r,s}:\hB_k^r \to \hB_k^s \) are well-defined linking maps.   The next lemma is straightforward:
\renewcommand{\labelenumi}{(\alph{enumi})}
\begin{lem}\label{lem:varinj} 
	The maps \( u_k^{r,s}: \hB_k^r \to \hB_k^s \) satisfy 
	\vspace{-.2in}
	\begin{enumerate} 
		\item \( u_k^{r,r} = I \); 
		\item \( u_k^{s,t} \circ u_k^{r,s} = u_k^{r,t} \) for all \( r \le s \le t \);
		\item  The image of \( u_k^{r,s} \) is dense   in \( \hB_k^s \).
	\end{enumerate}
\end{lem}

In Corollary \ref{cor:inj} below we will show that each 
 \( u_k^{r,s}: \hB_k^r \to \hB_k^s\) is an injection.  This will follow from knowing that forms of class \( B^r \) are approximated by forms of class \( B^{r+1} \) in the \( B^r \) norm. This is a standard result for \( C^r \) forms using convolution product (see \cite{whitney} Chapter V, Theorem 13A), but we include details to keep track of the indices: 

For \( c > 0 \), let \( \overline{\kappa_c}:[0,\i)] \to \R \) be a nonnegative smooth function that is monotone decreasing, constant on some interval \( [0, t_0] \) and equals \( 0 \) for \( t \ge c.\) Let \( \k_{c}: \R^n \to \R \) be given by \( \k_{c}(v) = \overline{\kappa_c}(\|v\|)  \). Let \( dV = dx_1 \wedge \cdots \wedge dx_n \) the unit \( n \)-form. We multiply \( \kappa_c \) by a constant so that \( \int_{\R^n} \kappa_c (v) dv = 1 \). Let \( a_r \) be the volume of an \( n \)-ball of radius \( r \) in \( \R^n \). For \( \o \in  \B_k^r \) a differential $k$-form and \( (p;\a) \) a simple $k$-element, let \[ \o_c(p;\a) := \int_{\R^n} \k_c(v)\o(p+v;\a)dv. \] 

\begin{thm}\label{thm:injection} 
	If \( \o \in  \B_k^r \)   and \( c > 0 \), then   \(    \o_c  \) is \( \i \)-smooth  and \vspace{-.2in}
	\begin{enumerate} 
		\item\( L_u( \o_c) = (L_u \o)_c \) for all \( u \in \R^n\); 
		\item \( \|\o_c\|_{B^r} \le \|\o\|_{B^r} \); 
		\item  \( \o_c \in \B_k^{r+1}  \); 
		\item \( \o_c(J) \to \o(J) \) as \( c \to 0 \) for all \( J \in \hB_k^r \).
	\end{enumerate} 
\end{thm}

\begin{proof} (a):
  \begin{align*} 	
  L_u (\o_c)(p;\a) &= \lim_{\e \to 0} \o_c((p+\e u;\a) - (p;\a))/\e = 
	\lim_{\e \to 0}  \int_{\R^n} \k_c(v)\o(p+ \e u +v;\a) - (p + v;\a)/\e dv \\& =
	 \int_{\R^n} \k_c(v)(L_u \o(p+v;\a))dv   = (L_u\o)_c (p;\a). 
	\end{align*}

	(b): Since \( \int_{\R^n} \k_c dv = 1 \), we know 
	\begin{align*}
		\frac{|\o_c(\D_{\s^j} (p;\a))|}{\|\s\|\|\a\|} &= \left| \int_{\R^n} \k_c(v)\o(T_v\D_{\s^j} (p;\a)/\|\s\|\|\a\|)dv \right|  
	 \le \sup_{v \ne 0}\frac{|\o(T_v \D_{\s^j} (p;\a))|}{\|\s\|\|\a\|}  
	\\& = \sup_{v \ne 0} \frac{|\o( \D_{\s^j} (p+v;\a))|}{\|\s\|\|\a\|}\le \sup_{ q \in \R^n} \frac{|\o(\D_{\s^j} (q;\a))|}{\|\D_{\s^j} (q;\a)\|_{B^r}} 
	 \le \sup_{0 \ne A \in \A_k} \frac{|\o(A)|}{\|A\|_{B^r}}  = \|\o\|_{B^r} 
	\end{align*} 
	for all \( 0 \le j \le r \). Therefore, \( |\o_c|_{B^j} \le \|\o\|_{B^r} \), and hence \( \|\o_c\|_{B^r} \le \|\o\|_{B^r} \).

 	(c): Suppose \( 0 \le j \le r \). Let \( \eta = L_{\s^j} \o \). Then \( |\eta|_{B^0} \le \|\o\|_{B^r} < \i \). By (a) we know \( \eta_c = L_{\s^j} (\o_c) \). Now \[ \eta_c(T_u (p;\a)) = \int_{\R^n} \k_c(v) \eta(T_{v+u} (p;\a))dv = \int_{\R^n} \k_c(v-u) \eta (T_v (p;\a)) dv, \] and \[ \eta_c (p;\a) = \int_{\R^n} \k_c(v) \eta(T_v (p;\a))dv. \] Since the integrand vanishes for \( v \) outside  ball of radius \( c \) about the origin, we have
	\begin{align*} 
		|\eta_c (p+u;\a)  - \eta_c (p;\a))| &= \left|\int_{\R^n} (\k_c(v -u) - \k_c(v)) \eta (p+v;\a)dv \right| \\
		&\le \int_{\R^n} |\k_c(v -u) - \k_c(v)| |\eta(p+v;\a)|dv \\
		&\le a_c |\k_c|_{Lip}\|u\||\eta|_{B^0}\|\a\|. 
	\end{align*} 
	Therefore, \( |\eta_c|_{Lip} \le a_c |\k|_{Lip} \|\o\|_{B^r} \). Using (b), Lemma \ref{lem:oncemore} and Theorem \ref{lem:seminorm}, we deduce \( \|\o_c\|_{B^{r+1}} < \i \).

 	(d): First of all \[(\o_c - \o) (p;\a)   =  \int_{\R^n} \k_c(v)\|v\| \o((p+v;\a) - (p;\a))/\|v\|dv =  \int_{\R^n} \k_c(v)\|v\| L_v\o(p;\a)dv + r_c   \] where \( r_c \to 0 \).  
   Let \( A \) be a Dirac chain. Then \[ |(\o_c - \o) (A)| \le c \left|  \int_{\R^n} \k_c(v)  L_v\o(A)dv \right| \le c \left|  \sup_v\{ \k_c(v)\}  \int_{\R^n}   L_v\o(A)dv \right| \le c \sup_v \{ \k_c(v)\} a_c |\o(A)| \le  c' \|\o\|_{B^r}\|A\|_{B^r}\] where \( c' \to 0 \) as \( c \to 0 \).  
 For   \( J \in \hB_k^r \) choose \( A_i \to J \) in \( \hB_k^r \) and use \( \|A_i\|_{B^r}\to \|J\|_{B^r} \).
	\end{proof}

\begin{cor}\label{thm:flat} 
	If \( J \in \hB_k^r \) is a differential \( k \)-chain and \( 0 \le r \), then \[ \|J\|_{B^r} = \sup_{0 \ne \o \in \B_k^{r+1}}  \frac{|\o(J)|}{\|\o\|_{B^r}}.
\] 
\end{cor}

\begin{proof}    

	According to Theorem \ref{lem:seminorm}, we have \( \|J\|_{B^r} =  \sup_{0 \ne \o \in \B_k^{r}}  \frac{|\o(J)|}{\|\o\|_{B^r}}.  \) Suppose \( \o \in \B_k^r  \). Let \( \e > 0 \). By Theorem \ref{thm:injection} (b)-(d), there exists \( c > 0 \) such that 
	\begin{align*} 
		\frac{|\o(J)|}{\|\o\|_{B^r}} &\le \frac{|\o_c(J)| + \e}{\|\o\|_{B^r}} \\
		&\le \frac{|\o_c(J)| + \e}{\|\o_c\|_{B^r}} \\&\le \sup_{\eta \in \B_k^{r+1}} \frac{|\eta(J)| + \e}{\|\eta\|_{B^r}}. 
	\end{align*}
	Since this holds for all \( \e > 0 \), we know \( \|J\|_{B^r} =  \sup_{\o \in \B_k^r } \frac{|\o(J)|}{\|\o\|_{B^r}} \le \sup_{\eta \in \B_k^{r+1} } \frac{|\eta(J)|}{\|\eta\|_{B^r}} \).  Equality holds since \( \B^{r+1} \subset \B^r \).
 
\end{proof}

\begin{cor}\label{cor:inj} 
	The  linking maps \( u_k^{r,s}: \hB_k^r \hookrightarrow \hB_k^s \) are injections for each \( r \le s \).
\end{cor}

\begin{proof}  It suffices to prove this for \( s = r+1 \).
	Suppose there exists \( J \in \hB_k^r \) with \( u_k^{r,r+1} J = 0 \).   Then \( \o( u_k^{r,r+1} J) = 0 \) for all \( \o \in \B_k^{r+1} \).   By Corollary \ref{thm:flat},  this implies that \( \|J\|_{B^r} = 0 \), and hence \( J = 0 \). 
\end{proof} 
 
 Let \( \pB_k  := \varinjlim\,\hB_k^r \), the inductive limit as \( r \to \i \), and  \( u_k^r:\hB_k^r\to  \pB_k \)  the associated inclusion mappings into the inductive limit. Each \( u_k^r(\hB_k^r)  \) is a Banach space with norm \( \|u_k^r J\|_r = \|J\|_{B^r} \).  Since the \( B^r \) norms are decreasing, then \( \{u_k^r(\hB_k^r)\} \) forms a nested sequence of increasing Banach subspaces of  \(  \pB_k \).  We endow \( \pB_k \) with the  the inductive limit topology \( \t_k \); it is the finest locally convex topology such that the maps \( u_k^r:\hB_k^r \to  \pB_k \) are continuous. If \( F \) is locally convex, a linear map \( T:( \pB_k,\t_k) \to F \) is continuous if and only if each \( T \circ u_k^r: \hB_k^r \to F \) is continuous\footnote{One can also introduce H\"older conditions into the classes of differential chains and forms as follows: For \( 0 < \b \le 1 \), replace \( |\D_{\s^j} (p;\a)|_{B^j} = \|u_j\| \cdots \|u_1\|\|\a\| \) in definition \eqref{def:norms} with \( |\D_{\s^j} (p;\a)|_{B^{j-1+\b}} = \|u_j\|^\b \cdots \|u_1\|\|\a\| \) where \( \s^j = = u_j \circ \dots \circ u_1 \). The resulting spaces of chains \( \hB_k^{r-1+\b} \) can permit fine tuning of the class of a chain. The inductive limit of the resulting spaces is the same \(  \pB_k \) as before, but we can now define the \emph{extrinsic dimension} of a chain \( J \in  \pB_k \) as \( \dim_E(J) := \inf \{ k+ j-1 + \b: J \in \hB_k^{j-1+\b} \} \). For example, \( \dim_E(\widetilde{\s}_k) = k \) where \( \s_k \) is an affine \( k \)-cell since \( \widetilde{\s}_k \in \hB_k^1 \) (set \( j=0, \b = 1 \)), while \( \dim_E(\widetilde{S}) = \ln(3)/\ln(2) \) where \( S \) is the Sierpinski triangle. The dual spaces are differential forms of class \( B^{r-1+\b} \), i.e., the forms are of class \( B^{r-1} \) and the \( (r-1) \) order directional derivatives satisfy a \( \b \) H\"older condition. The author's earliest work on this theory focused more on H\"older conditions, but she has largely set this aside in recent years.}. However, the inductive limit is not strict.  For example, $k$-dimensional dipoles are limits of chains in \( u_0(\hB_k^0) \) in \( (\pB_k, \t_k) \), but dipoles are not found in the Banach space \( u_0(\hB_k^0) \) itself.

\begin{defn}
	If \( J\in \pB_k = \varinjlim\,\hB_k^r \), the \emph{\textbf{type}} of \( J \) is the infimum over all \( r > 0 \) such that there exists  \( J_r\in\hB_k^r \) with \( u_k^r J_r = J \).
 \emph{Type} is well-defined since the Banach spaces \( u_k^r( \hB_k^r)\) generate  \(  \pB_k \) (see \cite{AG}, p. 136). \end{defn}

 \begin{thm}\label{thm:frechet} 
	The locally convex space \(  ( \pB_k,\t_k) \) is Hausdorff and separable. 
\end{thm}
\begin{proof}  
Suppose   \(  J \in \pB_k \) is nonzero.   Since the class of \( J \) is well-defined, there exists \( J_r \in \hB_k^r \) with \( u_k^r J_r = J \) where \( u_k^r:\hB_k^r \to  \pB_k  \) is the canonical injection.  Thus \( J_r \ne 0 \).  By Theorem \ref{thm:injection} there exists \( \o \in B^\i \) with \( \o(J_r) \ne 0 \). It follows that  \( \o(J) \ne 0 \) since \( J \) and \( J_r \) are approximated by the same Cauchy sequence of Dirac chains.  Therefore \(  \pB_k \) is separated.  This implies that \(   \pB_k  \) is Hausdorff (see \cite{AG} p. 59).    
 
	For the second part, we know that Dirac chains are dense in each \( \hB_k^r \). We can approximate a given Dirac chain \( D \) by a sequence of Dirac chains \( D_i=\sum_j (p_{j,i}; \a_{j,i}) \) whose points \( p_{j,i} \) and \( k \)-vectors \( \a_{j,i} \) have rational coordinates.  
\end{proof}
\begin{defn}
	Let \( \B_k=\B_k^\i \) denote the Fr\'echet space of bounded \( C^\i \)-smooth differential \( k \)-forms, with bounds on the \( j \)-th order directional derivatives for each \( 0\leq j<\i \).  The defining seminorms can be taken to be the \( B^r \) norms.
\end{defn}

 \begin{prop}\label{prop:isomo} The topological dual to \( \pB_k \) is isomorphic to \( \B_k \). 
 \end{prop} 

\begin{proof}   
	Since \( \pB_k = \varinjlim \hB_k^r \), we know \( \pB_k' = \varprojlim (\hB_k^r)' \cong \varprojlim \B_k^r = \B_k \). It is well known that the dual of the inductive limit topology is the projective limit of the duals, and vice versa. (See \cite{kothe}, \S 22.7, for example.).   
\end{proof}  

Let \( F \) be the Fr\'echet topology on the space of differential $k$-forms \( {\cal B}_k \) and \(  \b({\cal B}_k, \,^\prime {\cal B}_k) \) the strong topology of the dual pair \( ({\cal B}_k, \,^\prime {\cal B}_k) \).   
 The following Lemma was first established in \cite{topological}: 
\begin{lem}\label{lem:inclu}  
	\( ({\cal B}_k, \b({\cal B}_k, \,^\prime {\cal B}_k)) = ({\cal B}_k, F) \). 
  \end{lem}  
	The next theorem follows directly.  
	\begin{thm}\label{thm:isot} 
		The space of differential \( k \)-cochains \( (\pB_k)' \) with the strong (polar) topology is topologically isomorphic to the Fr\'echet space of differential \( k \)-forms \( \B_k \).
	\end{thm}

An equivalent way to construct the inductive limit \( ( \pB_k, \t_k) \) is via direct sums and quotients.  Endow \( \oplus_{r = 0}^\i \hB_k^r \) with the direct sum topology. For \( 0 \le k \le n \), let \( H_k = H_k(\R^n) \) be the linear span of the subset \( \{ J_r - u_k^{r,s} J_r: J_r \in \hB_k^r, s \ge r \ge 0 \} \subset \oplus_r \hB_k^r \). Then \( \oplus_r \hB_k^r /H_k \), endowed with the quotient topology, is topologically isomorphic to \( ( \pB_k,\t_k) \) (\cite{kothe}, \S 19.2, p.219). Since \( ( \pB_k,\t_k) \) is Hausdorff (Theorem \ref{thm:frechet}), it follows that \( H_k \) is closed. (\cite{kothe}, (4) p.216). Hence the projection \( \pi_k: \oplus_r \hB_k^r \to \oplus_r \hB_k^r /H_k \) is a continuous linear map (\cite{kothe}, \S 10.7 (3),(4)). Since canonical inclusion \( u_k^r: \hB_k^r \to \oplus_r \hB_k^r \) is   continuous, the inclusion \( \nu_k^r = \pi_k \circ u_k^r: \hB_k^r \to ( \pB_k,\t_k) \) is continuous.
 
Suppose \( T: \oplus_{k=0}^n \oplus_{r\ge 0} \hB_k^r \to \oplus_{k=0}^n \oplus_{r\ge 0} \hB_k^r \) is  continuous and bigraded with \( T(\hB_k^r) \subset \hB_\ell^s \), and \( T(H_k) \subset H_\ell \). By the universal property of quotients, \( T \) factors through a graded linear map \( \hat{T}: \oplus_k \oplus_r \hB_k^r/H_k \to \oplus_k \oplus_r \hB_k^r/H_k \) with \( T = \hat{T} \circ \pi \) where  \( \pi \) is the canonical projection onto the quotient. That is,  \( \hat{T}[H_k + J] = [H_k + T(J)] \).  Let   \( \pB = \pB^\i := \oplus_{k=0}^n \pB_k \), endowed with the direct sum topology.

\begin{thm}\label{thm:continuousoperators}  
	If \( T: \oplus_{k=0}^n \oplus_{r=0}^\i \hB_k^r \to \oplus_{k=0}^n \oplus_{r=0}^\i \hB_k^r \) is a  continuous bigraded linear map  with \( T(\hB_k^r) \subset \hB_\ell^s \), and \( T(H_k) \subset H_\ell \), then \( T \) factors through a continuous graded linear map  \( \hat{T}: \pB \to \pB \) with \( T = \pi \circ \hat{T} \).
\end{thm}

\begin{proof}
	 By the universal property of quotients, \( T \) factors through a graded linear map \( \hat{T}: \oplus_k \oplus_r \hB_k^r/H_k \to \oplus_k \oplus_r \hB_k^r/H_k \) with \( T = \hat{T} \circ \pi \) where  \( \pi \) is the canonical projection onto the quotient. That is,  \( \hat{T}[H_k + J] = [H_k + T(J)] \).  Let   \( \pB = \pB^\i := \oplus_{k=0}^n \pB_k \), endowed with the direct sum topology.	Furthermore, \( \hat{T}: \pB \to \pB \) is continuous and graded (see \cite{kothe} (\S 19.1 (7)).   
	
\end{proof}

\section{Primitive operators} \label{ssub:creation_and_annihilation_operators}

In this section we define a few simple operators on the space \( \pB \) of differential chains.  The ``primitive'' operators are all defined initially with respect to a vector \( v \in \R^n \) (or perhaps more formally but no more rigorously, a constant vector field determined by \( v \).)  In \S\ref{pro:IIAV}, we extend them to bona fide vector fields, and in \S\ref{sub:integral_theorems_in_open_sets}, we extend them again to differential chains on an open set \( U \).

\subsection{Extrusion} \label{ssub:extrusion}

\begin{defn}
	Let \( E_v:  \A_k \to \A_{k+1} \) be the  linear map, called  \emph{extrusion through \( v \)},  defined by its action on simple \( k \)-elements, \[ E_v(p; \a):=(p; v \wedge \a). \] 
\end{defn}

\begin{lem}\label{lem:IIA}
The map \( E_v: \A_k \to \A_{k+1} \) extends to a continuous linear map \( E_v: \hB_k^r \to \hB_{k+1}^r \) satisfying
\begin{enumerate}
	\item \( \|E_v(J)\|_{B^r} \le \|v\|\|J\|_{B^r} \) for all \( J \in \hB_k^r \), and
	\item  \( E_v(H_k) \subset H_{k+1} \).
\end{enumerate}  
\end{lem}

\begin{proof} 
(a):	Let \( T(p;\a) =  E_u (p;\a) \) where \( u \) is a unit vector in \( \R^n \).  . Then
	\begin{align*} 
		\|T(\D_{\s^j} (p;\a))\|_{B^r} &= \| \D_{\s^j} T(p;\a) \|_{B^r} \\
		&= \|\D_{\s^j} (p; u \wedge \a) \|_{B^r} \le \|\s\|\|u\|\|(p;\a)\|_{r-j} \le \|u\|\|\s\|\|\a\|.
	\end{align*}
	By Corollary \ref{cor:opr} it follows that \( \|E_v(A)\|_{B^r} \le \|v\|\|A\|_{B^r} \). The extension to \( \hB_k^r \) is immediate.  
Part (b) follows   since \( E_v u_k^{r,s} = u_k^{r,s} E_v \).   That is,
\begin{align*}
 E_v(H_k)   =  \{ E_v J_r - E_v u_k^{r,s} J_r: J_r \in \hB_k^r, s \ge r \ge 0 \} = \{ E_v J_r -   u_k^{r,s}E_v J_r: J_r \in \hB_k^r, s \ge r \ge 0 \}  \subset H_{k+1}. 
\end{align*} 

 \end{proof}

If \( S \) and \( T \) are operators, let \( [S,T] := ST - TS \) and \( \{S, T\} = ST + TS \).  The following relations are immediate:

\begin{prop}\label{pro:EE} 
	Let \( v,w \in \R^n \). Then 
	\begin{enumerate} 
		\item \( E_v^2= 0 \); 
		\item \( \{E_v, E_w\} = 0 \). 
	\end{enumerate} 
\end{prop}

Let \( E_\a := E_{v_k} \circ \cdots \circ E_{v_1} \) where \( \a = v_1 \wedge \cdots \wedge v_k \) is simple.  Then    \( E_\a(p;\b)=(p;\a\wedge\b) \).  The inequality in Lemma \ref{lem:IIA} readily extends to \(  \|E_\a(J) \|_{B^r} \le \|\a\| \|J\|_{B^r}  \). 

	Since  \( E_v(H_k) \subset H_{k+1} \) we know from Theorem \ref{thm:continuousoperators} that
	 \( E_v:  \oplus_{k=0}^n \oplus_{r=0}^\i \hB_k^r \to \oplus_{k=0}^n \oplus_{r=0}^\i \hB_k^r \) factors through a continuous graded linear map \( \hat{E}_v: \pB  \to \pB  \) with \( \hat{E}_v(\hB_k^r) \subset \hB_{k+1}^r \). To keep the notation simple, we will suppress the quotient and write \( E_v \) instead of  \( \hat{E}_v \).
	
\begin{figure}[ht] 
	\centering 
	\includegraphics[height=1.5in]{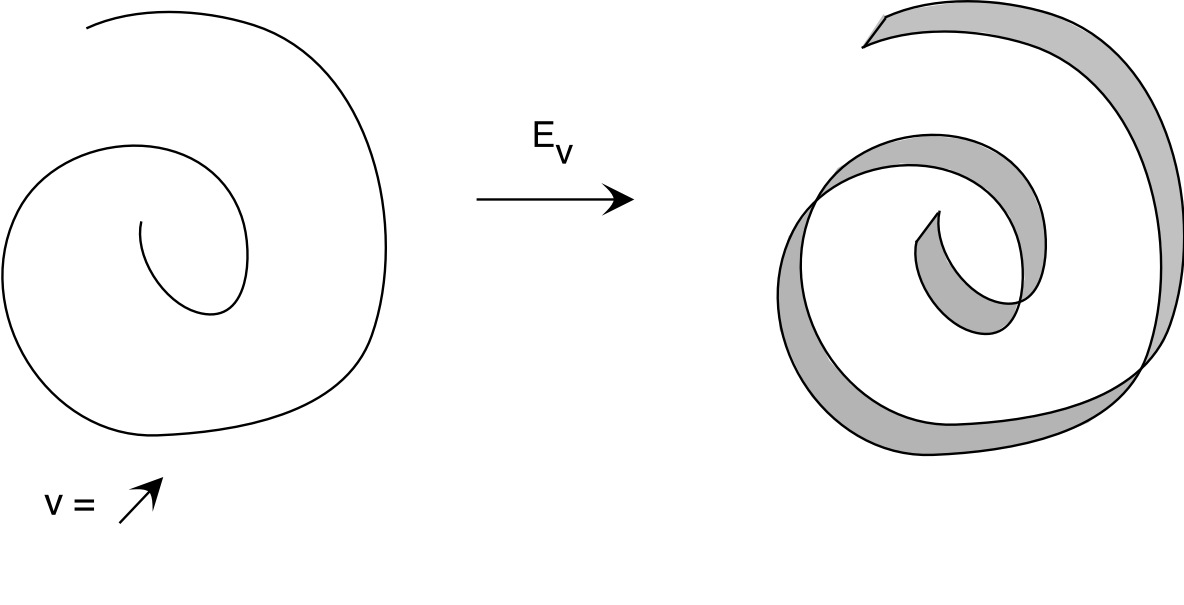}
	\caption{Extrusion of a curve through the vector field \( \frac{\p}{\p x} + \frac{\p}{\p y} \)} 
	\label{fig:Extrusion} 
\end{figure}

The dual operator on \( B \) is \emph{interior product} \( i_v: \B_{k+1}^r \to \B_k^r \) since \( i_v \) satisfies \( \o E_v (p;\a) =  i_v \o (p;\a) \). 

\begin{cor}\label{cor:ext}  
	Let \( v \in \R^n \). Then \( E_v \in {\cal L}(\pB) \) and \( i_v \in {\cal L}( \B) \) are continuous graded operators satisfying
	\begin{equation}
		\cint_{E_v J} \o = \cint_J i_v \o
	\end{equation} 
	for all matching pairs, i.e.,  \( J \in \hB_k^r \) and \( \o \in \B_{k+1}^r \), or \( J \in \pB_k \) and \( \o \in \B_{k+1} \).  	 Furthermore,  the bigraded multilinear maps  \( E \in {\cal L}(\R^n \times \pB, \pB) \), given by \( (v, J) \mapsto  E_v(J) \),  and       \( \phi_E \in {\cal L} (\R^n \times \pB \times \B, \R)  \),   given by  \( (v, J, \o) \mapsto \cint_{E_v  J} \o  \), are both separately continuous.
\end{cor}

\begin{proof}   The first part follows from the isomorphism theorems \ref{lem:seminorm}  and \ref{thm:isot}.  Separate continuity follows from Lemma \ref{lem:IIA}.  
\end{proof}

\begin{remarks}\mbox{}
	\begin{itemize}
		\item 
	Only a few operators we work with are closed on Dirac chains.  When this occurs, as it does for extrusion \( E_v \),  the corresponding integral relation given by duality is still nontrivial, because continuity must be shown, and because of the isomorphism theorems \ref{lem:seminorm} and \ref{thm:isot}.  
  \item  This result will be significantly extended in \S\ref{sub:extrusion} when we replace \( v \in \R^n \) with a vector field in a Fr\'echet space of smooth vector fields
 	\end{itemize}
	
\end{remarks}

\subsection{Retraction}\label{ssub:retraction}
 
Given a \( k \)-vector \( \a = v_1 \wedge \cdots \wedge v_k\), let \( \hat{\a}_i = v_1 \wedge \cdots \hat{v_i} \cdots \wedge v_k \).  For \( k=1 \), say \( \hat{\a}=1 \).  
\begin{defn}
Let \( E_v^\dagger:  \A_k \to \A_{k-1} \) be the  linear map, called  \emph{retraction through \( v \)}, defined by its action on simple \( k \)-elements, \[ E_v^\dagger (p; \a) :=\sum_{i=1}^k (-1)^{i+1} \<v,v_i\> (p; \hat{\a}_i). \]
\end{defn}

\begin{lem}\label{lem:IIB}
The map \( E_v^\dagger \) extends to a continuous linear map \( E_v^\dagger: \hB_k^r \to \hB_{k-1}^r \) satisfying
\begin{enumerate}
	\item \( \|E_v^\dagger(J)\|_{B^r} \le k\|v\|\|J\|_{B^r} \) for all \( J \in \hB_k^r \), and
	\item  \( E_v^\dagger(H_k) \subset H_{k-1} \).
\end{enumerate}  
\end{lem}
\begin{proof} 
(a):  	Let \( T \in {\cal L}(\A) \) be the operator determined by \( T(p;\a) = E_v^\dagger  (p;\a)/\|v\| \) for all \( k \)-elements \( (p;\a) \). By Corollary \ref{cor:opr} it suffices to show that \( \|T(\D_{\s^j}(p;\a))\|_{B^r} \le C \|\s\|\|\a\| \) for all \( 0 \le j \le r \). This follows since \[ \|T(\D_{\s^j}(p;\a))\|_{B^r} = \|E_v^\dagger(\D_{\s^j}(p;\a)/\|v\|\|_{B^r} \le \sum_{i=1}^k |(-1)^{i+1} \<v,v_i\> \D_{\s^j} (p; \hat{\a}_i) |_{B^j} \le k \|\s\|\|\a\| \] using \( \|\<u,v_i\>\hat{\a}_i\| \le \|\a\| \). The inequality follows.
Part (b) follows since \( u_k^{r-1,r} E_v^\dagger =   E_v^\dagger u_k^{r-1,r} \).
\end{proof}

The extension  to continuous graded linear maps \(E_v^\dagger: \pB \to \pB \)  is similar to our construction for extrusion \( E_v \) and we omit the details.

\begin{figure}[ht] 
	\centering
	\includegraphics[height=2in]{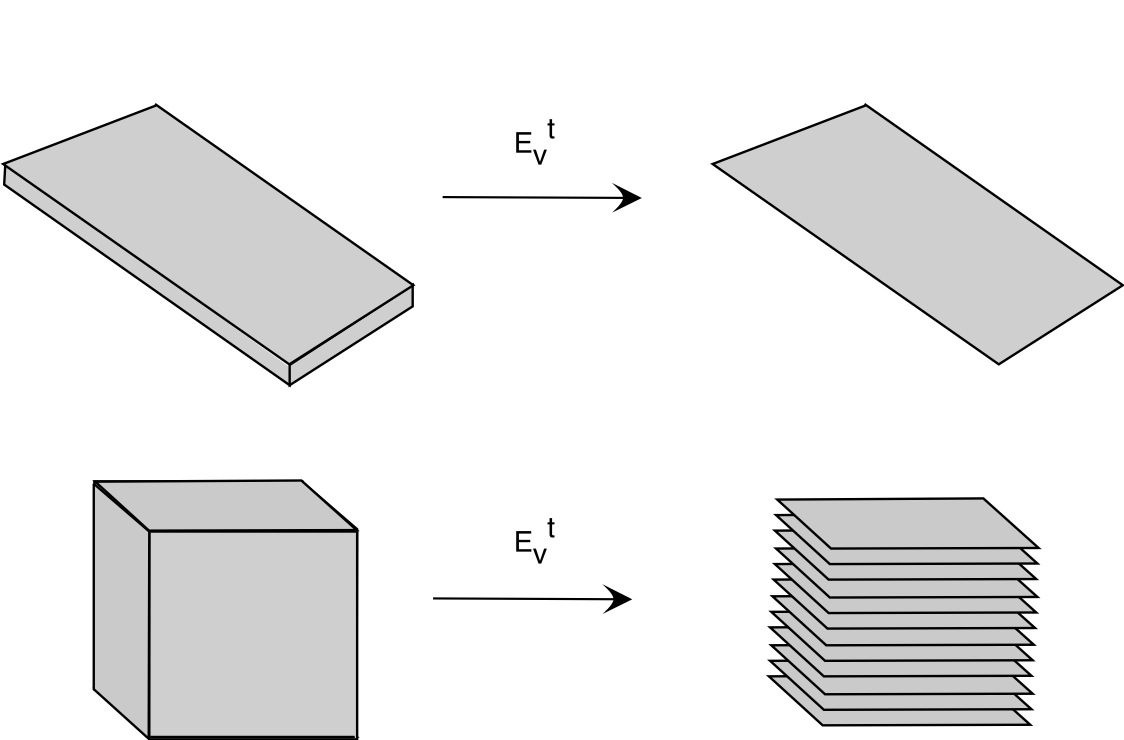} 
	\caption{Retraction of an extruded rectangle and a unit cube} 
	\label{fig:Retraction} 
\end{figure}

\begin{prop}\label{pro:asgg}Using the inner product in \S\ref{ssub:differential_chains}
	\begin{enumerate} 
  		\item \( \< E_v (p;\a), (p;\b) \> = \< (p;\a), E_v^\dagger(p;\b) \> \) for all \( k \)-vectors \( \b \) and \( (k-1) \)-vectors \( \a \);
		\item \( E_v^\dagger \circ E_v^\dagger = 0 \) for all \( v \in \R^n \);
		\item \( [E_v^\dagger, E_w^\dagger] = 0 \) for all \( v,w \in \R^n \); 
		\item \label{car} \( \{E_v^\dagger, E_w\} = \{E_v, E_w^\dagger\} = \<v(p),w(p)\> I \);
		\item \( (E_v + E_v^\dagger)^2 = \<v,v\> I \);
		\item\label{last} \( \<E_v (p,\a),(p,\b)\> = \<(p,\a), (p, E_v^\dagger \b) \> \).
\end{enumerate}
\end{prop}

\begin{proof} 
	Proof of (\ref{last}): It suffices to prove this for \( \a = e_I \) and \( \b = e_k \wedge e_I \) where \( e_I = e_1 \wedge \cdots \wedge e_{k-1} \). Then \( \< v \wedge \a,\b\> = \< v \wedge e_I, e_k \wedge e_I\> = \<v,e_k\> \). On the other hand, \( \< (p;\a), E_v^\dagger(p;\b) \> = \<v,e_k\> + \sum_{i=1}^k(-1)^{i} \<v,e_i\> \<e_1 \wedge \cdots \wedge e_{k-1}, e_k \wedge e_1 \wedge \cdots \hat{e_i} \cdots \wedge e_{k-1} \> = \<v,e_k\> \).

 	The remaining assertions follows easily from the definitions.
\end{proof} 

Property \ref{car} coincide with the canonical anticommutation relations (C.A.R.) of Fock spaces.

\begin{lem}\label{lem:evdag}
	The dual operator of \( E_v^\dagger \) is \( v^\flat	 \wedge \cdot \).
\end{lem}

\begin{proof} 
	It is enough to show that \( \o E_v^\dagger (p;\a) = v^\flat \wedge \o (p;\a) \) for all \( k \)-elements \( (p;\a) \) and all \( v \in \R^n \). This follows since \[ \o E_v^\dagger (p;\a) = \sum_{i=1}^n (-1)^{i+1} \o(p; \<v,v_i\>\a_i) = v^\flat \wedge w (p;\a). \] 
\end{proof}

The next result is similar to Corollary \ref{cor:ext} and we omit  its proof:

\begin{cor}\label{cor:daggerE} 
	Let \( v \in \R^n \). Then  
	\begin{equation}\label{primitiveretraction} 
		\cint_{E_v^\dagger J} \o = \cint_J v^\flat \wedge \o 
	\end{equation} 
	for all matching pairs, i.e., \( J \in \hB_k^r \) and \( \o \in \B_{k-1}^r \), or \( J \in \pB_k \) and \( \o \in \B_{k-1} \). 
Furthermore,  the bigraded multilinear maps  \( E^\dagger \in {\cal L}(\R^n \times \pB, \pB) \), given by \( (v, J) \mapsto  E_v^\dagger(J) \),  and       \( \phi_{E^\dagger} \in {\cal L} (\R^n \times \pB \times \B, \R)  \),   given by  \( (v, J, \o) \mapsto \cint_{E_v^\dagger  J} \o  \), are both separately continuous.
\end{cor}

\subsection{Prederivative} \label{ssub:prederivative} 
 In this section we introduce the concept of differentiation of chains, without regard to a function.
\begin{lem}\label{lem:pred} 
	Let \( (p;\a) \) be a simple \( k \)-element and \( v \in \R^n \). Then the sequence \( \{(T_{v/i} -I)(p; i\a)\}_{i \ge 1} \) is Cauchy in \( \hB_k^2 \).
\end{lem}  

\begin{proof} 
	For simplicity, assume \( p = 0, k = 0, \a= 1, u = e_1, n = 1 \). Telescoping gives us 
	\begin{equation}\label{equation1} 
		( 2^{-i};1) - (0;1) = \sum_{m=1}^{2^j} (m2^{-i+j}; 1) - ((m-1)2^{-i+j}; 1). 
	\end{equation} 
	Expanding, we have \[ (T_{2^{-i}u} - I)(0;2^i) - (T_{2^{-(i+j)}u} - I)(0;2^i) = ((2^{-i}; 2^i) - (0;2^i )) - ((2^{-(i+j)}; 2^{i+j}) -(0; 2^{i+j} )). \] We apply \eqref{equation1} to the first pair, and consider the second pair as \( 2^j \) copies of \( ((2^{-(i+j)}; 2^i) -(0; 2^i )) \). Then
	\begin{align*} 
		(T_{2^{-i}u} &- I)(0;2^i) - (T_{2^{-(i+j)}u} - I)(0;2^{i+j}) \\
		&= \sum_{m=1}^{2^j} (m2^{-(i+j)u}; 2^i) - ((m-1)2^{-(i+j)u}; 2^i) - ((2^{-(i+j)u}; 2^{i}) -(0; 2^{i})).
	\end{align*}

	Rewriting the right hand side as a sum of \( 2 \)-difference chains we set \( v = 2^{-(i+j)}u \).  Then
	\begin{align}\label{eqn2} 
		(T_{2^{-i}u} - I)(0;2^i) - (T_{2^{-(i+j)}u} - I)(0;2^{i+j}) = \sum_{m=1}^{2^j} \D_{(v, (m-1)v)}(0; 2^i).
	\end{align}

	We then use the triangle inequality to deduce 
	\begin{align*} 
		\| (T_{2^{-i}v} - I)(0;2^i) - (T_{2^{-(i+j)}v} - I)(0;2^{i+j})\|_{B^2}  &= \|((2^{-i}; 2^i) - (0;2^i)) - ((0; 2^{i+j} )- (2^{-(i+j)}; 2^{i+j}))\|_{B^2} \\
		& = \| \sum_{m=1}^{2^j} \D_{(v,(m-1)v)}(0; 2^i) \|_{B^2} \le \sum_{m=1}^{2^j} \|\D_{(v, (m-1)v)}(0; 2^i) \|_{B^2} \\
		& \le \sum_{m=1}^{2^j} |\D_{(v, (m-1)v)}(0; 2^i) |_2  \le 2^{-i}.
	\end{align*}
\end{proof}

Define the differential chain \( P_v (p;\a) := \lim_{i \to \i} (T_{iv} -I)_*(p;\a/i) \) and extend to a linear map of Dirac chains \( P_v: \A_k \to \hB_k^2 \) by linearity.  Let \( P: \R^n \times \A_k \to \pB_k \) be the resulting bilinear map.

\begin{lem}\label{lem:preext}
The  map \( P_v: \A_k \to \hB_k^2 \) extends to a continuous linear map \( P_v: \hB_k^r \to \hB_k^{r+1} \) satisfying
\begin{enumerate}
	\item \( \|P_v(J)\|_{B^{r+1}} \le \|v\|\|J\|_{B^r} \) for all \( J \in \hB_k^r \), \( v\in \R^n \), and  \( r \ge 1 \);
	\item \( P_v(H_k) \subset H_k \).  
\end{enumerate}
\end{lem}

\begin{proof} 
(a): It suffices to   show \( \|P(v, A)\|_{B^{r+1}} \le \|v\| \|A\|_{B^r} \) for all \( A \in \A_k \).  Since \( P_v \) commutes with translation, and using the definition of \( P_v \) as a limit of \( 1 \)-difference chains, 
	\begin{align*} 
		\|P_v(\D_{\s^j} (p;\a))\|_{B^{j+1}} = \lim_{t\to 0} \| \D_{tv}(\D_{\s^j} (p;\a/t))\|_{B^{j+1}} &\le \lim_{t \to 0}| \D_{tv}(\D_{\s^j} (p;\a/t))|_{B^{j+1}} \\
		&= \|v\| \|\s\|\|\a\|.
	\end{align*}

	Let \( A \in \A_k \) be a Dirac chain, \( r \ge 0 \), and \( \e > 0 \). We can write \( A = \sum_{i=1}^m \D_{\s_i^{j(i)}}(p_i;\a_i) \) as in the proof of Lemma \ref{lem:ineq}, with \( \|A\|_{B^r} > \sum_{i=1}^m \|\s_i\|\|\a_i\| - \e \). Then 
	\begin{align*} 
		\|P_v A\|_{B^{r+1}} \le \sum_{i=1}^m \|P_v \D_{\s_i^{j(i)}}(p_i;\a_i)\|_{B^{r+1}} &\le \|v\| \sum_{i=1}^m \|\s_i\|\|\a_i\| \\
		&< \|v\|( \|A\|_{B^r} + \e). 
	\end{align*} Finally, if \( J \in \hB_k^r \), choose \( A_i \to J \) in the \( B^r  \) norm and apply the previous inequality.  
	(b): This follows since \( P_v \circ u_k^{r,s} = u_k^{r,s}\circ P_v \).  
\end{proof}

 \begin{figure}[htbp] 
	\centering 
	\includegraphics[height=2.5in]{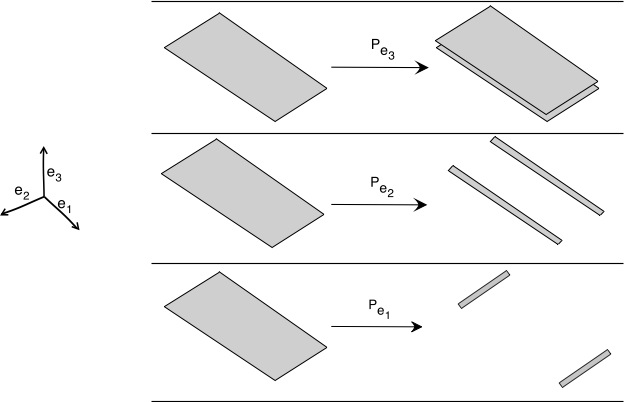}
	\caption{Prederivative of a \( 2 \)-cell in \( \R^3 \) in three different directions}
	\label{fig:Prederivative} 
\end{figure}

\begin{prop}\label{pro:IIC} 
	The map \( P:\R^n \times \pB_k \to \pB_k \) is bilinear and separately continuous.  
\end{prop}                       
     
\begin{proof} 
 	We show that \( P \) determines a bilinear map on Dirac chains. It is linear in the second variable by definition. Additivity in the first variable reduces to showing \[ \lim_{t \to 0} (p + t(v_1 + v_2); \a/t) - (p + tv_1; \a/t) - (p + tv_2; \a/t) + (p; \a/t) = 0 \] in \( \pB_k \). This follows from the definition of the \( B^2 \) norm of a \( 2 \)-difference chain: \(\|(p + t(v_1 + v_2); \a/t) - (p + tv_1; \a/t) - (p + tv_2; \a/t) + (p; \a/t)\|_{B^2} \le |(p + t(v_1 + v_2); \a/t) - (p + tv_1; \a/t) - (p + tv_2; \a/t) + (p; \a/t)|_2 \le t\|v_1\|\|v_2\|\|\a\| \). Homogeneity is immediate since \( \l(p;\a) = (p;\l \a) \).   
 \end{proof}

In particular, Theorem \ref{pro:IIC} shows that \( P_v \in {\cal L}(\pB) \) is a continuous bigraded operator.

\begin{prop}[Commutation relations]\label{ppp}
	\( [P_v,P_w] = [E_v, P_w] = [E_v^\dagger, P_w]= 0 \) for all \( v, w\in \R^n \).
\end{prop}

\begin{lem}
	The dual operator to \( P_v \) is \emph{directional\footnote{when the derivative is with respect to a bona fide vector field, we will call it the Lie derivative.} derivative} \(L_v \o(p;\a) := \lim_{h \to 0} \frac{\o(p+tv; \a) - \o(p;\a)}{t} \).
\end{lem}

\begin{proof} 
	We have \( \o P_v (p;\a) = \o \lim_{t \to 0} \frac{(p+tv;\a) - (p;\a)}{t} = \lim_{t \to 0} \frac{\o(p+tv;\a) - \o(p;\a)}{t} = L_v \o(p;\a). \)
\end{proof}
Note that when we exchange the limit, we may do so because \( \o \) is a continuous operator on differential chains.  Differentiability here is translated into continuity.  

\begin{cor}\label{thm:pre}  
	Let \( v \in \R^n \). Then \( P_v \in {\cal L}(\pB) \) and \( L_v \in {\cal L} (\B) \) are continuous bigraded operators satisfying
	\begin{equation}\label{primitiveprederivative} 
		\cint_{P_v J} \o = \cint_J L_v \o
	\end{equation} 
	for all matching pairs, i.e., for all \( J \in \hB_k^{r-1} \) and \( \o \in \B_k^r \) or \( J \in \pB_k \) and \( \o \in \B_k \). 
	The bigraded multilinear map \( P \in {\cal L}( \R^n \times \pB, \pB ) \), given by \( (v, J) \mapsto P_vJ \) is separately continuous.
\end{cor} 

\subsection{Chainlets} % (fold)
\label{sub:higher_order_dirac_chains}
Let \( \s = \s^s = \{u_1, \dots, u_s\} \) be a list of vectors in \( \R^n \) and \( P_{\s^s} (p;\a) = P_{u_1} \circ \cdots \circ P_{u_s} (p;\a) \) where \( (p;\a) \) is a simple \( k \)-element. We say \( p \) is the \emph{support}\footnote{a formal treatment of support will be given in \S\ref{sub:support}.} of \( P_{\s^s} (p;\a) \). Then the differential chain \( P_{\s^s} (p;\a) \in \hB_k^{s+1} \) is called a \emph{simple \( k \)-element of order} \( s \). For \( k = 0 \), these are \emph{singular distributions} represented geometrically as differential chains, e.g., Dirac deltas \( (s=0) \), dipoles \( (s = 1) \), and quadrupoles \( (s=2 )\). For \( s \ge 0 \), let \( \A_k^s(p) \) be the subspace of \( \hB_k^{s+1}(\R^n) \) generated by \( k \)-elements of order \( s \) supported at \( p \). Elements of \( \A_k^s(p) \) are called \emph{Dirac \( k \)-chains of order \( s \) supported at \( p \)}.  Let \( S(\R^n) \) be the symmetric algebra and \( S^s(\R^n) \) its \( s \)-th symmetric power.

\begin{prop}\label{pro:symmet} 
	Let \( p \in \R^n \).  The  vector space \( \A_k^s(p) \)  is isomorphic to  \( S^s(\R^n) \otimes \L_k(\R^n) \).
\end{prop}

\begin{proof} 
	Let \( k = 0 \). The linear map \( P_{u_s} \circ \cdots \circ P_{u_1}(p;\a)\mapsto u_s \circ \cdots \circ u_1 \otimes \a \) is an isomorphism preserving the symmetry of both sides since \( P_{u_1} \circ P_{u_2} = P_{u_2} \circ P_{u_1} \) and \( P_{u_1} \circ (P_{u_2} \circ P_{u_3}) = (P_{u_1} \circ P_{u_2}) \circ P_{u_3} \).
\end{proof} 

\begin{remark}\label{notation}
We may therefore use the notation \( (p;u \otimes \a) = P_u(p;\a) \), and, more generally, \( (p; \s \otimes \a) = (p;(u_s \circ \cdots \circ u_1 )\otimes \a) = P_{u_s} \circ \cdots \circ P_{u_1}(p;\a)  = P_\s (p;\a) \) where \( \s = u_s \circ \cdots \circ u_1 \). For example, \( P_u(p;\s \otimes \a) = (p; (u \circ \s) \otimes \a) \). We will be using both notations, preferring whichever one fits the situation best.  While \( P_\s \) emphasizes the operator viewpoint, the tensor product notation reveals the algebra of the Koszul complex \( \oplus_{s=0}^\i \oplus_{k=0}^n S^s \otimes \L_k \) more clearly. For example, the scalar \( t \) in \( t(\s \otimes \a) = t \s \otimes \a = \s \otimes t \a\) ``floats,'' while this is not immediately clear using the operator notation.	
\end{remark} 

 Define \( \A_\bullet^\bullet(p):= \bigoplus_{k=0}^n \bigoplus_{s\geq 0} \A_k^s(p) \).  
 
 \begin{defn}[Unital associative algebra at \( p \)]\label{def:symmetric_algebra}
	Let \( \cdot \) be the product on higher order Dirac chains \( \A_\bullet^\bullet(p)   \) given by \( (p;\s \otimes \a ) \cdot (p; \t \otimes \b) := (p; \s \circ \t \otimes \a \wedge \b) \).	
\end{defn}

\begin{remarks}\mbox{}
	\begin{itemize}
	 	\item For each \( p \in \R^n \) \( \A_\bullet^\bullet(p) \) is a unital, associative algebra with unit \( (p;1) \). 
	 	\item This product is not continuous and does not extend to the topological vector space \( \pB \).
	 	\item \( E_v^\dagger \) is a graded derivation on the algebra \( \A_\bullet^\bullet(p) \). Neither \( E_v \) nor \( P_v \) are derivations.  
	\end{itemize}
\end{remarks}

\begin{lem}\label{thm:bdthm}
	There exists a unique bigraded operator \( \p: \A_\bullet^\bullet(p) \to \A_\bullet^\bullet(p) \) such that 
\begin{enumerate} 
	\item[(a)] \( \p(p;v) = P_v(p;1) \);
	\item[(b)] \( \p(A \cdot B) = (\p A) \cdot B + (-1)^k A \cdot \p B \) for \( A, B \in \A_\bullet^\bullet(p) \) where \( \dim(A) = k \);
	\item[(c)] \( \p \circ \p= 0 \).
\end{enumerate}  

The operator \( \p \) additionally satisfies
	\begin{enumerate}
				\item[(d)] \( \p (p;\a) = \sum_i P_{e_i}E_{e_i}^\dagger (p;\a) \), where \( \{e_i\} \) forms an orthonormal basis of \( \R^n \).
 	\end{enumerate}  
\end{lem}

\begin{proof} 
The first part is virtually identical to the proof of existence of exterior derivative  and we omit it.  Part (d) is easy to prove using the basis \( \{e_I\}  \) for \( \L_k(\R^n) \) and properties (a)-(c).
\end{proof}

% \subsection{Higher order differential chains}
Let \( \A_k^s = \A_k^s(\R^n)  \) be the free space \( \R\<\cup_p  \A_k^s(p) \> \).  In particular, an element of \( \A_k^s  \) is a formal sum \( \sum_{i=1}^m (p_i; \s^i \otimes \a_i) \) where \( p_i \in \R^n \) (see \ref{notation}).  It follows from Proposition \ref{pro:symmet} that \( \A_k^s(\R^n) \) is isomorphic to the free space \( (S^s \otimes \L_k)\<\R^n\> \).  Now \( \A_k^s(\R^n) \) is naturally included in \( \hB_k^{s+1}(\R^n) \), and is thus endowed with the subspace topology\footnote{This inner product \( \<\cdot,\cdot\>_\otimes \) on \( \A_k^s(p) \) does not extend to a continuous inner product on \( \A_k^s(\R^n) \), but it can be useful for computations on Dirac chains of arbitrary order and dimension.}.   Let \( \A_\bullet^\bullet(\R^n) :=  \oplus_{k= 0}^n \oplus_{s \ge 0} \A_k^s(\R^n) \).   Elements of   \( \overline{\A_k^s(\R^n)} \), where the completion is taken in either the \( B^r \) norms, \( r \ge s+1 \),  or the inductive limit topology of \( \B_k(\R^n) \), are called $k$-\emph{chainlets\footnote{Earlier definitions of the word chainlets assumed \( s = 0 \) }} \emph{of (dipole) order} \( s \).    We will see next that \( \{\overline{\A_k^s(\R^n)}\} \) forms a bigraded chain complex.
   
\subsection{Boundary and the generalized Stokes' Theorem}\label{sub:boundary}   
   
\begin{thm}\label{thm:bod}
	There exists a unique continuous bigraded operator \( \p: \hB_k^r \to \hB_{k-1}^{r+1} \)  such that 
	\begin{enumerate}  
		\item[(a)] \( \p\sum(p_i;v_i) = P_{v_i}(p_i;1) \);
		\item[(b)] \( \p((p;\a) \cdot (p;\b)) = (\p (p;\a)) \cdot (p;\b) + (-1)^k (p;\a) \cdot \p (p;\b) \) for \( (p;\a), (p;\b) \in \A(p)  \) with \( \dim(p;\a) = k \);
		\item[(c)] \( \p \circ \p= 0 \).
	\end{enumerate}
	This operator \( \p \) additionally satisfies
	\begin{enumerate}
		\item[(d)] \( \p(\A_k^s) \subset \A_{k-1}^{s+1} \mbox{ for all }  k \ge 1  \);
		\item[(e)] \(  \p(\A_0^s) = \{0\} \);
		\item[(f)]\( \p = \sum_i P_{e_i}E_{e_i}^\dagger \).
	\end{enumerate}
\end{thm}

\begin{proof}   
	Parts (a)-(f) follow directly from Lemma \ref{thm:bdthm}, while continuity is a consequence of (f) and continuity of \( P_v \) and \( E_v^\dagger \) (see Theorems \ref{lem:IIB} and \ref{pro:IIC}).  
\end{proof}

\begin{cor}\label{cor:bdcont} 
	The linear map \( \p: \hB_k^r \to \hB_{k-1}^{r+1} \) is continuous with \( \|\p A\|_{B^{r+1}} \le kn \|A\|_{B^r} \). It therefore extends to a continuous operator \( \p : \pB \to \pB \) and restricts to completions of the chainlet complex \( \overline{\A_\bullet^\bullet} \). 
\end{cor}

\begin{thm}\label{thm:Gpush}\mbox{\quad} 
	\begin{enumerate} 
		\item \( \{\p, E_v \} = P_v \) (Cartan's magic formula for differential chains);
		\item \( [E_v^\dagger, \p] = 0 \); 
		\item \( [P_v, \p] = 0 \).
	\end{enumerate} 
\end{thm} 

\begin{thm}[Generalized Stokes' Theorem]\label{cor:stokes} 
	\[ \cint_{\p J} \o = \cint_J d \o \] for all \( J \in \hB_{k+1}^{r-1} \) and \( \o \in \B_k^r \), or \( J \in \hB_{k+1} \) and \( \o \in \B_k \), where \( r \ge 1 \) and \( 0 \le k \le n-1 \).
\end{thm}  

\begin{proof} 
	By Lemma \ref{lem:evdag} and Corollary \ref{thm:pre} \( \o \p (J) = \o \sum_{i=1}^n P_{e_i}E_{e_i}^\dagger (J) = \sum_{i=1}^n L_{e_i} \o E_{e_i}^\dagger(J) = \sum_{i=1}^n d e_i \wedge L_{e_i}\o(J) \).  Since \(  d \o :=\sum_{i=1}^n d e_i \wedge L_{e_i} \o \) is a classical definition of \( d \o \), the result follows.    
\end{proof}

\begin{lem}\label{lem:boundarystok}
Let \( W \) be a bounded open subset of \( \R^n \) with smooth boundary.  Then \( \p \widetilde{W} = \widetilde{\p W} \) (where the second \( \p \) refers to the boundary operator for cells in  topology. )
\end{lem}

\begin{proof} 
Write \( W = \cup Q_i \) as its Whitney decomposition into non-overlapping cubes.  Then \( \lim_{m \to \i} \sum_{i=1}^m \widetilde{Q_i} \to \widetilde{W} \) in the  \( B^1 \) norm.  This follows since each \( \widetilde{Q_i} \in B^1 \) by Lemma \ref{lem:cuberep} and the mass norm of the Cauchy differences tends to zero.  Hence \( \lim_{m \to \i} \p \sum_{i=1}^m \widetilde{Q_i} \to  \p \widetilde{W} \) as \( m \to \o \).  But \( \widetilde{\p W} = \lim_{m \to \i}   \sum_{i=1}^m \p \widetilde{Q_i}  \) since the region between \( \p W \) and \( \p (\cup_{i=1}^m   Q_i) \) has area tending to zero as \( m \to 0 \).    The result follows.  
\end{proof}

\begin{examples}\mbox{}
	\begin{itemize}
		\item \textbf{Stokes' theorem restricts to the classical result on bounded, open sets}. If \( W \subset \R^n \) is bounded and open with smooth boundary, then \( \int_W d\o =  \cint_{\tilde{W}} d\o = \cint_{\p \tilde{W}} \o  = \int_{\p W} \o \) for all \( \o \in \B_1^2 \) since \( \p \tilde{W} = \widetilde{\p W} \) by Lemma \ref{lem:boundarystok}. 
		\item \textbf{Algebraic boundary of the Cantor set}. In \S\ref{sub:poly} we found a differential \( 1 \)-chain \( \G = \lim_{n\to \i} (3/2)^n\widetilde{E_n}\) representing the middle third Cantor set where \( E_n = \cup [p_{n_i}, q_{n_i}] \) is the set obtained after removing middle thirds at the \( n \)-th stage. Then \( \p \G = \sum_{n=1}^{\i} (p_{n_i};(3/2)^n) - (q_{n_i};(3/2)^n). \)   Furthermore, \( \cint_{\p \G} \o = \cint_{\G} d \o \) for all \( \o \in C^{1 + Lip} \).  For example, \[ \cint_{\p \G} x = \cint_{\G} dx = \lim_{n \to \i} \cint_{(3/2)^n\widetilde{E_n}} dx = 1. \]    
	\end{itemize}
\end{examples}   

\subsubsection{Directional boundary and directional exterior derivative}\label{ssub:subsubsection_boundarydir}

Let \( v \in V \). Define the \emph{directional boundary} on differential chains by \(  \p_v: \hB_k^r \to \hB_{k-1}^{r+1} \) by \( \p_v := P_v E_v^\dagger \). See Figure \ref{fig:Partialboundary}.

\begin{figure}[ht]
	\centering 
	\includegraphics[height=2.2in]{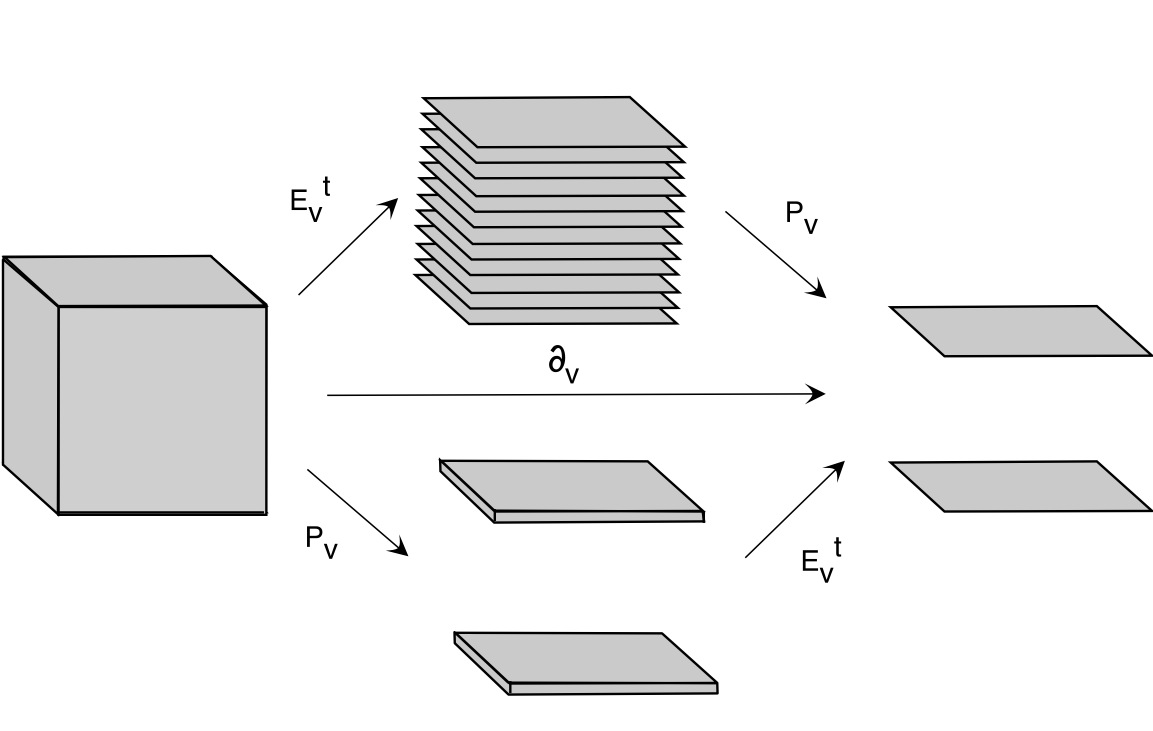}
	\caption{Directional boundary \( \p_v \)} 
	\label{fig:Partialboundary} 
\end{figure} 

For example, \( \p_{e_1}(p; 1\otimes e_1) = (p; e_1 \otimes 1) \), \( \p_{e_1}(p;1 \otimes e_1 \wedge e_2) = (p;e_1 \otimes e_2) \) and \( \p_{e_2}(p; 1 \otimes e_1 \wedge e_2) = (p;-e_2 \otimes e_1) \).  

For \( v\in V \), define the \emph{directional exterior derivative} on differential forms \( d_v \o = v^\flat \wedge L_v \o \). 

\begin{thm}\label{thm:directder} 
	\( d_v: \B_{k-1}^{r+1} \to \B_k^r \) satisfies \( d_v \o = \o \p_v \). 
\end{thm}

\begin{proof} 
	Since \( v^\flat \wedge L_v \o = L_v \o E_v^\dagger \), we have \( d_v \o = v^\flat \wedge L_v \o = \o P_v E_v^\dagger \), 
\end{proof}

\subsection{Perpendicular complement and classical integral theorems of calculus}\label{sec:clifford_algebra_(_)} 

\subsubsection{Clifford algebra and perpendicular complement}\label{sub:perp} 

Consider the subalgebra\footnote{The author thanks Patrick Barrow for pointing out that this operator algebra is isomorphic to the Clifford algebra.} \( {\cal C\ell}  \subset {\cal L}(\pB) \) of linear operators on \( \pB \) generated by \( \{C_v: v \in V \} \) where \( C_v = E_v + E_v^\dagger \).  According to  the bilinearity of \( E \) and \( E^\dagger \) (see Corollaries \ref{cor:ext} and \ref{cor:daggerE}), we know that \( {\cal C\ell} \) is isomorphic to the Clifford algebra.  However, the isomorphism depends on the inner product.   The algebra product\footnote{Some authors call \(   u \wedge v + \<u,v\> \) the   ``geometric product''.  This is found naturally as \( (E_u + E_u^\dagger)\circ (E_v + E_v^\dagger)(0;1) =  u \wedge v + \<u,v\>  \).} is composition of operators: \( C_u \cdot C_v :=  (E_u + E_u^\dagger)\circ (E_v + E_v^\dagger) \).  It is worth understanding what \( C_v \) does to a simple \( k \)-element \( (p;\a) \):  For simplicity, assume \( \a = e_I \) and \( v = e_i \), taken from an orthonormal basis \( \{e_i\} \) of \( \R^n \). If \( e_i \) is in the \( k \)-direction of  \( e_I \), then \( C_{e_i} \) ``divides it'' out of \( e_I \), reducing its dimension to \( k-1 \).  If \( e_i \) is not in the \( k \)-direction of \( e_I \), then \( C_{e_i} \) ``wedges it'' to \( e_I \), increasing its dimension to \( k+1 \).      

\begin{defn}
	Let \( \perp: \A_k \to \A_{n-k} \) be the operator on Dirac chains given by  \(  = C_{e_n} \circ \cdots \circ C_{e_1} = \Pi_{i=1}^n E_{e_i} + E_{e_i}^\dagger \).    Then \( \perp \) extends to a continuous linear map on \(\pB \) since \( E_v \) and \( E_v^\dagger \) are continuous. We call \( \perp \) \emph{perpendicular complement}.  Perpendicular complement does depend on the inner product, but not on the choice of orthonormal basis.
\end{defn}

It is not too hard to see that \( \perp \) behaves as we expect.  That is, 

\begin{prop}\label{prop:perpagain} 
	If \( \a \) is a \( k \)-vector, then \( \perp(p;\a) = (p; \b) \), where \( \b \) is the \( (n-k) \)-vector satisfying \( \a \wedge \b = (-1)^k \|\a\|^2 e_1 \wedge \cdots \wedge e_n \).  Furthermore, the \( k \)-direction of \( \a \) is orthogonal to the \( (n-k) \)-direction of \( \b \). Thus, \( \perp \circ \perp = (-1)^{k(n-k)} I\).
\end{prop}

\begin{figure}[ht] 
	\centering 
	\includegraphics[height=1in]{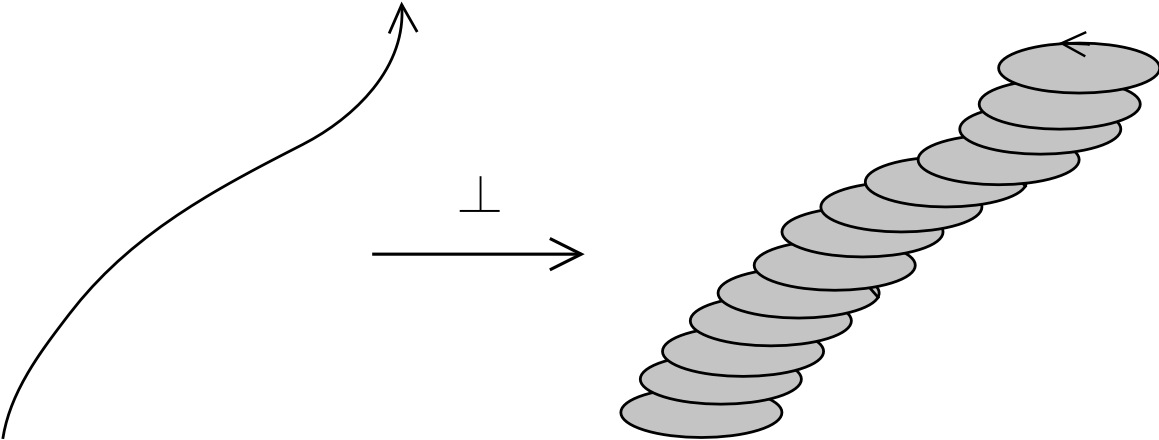}
	\caption{Perpendicular complement \( \perp \)} 
	\label{fig:Perp} 
\end{figure}

It follows that \( \star \o := \o \perp \) where \( \star \) is the classical Hodge \cite{thehodge} star operator on forms.

\begin{thm}[Star theorem]\label{thm:perpcont}
	The linear map \( 	\perp: \A_k  \to \A_{n-k}  \) determined by \( (p;\a)  \mapsto (p; \perp \a), \) for simple $k$-elements \( (p;\a) \), satisfies \[ \|\perp A\|_{B^r} = \|A\|_{B^r} \] for all \( A \in \A_k \). It therefore extends to a continuous linear map \( \perp: \hB_k^r \to \hB_{n-k}^r \) for each \( 0 \le k \le n, r \ge 0 \), and to a continuous graded operator \(\perp \in {\cal L}(\pB) \) satisfying \[ \cint_{\perp J} \o = \cint_J \star \o \] for all \( J \in \hB_k^r \) and \( \o \in \B_{n-k}^r \) or \( J \in \pB_k \) and \( \o \in \B_{n-k} \). 
\end{thm}

\begin{proof} 
	We know \( \|\perp J\|_{B^r} = \|J\|_{B^r} \) since \( \|J\|_{B^r} = \sup\frac{|\o(J)|}{\|\o\|_{B^r}} \) and \( \|\o\|_{B^r} = \|\star \o\|_{B^r} \). A direct proof using difference chains is straightforward and can be found in \cite{hodge}.

	The integral relation holds on Dirac chains since \( \star \o = \o \perp \) on Dirac chains and by continuity. Continuity of \(\perp: \pB \to \pB \) follows.
\end{proof}

\subsubsection{Geometric coboundary, Laplace, and Dirac operators}\label{sub:geometric_coboundary_laplace_and_dirac_operators}

Define the \emph{coboundary} operator \( \lozenge := \perp \p \perp \). Since \( \p \) decreases dimension, \( \lozenge \) increases dimension. Its dual operator is the codifferential \( \d \) where \( \d = \star d \star \).

\begin{figure}[htbp]
	\centering
	\includegraphics[height=1.5in]{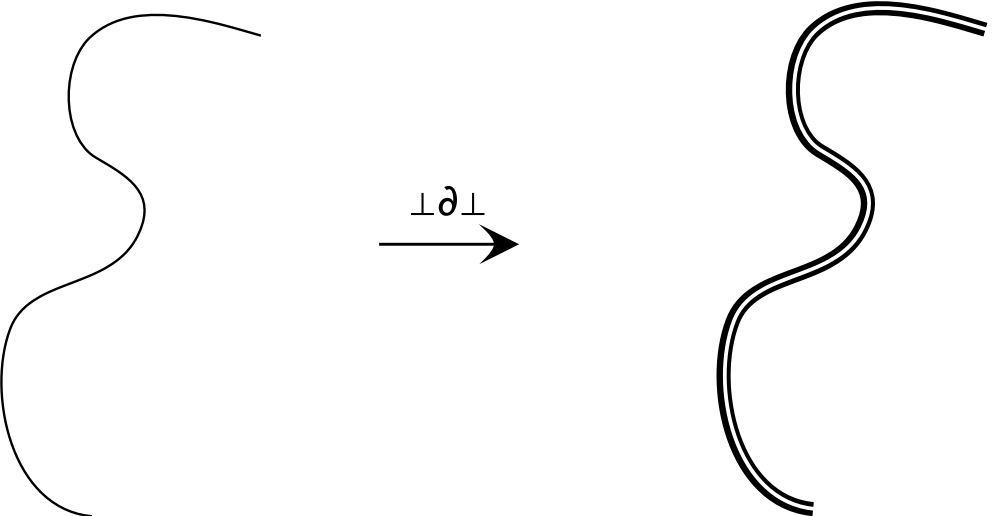}
	\caption{Coboundary operator}
	\label{fig:coboundary}
\end{figure}

Define the \emph{geometric\footnote{Yes, this is an abuse of the word ``geometric.''  The author is amenable to suggestions about what to call this operator.} Laplace} operator \( \square := \lozenge \p + \p \lozenge \). This operator preserves dimension and restricts to an operator on \( \pB_k \). The dual operator is the classical Laplace operator \( \D = d \d + \d d \) on differential forms. The \emph{geometric Dirac} operator \( \p + \lozenge \) dualizes to the Dirac operator \( d + \d \) on forms. (see \cite{hodge}). 

\begin{figure}[ht]
	\centering 
	\includegraphics[height=1.5in]{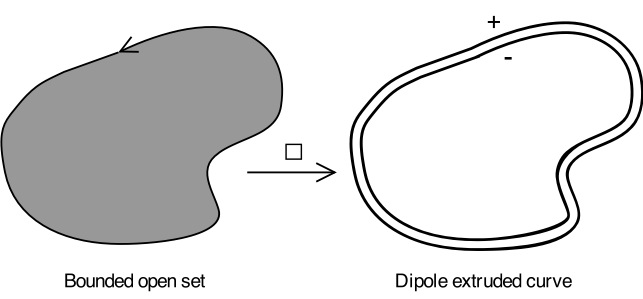} 
	\caption{Geometric Laplace operator \( \square \)  of an open set in \( \R^2 \)} 
	\label{fig:Laplace} 
\end{figure}
    
\begin{cor}[General Gauss-Green Divergence Theorem]\label{cor:div}
	The integral relation \[ \cint_J d \star \o =  \cint_{\perp \p J} \o \] holds for all  \( J \in \hB_{k}^r \)   and \( \o \in \B_{n-k+1}^{r+1} \),  or \( J \in \pB_k \) and \( \o \in \B_{n-k+1} \).
\end{cor}

\begin{example}
 	If \( \tilde{S} \) represents a surface with boundary in \( \R^3 \) and \(\o \) is a smooth \( 2 \)-form defined on \( S \), our general divergence theorem tells us \( \cint_{\tilde{S}} d \star \o =  \cint_{\perp \p \tilde{S}} \o \).  This result equates the net divergence of  \( \o  \)  within \( S \) with the net flux of \( \o \) (thought of as a \( 2 \)-vector field with respect to the inner product) across the boundary of \( S \).
\end{example} 

\begin{figure}[ht]
	\centering
    \includegraphics[height=2.0in]{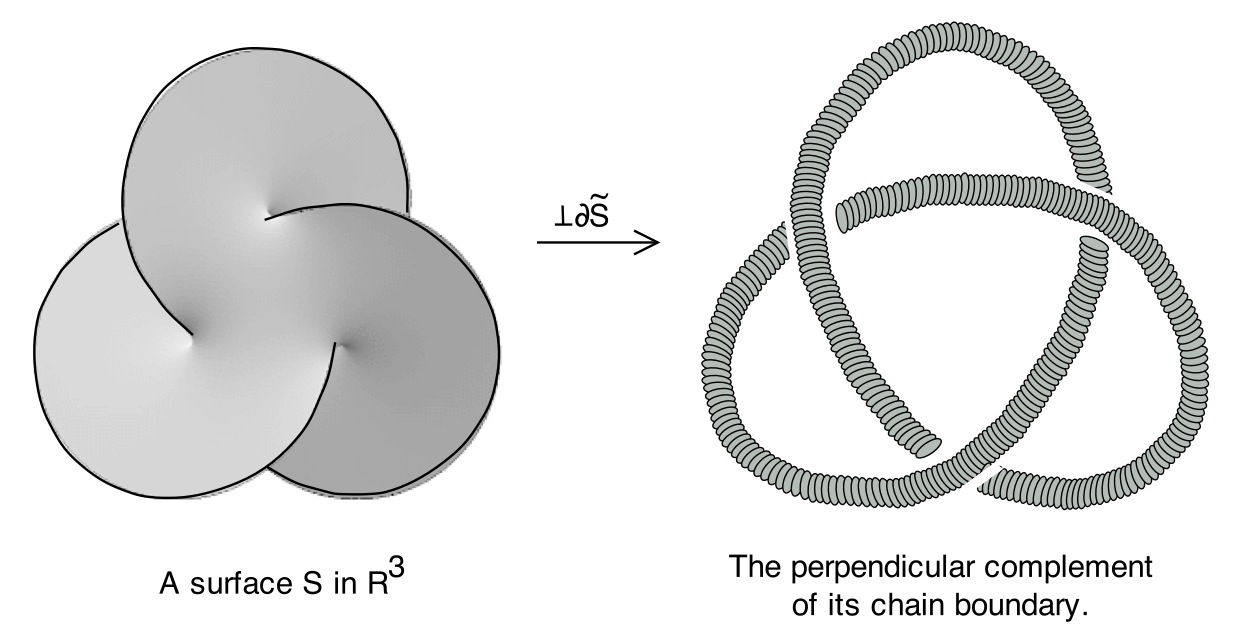}
    \caption{A domain of the divergence theorem for surfaces   in \( \R^3 \)}
    \label{fig:Divergence3D}
\end{figure}
          
\begin{cor}[General Kelvin-Stokes' Curl Theorem]\label{cor:curl}
	The integral relation \[ \cint_{\p J} \o = \cint_{\perp J} \star d \o  \] holds for all \( J \in \hB_{k+1}^r \)  and \( \o \in \B_k^{r+1} \),  or \( J \in \pB_{k+1} \) and \( \o \in \B_k \).
\end{cor}

\begin{cor}[Higher order divergence theorem]\label{cor:lapl} 
	The integral relation \[ \cint_{\square^s J} \o = \cint_J \D^s \o \] holds for all \( s \ge 0 \), \( J \in \hB_{k}^r \) and \( \o \in \B_k^{r+2} \), or \( J \in \pB_k \) and \( \o \in \B_k \).
\end{cor} 

Figure \ref{fig:Laplace} shows this Corollary \ref{cor:lapl} has deep geometric meaning for \( s = 1 \). Whereas the divergence theorem equates net flux of a \( k \)-vector field across a boundary with net interior divergence, Corollary \ref{cor:lapl} equates net flux of the ``orthogonal differential'' of a \( k \)-vector field across a boundary with net interior ``second order divergence''. In this sense Corollary \ref{cor:lapl} is a ``higher order divergence theorem''. As far as we know, it has no classical precedent.  

The following diagram depicts the underpinnings of the \emph{exterior differential complex}.  The ``boundary'' operators are boundary \( \p \) (boundary, predual to exterior derivative),  \( \lozenge \) (geometric coboundary, predual to \( *d* \)), extrusion \( E_v \), and retraction \( E_v^\dagger \).  The commutation relations are given in Propositions \ref{pro:EE}, \ref{pro:asgg},  \ref{ppp}, \ref{thm:Gpush}, and \ref{prop:perpagain}.

\begin{equation}\label{diagramK1} 
	\xymatrixcolsep{4pc}
	\xymatrixrowsep{4pc}
	\xymatrix{
		\vdots \ar@<0.05in>@{->}[d]^{E_v^\dagger} \ar@{<-}[d]_(.52){E_v} \ar@{->}[rd]^<<<<<<<{\p}   & \vdots \ar@<0.05in>@{->}[d]^{E_v^\dagger} \ar@{<-}[d]_(.52){E_v}\ar@{->}[rd]^<<<<<<<{\p}   & \vdots \ar@<0.05in>@{->}[d]^{E_v^\dagger} \ar@{<-}[d]_(.52){E_v}\ar@{->}[rd]^<<<<<<<{\p}   &  \vdots\cdots \\ 
		\hB_2^0 \ar@<0.05in>@{->}[d]^{E_v^\dagger} \ar@{<-}[d]_(.52){E_v}\ar@{->}[rd]^<<<<<{\p}   \ar@<0.05in>@{->}[r]^{P_v} \ar@{->}[ru]^(.68){\lozenge}|\hole  & \hB_2^1 \ar@<0.05in>@{->}[d]^{E_v^\dagger} \ar@{<-}[d]_(.52){E_v}\ar@{->}[rd]^<<<<<{\p}   \ar@<0.05in>@{->}[r]^{P_v} \ar@{->}[ru]^(.68){\lozenge}|\hole & \hB_2^2  \ar@<0.05in>@{->}[d]^{E_v^\dagger} \ar@{<-}[d]_(.52){E_v} \ar@{->}[ru]^(.68){\lozenge}|\hole 	\ar@{->}[rd]^<<<<<{\p}   \ar@<0.05in>@{->}[r]^{P_v}   & \cdots \\
		\hB_1^0 \ar@<0.05in>@{->}[d]^{E_v^\dagger} \ar@{<-}[d]_(.52){E_v}\ar@{->}[rd]^<<<<<{\p}   \ar@<0.05in>@{->}[r]^{P_v}\ar@{->}[ru]^(.70){\lozenge}|\hole & \hB_1^1 \ar@<0.05in>@{->}[d]^{E_v^\dagger} \ar@{<-}[d]_(.52){E_v}\ar@{->}[rd]^<<<<<{\p}   \ar@<0.05in>@{->}[r]^{P_v}\ar@{->}[ru]^(.70){\lozenge}|\hole & \hB_1^2 \ar@<0.05in>@{->}[d]^{E_v^\dagger} \ar@{<-}[d]_(.52){E_v}\ar@{->}[rd]^<<<<<{\p}   \ar@<0.05in>@{->}[r]^{P_v}\ar@{->}[ru]^(.70){\lozenge}|\hole & \cdots \\ 
		\hB_0^0 \ar@<0.05in>@{->}[r]^{P_v}\ar@{->}[ru]^(.70){\lozenge}|\hole & \hB_0^1 \ar@<0.05in>@{->}[r]^{P_v}\ar@{->}[ru]^(.70){\lozenge}|\hole & \hB_0^2 \ar@<0.05in>@{->}[r]^{P_v}\ar@{->}[ru]^(.70){\lozenge}|\hole &\cdots
	} 
\end{equation}  
 
 \section{Multiplication by a function, vector fields, and partitions of unity}\label{sub:multiplication_by_a_function}  
 
\subsection{Multiplication by a function and change of density}\label{sub:the_operator}

We prove that the topological vector space of differential chains \( \pB(\R^n) \) is a graded module over the ring of functions (\( 0 \)-forms) \( \B_0(\R^n) \).   

\begin{lem}\label{lem:X} 
	The Banach space \( \B_0^r(\R^n) \) is a ring with unity via pointwise multiplication of functions. Specifically, if \( f, g \in \B_0^r(\R^n) \), then \( \|f \cdot g\|_{B^r} \le nr \|f\|_{B^r} \|g\|_{B^r} \). 
\end{lem}

\begin{proof} 
	The unit function \( u(x) = 1 \) is an element of \( \B_0^r(\R^n) \) for \( r \ge 0 \) since \( \|u\|_{B^r} = 1 \). The proof \( \|f \cdot g\|_{B^r} \le r \|f\|_{B^r} \|g\|_{B^r} \) for \( n = 1 \) is a straightforward application of the product rule and we omit it.  The general result follows by taking coordinates.  
\end{proof} 

Given \( f:\R^n \to \R \),  let \( m_f: \A_k \to \A_k \) be the linear map of Dirac chains determined by \( m_f (p;\a) :=f(p) (p; \a) \) for all \( k \)-elements \( (p;\a) \).

\begin{lem}\label{lem:mflem}
	The map \( m_f: \A_k \to \A_k \) extends to a continuous linear map \( m_f: \hB_k^r \to \hB_k^r \) satisfying
	\begin{enumerate}
		\item \( \|m_f(J)\|_{B^r} \le nr\|f\|_{B^r}\|J\|_{B^r} \) for all \( J \in \hB_k^r \), and
		\item  \( m_f(H_k) \subset H_k \).
	\end{enumerate}
\end{lem}

\begin{proof} 
(a): We first prove this for nonzero \( A \in \A_k \). By Lemma \ref{lem:X}    \( \frac{|\cint_{m_f A} \o|}{\|\o\|_{B^r}} \le \frac{\|f \cdot \o\|_{B^r} \|A\|_{B^r}}{\|\o\|_{B^r}} \le nr\|f\|_{B^r}\|A\|_{B^r}. \)  The inequality follows, since \( \|m_f A\|_{B^r} = \sup \frac{|\cint_{m_f A} \o|}{\|\o\|_{B^r}} \).  We can therefore extend \( m \) to \( \hB_k^r \) via completion, and the operator \( m \) will still satisfy the inequality.
(b): This is immediate from the definitions.
\end{proof}

\begin{prop}\label{lem:mf} 
	Let \( f , g \in \B_0^r(\R^n) \) and \( v\in \R^n \). Then
	\begin{enumerate} 
		\item \( [m_f, m_g] = [m_f, E_v] = [ m_f,E_v^\dagger] =   0 \); 
		\item \( [m_f, P_v ] = m_{L_v f} \); 
		\item \( [m_f, \p] =m_{df}. \) 
	\end{enumerate}
\end{prop}

\begin{proof} 
	(a): These follow directly from the definitions. (b): By the Mean Value Theorem there exists \( q_t = p+stv, 0 \le s \le 1 \) such that \( \frac{f(p+tv) - f(p)}{t} = L_vf(q_t) \). Then
	\begin{align*} 
		m_f P_v (p;\a) = m_f \lim_{t \to 0} (p+tv; \a/t) - (p;\a/t) & = \lim_{t \to 0} (p+tv; f(p+tv)\a/t) - (p; f(p) \a/t) \\
		& = \lim_{t\to 0} (p+tv; (f(p) + t L_v f(q_t)) \a/t) - (p; f(p) \a/t) \\
		& = \lim_{t \to 0} (p+tv; f(p) \a/t) - (p; f(p) \a/t) + \lim_{t \to 0} ( p; L_v f(q_t) \a ) \\
		& = P_v m_f (p;\a) + m_{L_v f} (p;\a).
	\end{align*}
	(c):  This follows since \( d(f \o) = df \wedge \o + f \wedge d \o \). 
\end{proof}
 
 We denote the (continuous) dual operator by \( f \cdot \in {\cal L}(\B) \) where \( f \in \B_0 \).

\begin{thm}[Change of density]\label{thm:funcon} The space of differential chains \( \pB(\R^n) \) is a graded module over the ring \( \B_0 (\R^n) \).    If \( f \in \B_0 \), then \( m_f \in {\cal L}(\pB) \) and \( f \cdot \in {\cal L}( \B) \) are continuous graded operators satisfying
	\begin{equation}
		\cint_{m_f J} \o = \cint_J f \cdot \o
	\end{equation} 
	for all matching pairs, i.e.,  \( J \in \hB_k^r \) and \( \o \in \B_k^r \), or \( J \in \pB_k \) and \( \o \in \B_k \).  	 Furthermore,  the bigraded multilinear maps  \( m \in {\cal L}(\B_0 \times \pB, \pB) \), given by \( (f, J) \mapsto  m_f(J) \) is separately continuous.	
	\end{thm}

\begin{proof} 
Separate continuity for elements of \( {\cal L}(\B_0^r \times \hB_k^r, \hB_k^r)  \) follows from \ref{lem:mflem}. 
Now if \( f\in \B_0 \), then each map \( m_f: \hB_k^r \to \hB_k^r \) is continuous. Since \( m_f\circ u_k^r = u_k^r\circ m_f \), we know \( m_f(H_k) \subseteq H_k \).   Separate continuity follows for elements of \( {\cal L}(\B_0 \times \pB, \pB) \).
 \end{proof}

We can actually say a bit more about the continuity of \( m_f \).  Convergence in \( \B_k \) is quite restrictive, and does not lend itself nicely to bump functions and partitions of unity.  In particular, if \( \{\phi_i\} \) is a partition of unity, then \( \sum_{i=1}^{N} \phi_i \nrightarrow 1  \)   in the \( B^r \) norm.   We want \( m_f \) to be continuous under a more local notion of convergence (the \( B^r \) version of the compact-open topology.)  Indeed\footnote{The following lemma and its proof are due to Harrison Pugh.},

\begin{lem}\label{hpugh}
	Let \( f_i, f \in B_0^r \), such that \( f_i \to f \) pointwise and \( \| f - f_i \|_{B^r} < C \) for some \( C \) independent of \( i \).  Then \( m_{f_i} J \to m_f J \) in \( B^r \)-norm for all \( J\in \hB_k^r \).
\end{lem}

\begin{proof}
	Let \( \e>0 \). We show there exists \( N \) such that if \( i>N \) then \( \| m_f J - m_{f_i} J \|_{B^r} < \e \). Suppose not.  Then for all N there exists \( i_N > N \) and \( \o_N \in \B_k^r \) with \( \| \o_N \|_{B^r} = 1 \) such that \[ ((f-f_{i_N}) \o_N)(J) > \e. \] (This is by the definition of the \( B^r \) norm on chains as a supremum over forms of norm \( 1 \).)	Now, let \( A_i \to J \) be Dirac chains.  Then there exists \( M \) such that if \( j>M \), then \( \o(J)- \o(A_i) < \e/2 \) for all \( \o\in \B_k^r \) with \( \|\o\|_{B^r} < nrC \).  	Therefore, putting these two together, and using the fact that \( \|((f-f_{i_N}) \o_N)\|_{B^r} \leq nr \| f - f_i \|_{B^r} \|\o_N\|_{B^r} < nrC \), we have \[ ((f-f_{i_N}) \o_N)(A_j) > \e/2 \] for all \( N \).  Now, fix such a \( j \).  Say \( A_j = \sum_s (p_s;\a_s) \).  Then \[ ((f-f_{i_N}) \o_N)(A_j) = \sum_s (f(p_s)-f_{i_N}(p_s))\o_N(p_s;\a_s). \] Since \( f_i\to f \) pointwise, we can make \( N \) large enough such that \[ \left| f(p_s)-f_{i_N}(p_s)\right| < \frac{\e}{2\max\|\a_s\|_0}\] for all \( s \). Therefore, \[ \sum_s (f(p_s)-f_{i_N}(p_s))\o_N(p_s;\a_s) < \e/2, \] which is a contradiction.
\end{proof}

\subsection{Representatives of functions and vector fields as chains}\label{sub:vector_fields}

We say that a \( k \)-chain \( \widetilde{X} \in \hB_k^1 \) \emph{represents} a \( k \)-vector field \( X \) if \[ \cint_{\widetilde{X}} \o = \int_{\R^n}\o(X(p)) dV \] for all \( \o \in \B_k^1 \).

\begin{thm}\label{thm:jX} 
	If \( X \) is a \( k \)-vector field with compact support and with coefficients in \( \B_0^1 \), then there exists a unique \( k \)-chain \( \widetilde{X} \in \hB_k^1\) that represents \( X \). 
\end{thm}

\begin{proof} 
	Let \( U \) be a bounded open set containing \( \supp(X) \) and \( \widetilde{U} \) the \( n \)-chain representing \( U \) (see \S\ref{sub:cells}). Given a \( k \)-vector \( \a \), then \( E_\a \perp \widetilde{U} \) represents the constant \( k \)-vector field \( (p;\a) \) over \( U \) since for all \( \o \in \B_k^1 \), we have \[ \cint_{E_\a \perp \widetilde{U}} \o = \cint_{\perp \widetilde{U}} i_\a \o = \cint_{\widetilde{U}} \star i_\a \o = \int_U \star i_\a \o = \int_{\R^n}\o(p;\a) dV \] using Corollaries  \ref{cor:ext}, \ref{thm:perpcont},  Proposition \ref{cells}, and the definition of interior product.

	Now suppose \( X(p) = \sum (p; f_I(p)\a_I) \) is a vector field with compact support in \( U \) where \( f_I \in \B_0^1 \). Then \( \widetilde{X} = \sum m_{f_I} E_{\a_I} \perp \widetilde{U} \) represents \( X \) since \[\cint_{m_{f_I}E_{\a_I}\perp \widetilde{U}} \o = \cint_{ E_{\a_I} \perp \widetilde{U}} f_I \o = \int_{\R^n} f_I \o (p;\a_I) dV.\] Therefore, \[ \cint_{\widetilde{X}} \o = \int_{\R^n} (\sum f_I \o) (p; \a_I) dV = \int_{\R^n} \o(\sum (p; f_I(p)\a_I))dV = \int_{\R^n}\o(X(p)) dV. \] 
\end{proof}   

For \( k = 0 \), we have found a \( 0 \)-chain \( \widetilde{f} \) representing a function \( f \in {\cal B}_0^1 \) with compact support. That is, \[ \cint_{\widetilde{f}} g = \int_{\R^n} f\cdot g dV = \int_{\R^n} f \wedge \star g \] for all \( g \in \B_0^1 \).  As an example of this construction, we have the following corollary: 

\begin{cor}\label{thm:df} 
	Suppose \( f \) is a Lipschitz function with compact support. Then \[ \cint_{\perp \p \perp \widetilde{f}} \o = - \int_{\R^n} df \wedge \star \o \] for all \( \o \in \B_1^2 \). 
\end{cor}

\begin{proof} 
	Let \( \o \in \B_1^2 \). By Theorem \ref{thm:perpcont} and Stokes' Theorem \ref{cor:stokes} we have \( \cint_{\perp \p \perp \widetilde{f}} \o = \cint_{\widetilde{f}} \star d \star \o = \int_{\R^n} f \wedge d \star \o = - \int_{\R^n} df \wedge \star \o. \)  
\end{proof}

\subsection{Partitions of unity}\label{ssub:partitions_of_unity} 

If \( A \in \A_k(\R^n) \) is a Dirac chain, and \( W \subseteq \R^n \) is open, let \( A \lfloor_W \) be the \emph{restriction of \( A \) to \( W \)}. That is, if \( A = \sum_{i=1}^s (p_i; \a_i) \), then \( A \lfloor_W = \sum (p_{i_j}; \a_{i_j}) \) where \( p_{i_j} \in W \).

\begin{example}\label{sheaf} 
	Let \( Q \) be the open unit disk in \( \R^2 \). For each integer \( m \ge 1 \), let \( A_m = ((1+1/2m,0); m) - ((1-1/2m,0); m) \) and \( B_m = A_m - ((1,0); e_1 \otimes 1) \). Then \( B_m \to 0 \) as \( m \to \i \) in \( \hB_0^2(\R^2) \), and yet \( B_m \lfloor_Q = (1-1/2m; -m) \) diverges as \( m \to \i \). This example shows that it can be problematic to define the part of a chain in every open set. We will see however in (see \S\ref{sub:inclusions}) that ``inclusion'' is well-defined, so it is more natural to use cosheaves with differential chains, as opposed to sheaves. 
\end{example}

\begin{thm}\label{thm:pou} 
	Suppose \( \{U_i\}_{i=1}^\i \) is a locally finite bounded open cover of \( \R^n \) and \( \{\Phi_i\}_{i=1}^\i \) is a partition of unity subordinate to \( \{U_i\}_{i=1}^\i \) with \( \Phi_i \in \B_0^r(\R^n) \). If \( J \in \hB_k^r(\R^n) \), then \[ J = \sum_{i=1}^\i (m_{\Phi_i} J) \]  where the convergence is in  the \( B^r \) norm.
\end{thm}

\begin{proof} 
 Apply  Pugh's Lemma\footnote{This is a much more elegant approach than the author's original proof using weak convergence in \( {\cal D}' \).} \ref{hpugh} to \( \sum_{i=1}^N \Phi_i \) and the function \( 1 \in \B_0 \).  Then \( \sum_{i=1}^N (m_{\Phi_i} J) =  (\sum_{i=1}^N m_{\Phi_i}) J  \to J \) in the \( B^r \) norm. 
\end{proof}
 
\subsubsection{Partitions of unity and integration} 

If \( I = (0,1) \) and \( \{\phi_s\} \) is a partition of unity subordinate to a covering of \( I = \cup_{s=1}^k I_s \), then we can write the chain representative \( \widetilde{I} \) of \( I \) as a sum \( \widetilde{I} = \sum_{s=1}^k m_{\phi_s} \widetilde{I_s} \). This ties the idea of an atlas and its overlap maps to integration.  Roughly speaking, the overlapping subintervals have varying density, but add up perfectly to obtain the unit interval. We obtain \( \cint_{\widetilde{I}} \o = \sum_{s=1}^k \cint_{m_{\phi_s} \widetilde{I_s} } \o \). This is in contrast with our discrete method of approximating \( \widetilde{I} \) with Dirac chains \( A_m \) to calculate the same integral \( \cint_{\widetilde{I}} \o = \lim \cint_{A_m} \o \).

\section{Support}\label{sub:support}

\subsection{Existence of support}\label{sub:existence_of_support}
 
We show that associated to each chain \( J \in \pB(\R^n) \) is a unique closed subset \( \supp(J) \subseteq \R^n \) called the \emph{support} of \( J \).  

Recall \( \A_k(W) \) is the subspace of Dirac chains \( A \in \A_k(\R^n) \) supported in \( W \subseteq \R^n \). Let \( J \in \pB_k(\R^n) \) and \( X \subseteq \R^n \). We say that \( J \) is \emph{accessible} in \( X \), if for every open set \( W \subseteq \R^n \) containing \( X \), there exists \( A_i \to J \) in \( \pB_k(\R^n) \) with \( A_i \in \A_k(W) \). Clearly, \( J \) is accessible in \( \R^n \). If there is a smallest closed set \( X \) in which \( J \) is accessible, then \( X \) is called the \emph{support} of \( J \) and denoted by \( \supp(J) \). This coincides with the definition of support of a nonzero Dirac chain \( A \) by choosing \( A_i = A \), as well as support of Dirac chains of higher order in \S\ref{ssub:prederivative}. Support is just a map  from \( \pB \) to the power set of \( \R^n \).  It is not linear, nor is it continuous in the Hausdorff metric. Consider, for example, \( A_t =(p; t\a) + (q; (1-t)\a) \). Then \( \supp(A_t) = \{p,q\} \) for all \( 0 < t < 1 \), but \( \supp(A_0) = q \) and \( \supp(A_1) =p \).  

\begin{lem}\label{pro:spp} 
	If \( J \in \pB_k \) is nonzero and there exists a compact set \( X_0 \) in which \( J \) is accessible, then \( \supp(J) \) is a uniquely defined nonempty closed set. 
\end{lem}

\begin{proof} 
	Let \( F \) be the collection of all closed \( X \) such that \( X \subseteq X_0 \) and \( J \) is accessible in \( X \). To show that \( F\) has a smallest element, it suffices to show  
 	\begin{enumerate}
	 	\item[(a)] If \( F' \) is any sub-collection of \( F \) which is totally ordered by \( \subseteq \), then \( X' = \bigcap_{X \in F'} X \) belongs to \( F \); 
		\item[(b)] \( X,Y \in F \implies X \cap Y \in F \). 
 	\end{enumerate}

 	(a): If \( W \) is any open set containing \( X' \), then \( W \) contains some \( X \in F' \). Otherwise, the collection of compact sets \( X - W, X \in F' \) would have the finite intersection property, and hence nonempty intersection.

 	(b): Let \( X, Y \in F \).   Choose sequences of nested open sets \( S_j, T_j \) such that \( \cap_{j=1}^\i S_j = X \) and \( \cap_{j=1}^\i T_j = Y \). We know \( J \) is accessible in both \( X \) and \( Y \).  Thus,  for each \( j \ge 1 \), there exist sequences of Dirac chains \( A_{j,s} \to J, B_{j,s} \to J \) as \( s\to\i \) in \( \pB_k(\R^n) \) such that \( \supp(A_{j,s}) \subset S_j \) and \( \supp(B_{j,s}) \subset T_j \). 

	Let \( R \) be an open neighborhood of \( X \cap Y \). It suffices to show that for large enough \( j \), \( A_{j,s}|_R \to J \) as \( s\to\i \).  Cover the compact set \( X_0 - R \) with finitely many open disks \( \{ W_i \} \) such that either \( \overline{W_i} \cap X \) or \( \overline{W_i}\cap Y \) is empty. Note that if \( X_0 - R \) is empty, then the result is trivial. Choose a partition of unity \(\{ \Phi_i \}_{i=1}^m \) of smooth functions subordinate to \( \{ W_i \}_{i=1}^m \). Since each \( \Phi_i \) is smooth with compact support, it is an element of \( \B_0(\R^n) \). Therefore \( m_{\Phi_i} \) is continuous by Theorem \ref{thm:funcon}. Since \( A_{j,s} - B_{j,s} \to 0 \) in \( \pB_k \), we have \( m_{\Phi_i}(A_{j,s} - B_{j,s}) \to 0 \) in \( \pB_k \) for all \( i \). If \( \overline{W_i} \cap X \) is empty, then \( m_{\Phi_i} A_{j,s} = 0 \) for large enough \( j \). Likewise, if \( \overline{W_i} \cap Y \) is empty, then \( m_{\Phi_i} B_{j,s} = 0 \) for large enough \( j \). This, together with \( m_{\Phi_i}(A_{j,s} - B_{j,s}) \to 0 \) implies that for large enough \( j \), \( m_{\Phi_i} A_{j,s} \to  0 \) as \( s \to \i \). 
	
	By Theorem \ref{thm:pou} \( A_{j,s} = A_{j,s}\lfloor_{R \cup_i W_i} \) and \( A_{j,s}\lfloor_{\cup_i W_i} = \sum_i (m_{\Phi_i} A_{j,s}) \to 0 \). Since \( A_{j,s} \to J \), we deduce that \( A_{j,s}\lfloor_R \to J \).   
\end{proof} 

As a result of this Lemma, if there exists compact \( X \) in which \( J \) is accessible, then \( \supp(J) \) exists, \( \supp(J) \subseteq X \), and \( \supp(J) \) and can be written as the intersection of all compact sets in which \( J \) is accessible.

 \begin{lem}\label{lem:sup1} \mbox{}
	   
\(\quad \hspace{.08in} (a) \)  Suppose \( J \in \pB_k(\R^n) \) has compact support and \( f \in \B_0(\R^n) \).  Then   \( \supp(m_f J) \subseteq \supp(f) \cap \supp(J) \).  

 Suppose  \( J_1, J_2, \dots \in \pB_k(\R^n) \) have compact support and \( \sum_{j=1}^\i J_j \) converges in  \( \pB_k(\R^n) \).  Then	  
	\begin{enumerate}  
		\item[(b)]   \( \supp(\sum_{j=1}^\i J_j) \subseteq \cup_{j=1}^\i \supp(J_j) \);
		\smallskip 
		\item[(c)] \( \supp(\sum_{j=1}^\i J_j) = \cup_{j=1}^\i \supp(J_j) \) if \( \{\supp(J_j)\} \) are pairwise disjoint.
	\end{enumerate} 
\end{lem}                   

\begin{proof} 
	(a): We first show \( m_fJ \) is supported in \( \supp(J) \): Let \( N \) be a neighborhood of \( \supp(J) \) in \( \R^n \). Then there exist Dirac chains \( A_i \to J \) in \( \pB_k(\R^n) \) and supported in \( N \).  By Lemma \ref{lem:mflem}  \( m_{f} A_i \to m_{f} J \). Since \( A_i \) has finite support in \( N \), so does \( m_f A_i \). Therefore \( m_{f} J \) is accessible in the compact set \( \supp(J) \).   Likewise, suppose \( N \) is an open neighborhood of \( \supp(f) \) in \( \R^n \). Choose any \( A_i \in \A_k(\R^n) \) with \( A_i \to J \). Then \( m_fJ \) is supported in \( \supp(f) \)   since \( m_{f} A_i \to m_{f} J \) and \( \supp(m_{f} A_i) \subseteq \supp(f)  \).

	(b): For each \( j \ge 1 \), there exist nested open neighborhoods \( \{V_j^i\}_i \) of \( \supp(J_j) \)  such that \( \supp(J_j) = \cap_i V_j^i \) for each \( j \ge 1 \) (see the remark after Lemma \ref{pro:spp}).  We show that \( \sum_{j=1}^\i J_j \) is accessible in \( \cup_{j=1}^\i \supp(J_j) \). Let \( N \) be an open neighborhood of \( \cup \supp(J_j)\) in \( \R^n \). For each \( j \ge 1 \) and  \( i_j \) large enough, we therefore have \( \cup_{j=1}^\i V_j^{i_j}  \subseteq N \).  For each \( j \ge 1 \),  there exist Dirac chains \( A_j^s \to J_j \) in the \( B^r \) norm with \( \supp(A_j^s) \subseteq V_j^{i_j} \). It follows that there exists \( s(j,m) \) with \( \sum_{j=1}^m ( A_j^{s(j,m)} - J_j)   \to 0 \) as \( m \to \i \).   Therefore,  \( \supp( \sum_{j=1}^\i  A_j^{i_j}) \subseteq  \cup_j V_j^{i_j}  \subseteq N \).    It follows that \( \sum_{j=1}^\i J_j \) is accessible in \( \cup_{j=1}^\i \supp(J_j) \).  Hence \( \supp(\sum_{j=1}^\i J_j) \subseteq \cup_{j=1}^\i \supp(J_j) \).

	(c): We first show that  \( \supp(J) \) is accessible in    \( \supp(J+K) \) if the supports of \( J \) and \( K \) are disjoint.  There exists \( f:\R^n \to \R \) in \( \B_0^r \) with \( f\equiv 1\) on a neighborhood of \( \supp{J} \) and \( f\equiv 0 \) on a neighborhood of \( \supp{K} \).  It follows that  \( m_f J = J \) and \( m_f K = 0 \).    Suppose \( N \) is a neighborhood of \( \supp(J+K) \). There exist Dirac chains \( A_j \to J+K \)  with \( \supp(A_j) \subset N \), and thus \( m_f A_j \to  m_f J + m_f K = J   \).  Since \( m_f A_j \) also has support in \( N \), we are done. Thus \( \supp(J) \subseteq \supp(J+K) \).  Similarly,  \( \supp(K) \subseteq \supp(J+K) \) and we are done
	
	 Fix \( j_0 \ge 1 \).  We show that \( \supp(J_{j_0}) \subseteq \supp(\sum_j J_j) \). Let \( K=   \sum_{j \ne j_0} J_j  \).   Since the supports of the \( J_j \) are pairwise disjoint,  \( \supp(J_{j_0}) \cap \supp(K) = \emptyset \).  According to the preceding paragraph  \( \supp(J_{j_0} ) \cup \supp(K) \subseteq \supp(\sum_{j=1}^\i J_j) \).
	
\end{proof}

We next show \( \supp(J) \) is well-defined for all \( J \in \pB_k \). This is not clear since the sets in which it is accessible may not be bounded.

\begin{thm}\label{thm:supportwell} 
	If \( J \in \pB_k \) is nonzero, then \( \supp(J) \) is a nonempty closed set. 
\end{thm}

\begin{proof} 
		Cover \( \R^n \) uniformly with a locally finite family of balls of radius one \( \{R_i\} \). We may arrange the order such that each \( \bigcup_{i=0}^s R_i \) is connected. Let \( W_0 = R_0 \). Inductively let \( W_k \) be the union of all \( R_i \) that meet \( W_{k-1} \). Let \( \{f_i\} \) be a partition of unity subordinate to the cover \( \{ W_i\} \). Let \( W_e = \bigcup_{i=0} ^\i W_{2i} \) and \( W_o = \bigcup_{i=0}^\i W_{2i+1}\).

	 	Let \( J_s = \sum_{i=1}^s m_{f_i} J \). Then \( J_{2s+1} = {\cal K}_s + {\cal L}_s \) where \( {\cal K}_s = \sum_{i=1}^s m_{f_{2i}} J\) and \( {\cal L}_s = \sum_{i=1}^s m_{f_{2i+1}} J \). Since the supports of \( m_{f_{2i}} J \) are pairwise disjoint and closed, Lemma
	\ref{lem:sup1} (c) implies \( \supp({\cal K}_s) = \bigcup_{i=1}^s \supp (m_{f_{2i}} J)
	\). Similarly, \( \supp({\cal L}_s) = \bigcup_{i=1}^s \supp (m_{f_{2i+1}} J) \). By
	Corollary \ref{thm:pou} \( J = \sum_{i=1}^\i (m_{f_i} J)\). 
	We apply Lemma \ref{hpugh} to conclude
	\( \sum_{i=1}^{\i}(m_{f_{2i}} J) = (\sum_{i=1}^{\i} f_{2i} ) J \) is a well-defined
	differential chain \( K \). 	 Similarly, \( \lim_{s \to \i}(\sum_{i=1}^s m_{f_{2i+1}} J) \) converges to a
	differential chain \( L \). It follows that \( J = K + L \) where \( K = \sum_{i=1}^\i
	m_{f_{2i}} J \) and \( L = \sum_{i=1}^\i m_{f_{2i+1}} J \).

	 Since \( J \ne 0 \) we may assume \( K \ne 0 \), say.  At least one \( m_{f_{2i}} J \ne 0\). By Lemma \ref{lem:sup1} (c) \(
	\supp(K) = \bigcup_{i=1}^{\i} \supp(f_{2i} J) \) since the \( f_{2i} J \) have pairwise disjoint
	supports. Since this chain is compactly accessible,
	we know that its support is nonempty by Lemma \ref{pro:spp}. Thus \( \supp(K) \ne
	\emptyset \). It remains to show that \( \supp(J) \ne \emptyset \).

 	Let \( \psi_N = \sum_{i=1}^{N} f_{2i} \). 
 	By Lemma \ref{lem:sup1} (a) \( \supp(\psi_N J ) \subseteq \supp(J) \).
 	We know \( \supp(\psi_N J ) \ne \emptyset \) for \( N \) sufficiently
	large, and thus   \( \supp(J) \ne \emptyset \).
\end{proof}

 If \( X \) is any subset of \( \R^n \) with \( \supp(J) \subseteq X \), then we say \( J
\) is \emph{supported in \( X \)}. 

\begin{remarks} \mbox{} 
	\begin{itemize} 
		\item Any closed set \( S \) supports a chain. Simply
choose a dense countable subset \( \{ p_i \} \subseteq S \), and note that \( S \) supports
the chain \( J= \sum_{i=1}^{\i}(p_i; 1/2^i) \). There are infinitely many distinct chains
with support \( S \) found by multiplying \( J \) by a function 
		\item The unit interval
\( I = (0,1) \times \{0,0\} \subset \R^2 \) is the support of infinitely many
\( k \)-dimensional chains in \( \pB_k(\R^2) \) for each \( 0 \le k \le 2 \).
Examples include \( E_{e_3} \widetilde{I}, \perp \widetilde{I}, \mbox{ and } P_{e_3}
\widetilde{I} \). 
	\end{itemize} 
\end{remarks}
\subsection{Tests for support}
\label{sub:tests_for_support}

\begin{lem}\label{thm:su} Suppose \( J \in \pB_k(\R^n) \). Then \( \o(J)
= 0 \) for each \( \o \in \B_k(\R^n) \) supported in \( \supp(J)^c \). 
\end{lem}

\begin{proof} By Theorem \ref{thm:supportwell}  \( \supp(J) \) is a well-defined nonempty compact set. Suppose \( \o
\) is supported in \( \supp(J)^c \). There exists a neighborhood \( W \) of \( \supp(J)
\) missing \( \supp(\o) \). By the definition of support, there exists \( A_i \to J \)
with \( A_i \in W \). Therefore \( \o(A_i) \to \o(J) = 0 \). 
\end{proof}

\begin{thm}\label{thm:open}  Suppose \( J \in \pB_k(\R^n) \).
The support of \( J \) is the smallest closed subset \( E \subseteq  \R^n \) such that  \( \cint_J \o = 0 \) for all smooth \( \o \) with compact support disjoint from \( E \).  
\end{thm} 

\begin{proof}   Suppose \( E \subseteq  \supp(J) \) is a closed subset such that   \( \cint_J \o = 0 \) for all smooth \( \o \) with compact support disjoint from \( E \).  We know that \( \supp(J) \) is a well-defined nonempty closed set.  
 If there exists \( p\in \supp(J) - E \), then there exists \( B_\e(p) \) which misses \( E \).  We show \( J \) is accessible in \( \supp(J) - B_\e(p) \):  Let \( U \) be a neighborhood of \( \supp(J) - B_\e(p)  \).  Since \( U \cup B_\e(p) \) is a neighborhood of \( \supp(J) \) there exist Dirac chains \( A_i \to J \) in \( U \cup B_\e(p)  \).  Let  \( f \in  \B_0(\R^n) \) which is identically one on \( U \) and \( 0 \) on \( B_{\e/2}(p) \).  Then \( m_f A_i \to m_f J \).  Suppose \( m_f J \ne J \).  Then there exists \( \o \in \B_k(\R^n) \) with \( \cint_{m_f J - J} \o \ne 0 \).  But \( f \o - \o \) is supported in \( B_\e(p) \) which implies \( \cint_{m_f J - J} \o =  \cint_J f \o - \o = 0 \), a contradiction.  
\end{proof}

\begin{prop}\label{prop:supportW}
  If \( W \subset \R^n \) is bounded and open, then \( \supp(\widetilde{W}) = \overline{W} \).
\end{prop}  

\begin{proof}
  We have to show that \( \supp(\widetilde{W}) \subseteq \overline{W} \) and that \( \overline{W} \) is the smallest closed set supporting \( \widetilde{W} \).  Suppose \( \o \) is supported in \( \overline{W}^c \).  Then \( \o(\widetilde{W}) = 0 \) since we can approximate \( \widetilde{W} \) with Dirac chains supported in \( W \).  Suppose \( E \) is a smaller closed set containing \( \supp(\widetilde{W}) \).  Then there is some point \( x \in W - E \).   Choose an \( \e \)-neighborhood of \( x \) contained in \( W -E \).  Then \( \o(\widetilde{W}) = 0\) for all \( \o \) supported in \( B_\e(x) \), a contradiction.  
\end{proof}

\begin{lem}\label{lem:operator} If \( T \) is a continuous operator on \(
\pB(\R^n) \) with \( \supp(T A) \subseteq \supp(A) \) for all \( A \in {\cal
P} \), then \( \supp(TJ) \subseteq \supp(J) \) for all \( J \in \pB(\R^n)
\). 
\end{lem}

\begin{proof} Suppose \( J \) is supported in \( X \). We show that \( TJ \) is also
supported in \( X \). Let \( W \) be a neighborhood of \( X \). Then there exists \( A_i
\to J \) with \( A_i \) supported in \( W \). Since \( T \) is continuous, \( T(A_i) \to
T(J) \). But \( T(A_i) \in \A \) is supported in \( W \). It follows that \( T(J)
\) is supported in \( X \).
\end{proof}

\section{Differential chains in open sets}\label{sec:chains_in_open_set} 
        
\subsection{Definitions}\label{sub:definitions}

Let \( U \subseteq \R^n \) be open. We say that a \( j \)-difference \( k \)-chain \( \D_{\s^j} (p;\a) \) is \emph{inside} \( W \) if the convex hull of \( \supp(\D_{\s^j} (p;\a))\) is contained in \( W \).    
 
\begin{defn}\label{def:norminopen}   
	\[ \|A\|_{B^{r,U} } := \inf \left \{\sum_{i=1}^{m} \|\s_i^{j(i)}\|\|\a_i\|: A = \sum_{i=1}^{m} \D_{\s_i^{j(i)}}(p_i;\a_i), \,\,\D_{\s_i^{j(i)}}(p_i;\a_i) \mbox{ is inside } U, \mbox{ and } j(i)\leq r \mbox{ for all } i \right\}.\] 
\end{defn}   

The proof that \( \|\cdot\|_{B^{r,U}} \) is a norm on \( \A_k(U) \) is similar to the proof that \(  \|\cdot\|_{B^{r}} \) is a norm and we omit it. Using the norm \eqref{def:norminopen}, we complete the free space \( \A_k(U) \), and obtain a Banach space \( \hB_k^r(U)  \).   Let \( \pB_k(U)  := \varinjlim_{r} \hB_k^r(U) \), endowed with the inductive limit topology \( \t_k(U) \), and \( \pB(U) := \oplus_{k=0}^n \pB_k(U) \), endowed with the direct sum topology.  

\begin{example} 
	Let \( U \) be the open set \( \R^2 \) less the nonnegative \( x \)-axis \footnote{There is a helpful discussion of problems encountered when trying to define chains in open sets, or restrictions of chains to open sets in \cite{whitney} Chapter VIII where we found this example.  However, we do not follow Whitney's lead who required chains in open sets \( U \) to be supported in \( U \).}.  The series\\ \( \sum_{n=1}^\i((n, 1/n^2);1) + ((n,-1/n^2);-1) \) of Dirac \( 0 \)-chains converges in \( \hB_0^1(\R^2) \), but not in \( \hB_0^1(U) \). 
\end{example} 

\subsection{Inclusions} % (fold)
\label{sub:inclusions}

\begin{lem}\label{lem:aseth}  
	Suppose \( U_1 \subseteq U_2 \subseteq \R^n \). If \( A \in \A_k(U_1) \), then \( \|A\|_{B^{r,U_2}} \le  \|A\|_{B^r, U_1} \).  
\end{lem}

\begin{proof} 
	This follows since the definition of \( \|A\|_{B^r, U_2} \) considers more difference chains than does that of \( \|A\|_{B^r, U_1} \). 
\end{proof} 

Define a linear map \( \psi_{k,r}:\hB_k^r(U_1) \to \hB_k^r(U_2) \) of Banach spaces as follows: Suppose \( J \in \hB_k^r(U_1) \). Then there exist Dirac chains \( A_i \to J \) in the \( B^{r,U_1} \) norm.  By Lemma \ref{lem:aseth} \( \{A_i\} \) forms a Cauchy sequence in the \( B^{r,U_2} \) norm.  Let \( \psi_{k,r}(J) := \lim_{i \to \i} A_i \) in the \( B^{r,U_2} \) norm.  The continuous linear map \( \psi_{k,r} \) extends to continuous linear maps \( \psi_k: \pB_{k}^{\i}(U_1) \to \pB_{k}^{\i}(U_2) \), \( \psi_k: \pB_k(U_1) \to \pB_k(U_2) \), and \( \psi: \pB(U_1) \to \pB(U_2) \). These linear maps are injections when restricted to the subspace of chains supported in \( U_1 \), but may not be generally injective unless \( U \) is regular (see the last example in Examples \ref{ex:QQ}).   

Suppose \( U \subseteq \R^n \) is open.  If \( J \in \pB_k(U)  \), define the \emph{support} of \( J \) to be the support of its image in \( \pB_k(\R^n) \).       
We remark that elements of \( \pB_k(U) \) are not necessarily supported in \( U \).  For example, see the first example in Examples \ref{ex:QQ} below.

% subsection inclusions (end)
The Banach space \( (\hB_k^r(U))' \) of differential cochains is isomorphic to the Banach space \( \B_k^r(U) \) of differential forms defined on \( U \) with bounds on each of the directional derivatives of order \( \le r \). (We sketch the proof below.) Elements of \( \B_k^r(U) \) are piecewise \( B^r \)-smooth \( k \)-forms that are continuous on connected components of  \( U \).  
 For example, the characteristic function \( \chi_U \) is an element of \( \B_n^r(U) \). The Banach space \( \B_k^r(U) \) contains all differential forms \( \B_k^r(\R^n) \) restricted to \( U \), but not all elements of \( \B_k^r(U) \) extend to elements of \( \B_k^r(\R^n) \). 

Let \( \B_k(U) = \B_k^\i(U) := \varprojlim \B_k^r(U) \) endowed with the projective limit topology, making it into a Fr\'echet space.   Using the Isomorphism Theorems \ref{thm:dualspaceopen} and \ref{thm:isotW}   below we define \( \cint_J \o := \lim_{i \to \i} \o(A_i) \) for each \( \o \in \B_k^r(U) \) and \( J \in \hB_k^r(U) \) or \( \o \in \B_k(U) \) and \( J \in \pB_k(U) \) where \( A_i \to J \) in the appropriate topology.   

\begin{examples}\label{ex:QQ}\mbox{}
	\begin{itemize}    
		\item Let \( Q\) be the open unit disk in \( \R^2 \). The sequence \( ((1-1/n,0); e_1\wedge e_2) \) converges to \( ((1,0); e_1\wedge e_2) \) in both \( \hB_{2}^{1}(Q) \) and \( \hB_{2}^{1}(\R^2) \). Now \(  (1,0) \notin Q  \), but \( ((1,0); e_1\wedge e_2) \) is still a chain ready to be acted upon by elements \( \o \in \B_2(Q) \). For example, \[ \cint_{((1,0); e_1\wedge e_2)} \chi_Q dxdy = \lim_{n \to \i} \cint_{((1-1/n,0); e_1\wedge e_2)}  \chi_Q dxdy = 1. \]   
		\item Let \( Q' \) be  the unit disk \( Q \) in \( \R^2 \) minus the closed set \( [0,1] \times \{0\} \), and define \[ \o_0(x,y) =
			\begin{cases} 
				\inf\{x,1\}, &\mbox{ if } 0 < x < 1, 0 < y < 1\\
				0, &\mbox{ else}
			\end{cases}. 
			\]  
  			This Lipschitz function \( \o_0 \in \B_0^1(Q') \) is not extendable to a Lipschitz function on \( \R^2 \).   (See Figure \ref{fig:Mapping} for a related example.)
		\item In Figure \ref{fig:openset} there are two sequences of points \( p_i, q_i \) converging to \( x \) in \( \R^2 \). We know  both \( (p_i;1) \to (x;1) \) and \( (q_i;1) \to (x;1) \) in \( \hB_0^1(\R^2) \). Each of the sequences \( \{(p_i;1)\} \) and \( \{(q_i;1)\} \) is Cauchy in \( \hB_0^{1}(U) \) since  the intervals connecting \( p_i \) and \( p_j \) those connecting \( q_i \) and \( q_j \) are subsets of \( U \). Therefore \( (p_i;1) \to A_p \) and \( (q_i;1) \to A_q \) in \(  \hB_0^{1}(U) \). However,  the interval connecting \( p_i \) and \( q_i \) is \emph{not}  a subset of \( U \). The chains \( A_p \ne A_q \) in \( \hB_0^{1}(U) \) since they are separated by elements of the dual Banach space \( \B_0^1(U) \).  For example, let \( \o_1 \) be a smooth form defined on \( U \) which is identically one in the lighter shaded elliptical region, and identically zero in the darker shaded region.  Then \( \cint_{A_p} \o_1  = \lim_{i\to \i} \o_1(p_i;1) = 1 \) and \( \cint_{A_q} \o_1  = \lim_{i \to \i}\o_1(q_i;1) = 0 \). Since \( \psi_{0,1}(A_p) = \psi_{0,1}(A_p) = (x;1) \) we know that the linear map \( \psi_{0,1} \) is not injective. 
	\end{itemize}
\end{examples} 
 
\begin{figure}[ht] 
	\centering 
	\includegraphics[height=2in]{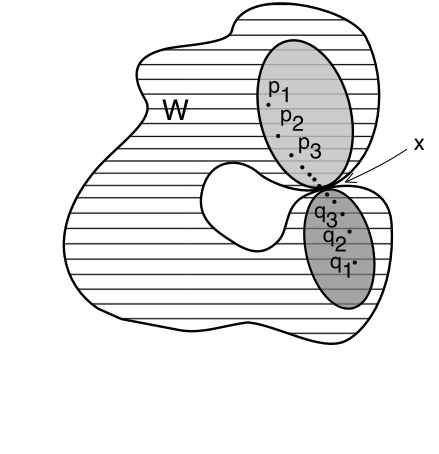}
	\caption{Distinct limit chains supported in \( x \) are separated by a form in \( \pB_0(U) \)} 
	\label{fig:openset} 
\end{figure}

\begin{prop}\label{prop:opensetsinopensets}
 	Suppose \( U \) and \( W \subseteq U \) are open in \( \R^n \). Then there exists a unique \( n \)-chain  \( \widetilde{W}_U \in \hB_n^1(U) \) which represents \( W \). In other words, \( \cint_{\widetilde{W}_U} \o = \int_W \o \) for all \( \o \in \B_n^1(U) \). 
\end{prop}

\begin{proof} 
 	Write \( W = \cup Q_i \) where the sequence \( \{Q_i\} \) is the Whitney decomposition of \( W \).  Each  set \( \overline{Q_i} \) is contained in \( W \subseteq \R^n \), and we may therefore define \( \widetilde{Q_i}_U: = \psi_{n,1}^{-1}\widetilde{Q_i} \), where \( \widetilde{Q_i} \) is the representative of \( Q_i \) in \( \hB_n^1(\R^n) \) (see Theorem \ref{thm:opensets}).   The series \( \widetilde{W}_U := \sum_{i=1}^\i \widetilde{Q_i}_U \) is well-defined since the series converges in the \( B^{1,U} \) norm.  (We do not have to use any \( 1 \)-difference chains to prove convergence.)  Finally, \( \cint_{ \widetilde{W}_U} \o =  \lim_{N \to \i} \sum_{i=1}^N \cint_{ \widetilde{Q_i}_U} \o = \lim_{N \to \i} \sum_{i=1}^N \cint_{ \widetilde{Q_i}} \o = \int_U \o \) for all \( \o \in \B_n^1(U) \).  The last integral is the Riemann integral which is well-defined.
\end{proof}  

If \( W \subset U \) is open with \( \overline{W} \subset U \), then  \( \widetilde{W}_U \in \hB_n^1(U) \) is canonically identified with \( \widetilde{W}  \in \hB_n^1(\R^n) \) since \( \widetilde{W} \) is supported in \( U \).  We therefore may write \(  \widetilde{W}_U =  \widetilde{W} \) without confusion.  However, we do not wish to limit ourselves to open subsets \( W \) of \( U \) with \( \overline{W} \subset U \).   

We can show that the operators \( E_v, E_v^\dagger, P_v  \) are well-defined, graded and continuous on \( \pB(U) \).  Here is the main step:   Let \( A_i \to J \) in \( \pB_k(U) \) where \( A_i \) are Dirac chains supported in \( U \).  Then \( \{A_i\} \) is Cauchy in \( \pB_k(U) \) implies the sequences \( \{E_v A_i\}, \{E_v^\dagger A_i\}, \{P_v A_i\} \) are all Cauchy in \( \pB(U) \).  Define \(  E_v J:= \lim_{i \to \i}E_vA_i, \, E_v^\dagger J:= \lim_{i \to \i}E_v^\dagger A_i \) and \( P_v J:= \lim_{i \to \i}P_vA_i \).    

\begin{example}\label{ex:Qpic}\mbox{}
	Again consider \( Q' \), the open unit disk \( Q \) in \( \R^2 \), less the interval \( [0,1] \times \{0\} \). It follows that \( \widetilde{Q'} = \widetilde{Q} \in \hB_2^1(\R^2) \) since evaluations by \( n \)-forms in \( \B_n(\R^2) \) are the same over \( Q' \) and \( Q \). But \( \widetilde{Q'}_{Q'} \in \hB_2^1(Q') \) and is simply not the same chain as \( \widetilde{Q'} \). Not only are these chains in different topological vector spaces, but their boundaries are qualitatively different. Let \( L_+ \in  \hB_1^2(Q') \) be the \( 1 \)-chain representing the interval \( [0,1] \times \{0\} \) found by approximating \( [0,1] \times \{0\} \) with intervals \( [0,1] \times \{1/n\} \cap Q \).  Similarly, let \( L_- \in \hB_1^2(Q') \) be the \( 1 \)-chain representing the interval \( [0,1] \times \{0\} \) found by approximating \( [0,1] \times \{0\} \) with intervals \( [0,1] \times \{-1/n\} \). Recall \( \o_0 \) as defined in Examples \ref{ex:QQ}. Then \( \o_0^2 dx \in \B_1^2(Q') \), \( \cint_{L_+} \o_0^2 dx = \lim_{n \to \i}\cint_{\widetilde{[0,1] \times \{1/n\}}} \o_0^2 dx = 1 \) and \( \cint_{L_+} \o_0^2 dx = \lim_{n \to \i} \cint_{\widetilde{[0,1] \times \{1/n\}}} \o_0^2 dx = 0 \).  It follows that \( L_+ - L_- \ne 0 \) and thus \( \p(\widetilde{Q'}_{Q'}) = L_+ - L_- + \widetilde{S^1}_{Q'} \ne \widetilde{S^1}_{Q'} \) while \( \p \widetilde{Q'}=\widetilde{S^1} \) (see Figure \ref{fig:Mapping}).    
\end{example}

\subsection{Multiplication by a function in an open set} \label{sub:mult}

We next show that the Banach space \( \hB_k^r(U) \) is a module over the ring of operators \( \B_0^r(U) \).   The next result is a straightforward extension of Lemma \ref{lem:X}.

\begin{lem}\label{lem:XW} 
	For each \( U \subseteq \R^n \) open, the space of functions \( \B_0^r(U) \) is a ring such that  if \( f, g \in \B_0^r(U) \), then \( f \cdot g,  f+ g \in \B_0^r(U) \), \( \|f \cdot g\|_{B^{r,U}} \le r \|f\|_{B^{r,U}} \|g\|_{B^{r,U}} \), and \( \|f + g \|_{B^{r,U}} \le \|f\|_{B^{r,U}} + \|g\|_{B^{r,U}} \). 
\end{lem}

Given \( f:U \to \R \) where \( U \subseteq \R^n \) is open, define the linear map \( m_f: \A_k(U) \to \A_k(U) \) by \( m_f (p;\a) := f(p)(p; \a) \) where \( (p;\a) \) is an arbitrary \( k \)-element with \( p \in U \).

\begin{thm}\label{thm:continfuncII} 
	Let \( U \) be open in \( \R^n \) and \( r \ge 1 \). The bilinear map 
	\begin{align*} 
		m:  \B_0^r(U) \times \A_k(U) &\to \A_k(U) \\ 
		(f, A) &\mapsto m_f A
	\end{align*}  
	satisfies \( \|m_f A\|_{B^{r,U}} \le nr \|f\|_{B^{r,U}} \|A\|_{B^{r,U}} \). It therefore extends to  continuous bilinear maps \( \hat{m}: \B_0^r(U) \times \hB_k^r \to \hB_{k+1}^r \) for each \( 0 \le k \le n-1 \) and each \( 0 \le r < \i\), and \( E: \B_0(U) \times \pB \to \pB \).
\end{thm}

\begin{proof} 
	Suppose \( \o \ne 0 \). Then \( \frac{|\cint_{m_f A} \o|}{\|\o\|_{B^{r,U}}} \le \frac{\|f \cdot \o\|_{B^{r,U}} \|A\|_{B^{r,U}}}{\|\o\|_{B^{r,U}}} \le nr\|f\|_{B^{r,U}}\|A\|_{B^{r,U}} \) by Lemma \ref{lem:XW}. The result follows. 
\end{proof} 

\begin{cor}\label{cor:changeofden}   
	Let \( U \) be open in \( \R^n \),  and \( f \in  \B_0^r(U)\).  Then \( m_f \in {\cal L}(\pB) \) and \( f \cdot \in {\cal L}( \B) \) are continuous graded operators.  Furthermore, 
		\begin{equation}\label{primitivedensity} 
			\cint_{m_f J} \o = \cint_J f \cdot \o
		\end{equation} 
	for all matching pairs, i.e., \( J \in \hB_k^r(U) \) and \( \o \in \B_k^r(U) \), or \( J \in \pB_k(U) \) and \( \o \in \B_k(U) \).
\end{cor}

\subsection{Sketch proof of the isomorphism theorems for chains in open sets}\label{sub:sketch_proof_of_the_isomorphism_theorem}
 
For \( X \in \A_k(\R^n)^* \) let 
\begin{equation}\label{XjU}
	|X|_{B^{j,U}} := \sup \{ |X(\D_{\s^j} (p;\a))|: \|\s\|\|\a\| = 1,  \D_{\s^j} (p;\a)  \mbox{ is inside } U \}. 
\end{equation}

Let \( \|X\|_{B^{r,U}} = \max\{|X|_{B^{0,U}}, \dots, |X|_{B^{r,U}} \} \).

\begin{lem}\label{lem:ineqW} 
	Let \( A \in \A_k(U) \) be a Dirac chain and \( X \in \A_k(\R^n)^* \) a differential cochain. Then \[ |X(A)| \le \|X\|_{B^{r,U}} \|A\|_{B^{r,U}}. \]
\end{lem}

\begin{proof} 
	The proof of Lemma \ref{lem:ineq} extends with appropriate changes.
\end{proof} 

\begin{thm}\label{thm:normspaceW} 
	\( \| \cdot \|_{B^{r,U}} \) is a norm on the Banach space  \( \hB_k^r(U) \) of differential \( k \)-chains in \( U \). 
\end{thm}   

\begin{proof} 
	The proof of Theorem \ref{thm:normspace} carries over without a problem.
\end{proof}

Let \( \|X\|_{B^{r,U}}= \sup_{0 \ne A} \frac{|X(A)|}{\|A\|_{B^{r,U}}} \). The proof that \( \|X\|_{B^{r,U}} = \|X\|_{B^{r,U}} \) is the same as before (with \( U = \R^n \)), except we only consider difference chains \( \D_{\s^j} (p;\a) \) inside \( U \).

For \(\o \in \A_k(U)^* \), let \( \|\o\|_{B^{r,U}} : \sup \{\cint_{\D_{\s^j} (p;\a)} \o/\|\s\|\|\a\| :  \D_{\s^j} (p;\a)  \mbox{ is inside } U \}\). Let \( \B_k^r(U) := \{ \o \in \A_k(U)^*: \|\o\|_{B^{r,U}} < \i\} \). Then \( \B_k^r(U) \) consists of elements \( \o \in \A_k(U)^* \) with \(f \cdot \o \in \B_k^r(\R^n) \) for each infinitely smooth \( f \) supported in \( U \). (We extend \( f \cdot \o \) to all of \( \R^n \) by setting it equal to zero outside of \( \overline{U} \).)

Let \( \Theta_{k, r}:\B_k^r(U) \to (\A_k(U))^* \) be given by \( \o \mapsto \o(A) \) for all \( A \in \A_k(U) \).

\begin{thm}\label{thm:dualspaceopen} 
	\( \Theta_{k,r}:\B_k^r(U) \to (\hB_k^r(U))' \) is an isomorphism for all \( r \ge 1, 0 \le k \le n \) and \( U \subseteq \R^n \) open and regular. Furthermore, \( \|\Theta_{k,r}(\o)\|_{B^{r,U}} = \|\o\|_{B^{r,U}} \) for all \( 0 \le r < \i \). 
\end{thm} 
                                
\begin{proof} 
	Let \( \o \in \B_k^r(U) \). Then \( f \cdot \o \in \B_k^r(\R^n) \) for all smooth \( f \) with compact support in \( U \). We first show \( \Theta_{k,r}( \o) \in (\hB_k^r(U))' \). Let \( A_i \to 0 \) in \( \hB_k^r(U) \). Then \( A_i \to 0 \) in \( \hB_k^r(\R^n) \) by Lemma \ref{lem:aseth}. Since \( m_f \) is continuous in \( \hB_k^r(\R^n) \), then \( m_f A_i \to 0 \). Therefore, \( \o(m_f A_i) = (f \cdot \o)(A_i) \to 0 \) since \( f \cdot \o \in \B_k^r(\R^n) \). But \( A_i = \sum m_{f_j} A_i \) where \( \{f_j\} \) is a partition of unity of \( U \). Since the partition of unity is locally finite, \( \o(A_i) = \sum \o(m_{f_j} A_i) \to 0 \). Therefore, \( \Theta_{k,r}(\o) \in (\hB_k^r(U))' \).  Suppose \( \Theta_{k,r}(\o) = 0 \). Then \( \o (p;\a) = 0 \) for all \( (p;\a) \). Therefore, \( \o = 0 \).  This shows \( \Theta_{k,r} \) is injective.

	Finally, we show \( \Theta_{k,r} \) is surjective. Suppose \( X \in (\hB_k^r(U))' \). Let \( \o (p;\a) := X (p;\a) \) for all \( p \in U \). We show \( \o \in \B_k^r(U) \). This reduces to showing \( f \cdot \o \in \B_k^r(\R^n) \), which itself reduces to showing \( X m_f \in (\hB_k^r(\R^n))' \) by  Theorem \ref{lem:seminorm}. Suppose \( A_i \to 0 \) in \( \hB_k^r(\R^n) \). Then \( m_f A_i \to 0 \) in \( \hB_k^r(\R^n) \) and thus \( m_f A_i \to 0 \) in \( \hB_k^r(U) \) by Lemma \ref{lem:aseth}. Therefore \( X(m_f A_i) \to 0 \), and we conclude that \( X m_f \in (\hB_k^r(\R^n))' \). Since \( \Theta_{k,r} (\o) = X \), we have established that \( \Theta_{k,r} \) is surjective.

	Preservation of norms is similar to \ref{lem:seminorm}.
\end{proof}

Let \( \pB_k(U) = \varinjlim \hB_k^r(U) \) and endowed with the inductive limit topology.

\begin{thm}\label{thm:isotW} 
	The topological vector space of differential \( k \)-cochains \( (\pB_k(U))'  \) is isomorphic to the Fr\'echet space of differential \( k \)-forms \( \B_k(U) \). 
\end{thm}

The proof is the same as that for Theorem \ref{thm:isot}.

\subsection{Pushforward and change of variables} \label{ssub:pushforward}

Suppose \( U_1 \subseteq \R^n\) and \( U_2 \subseteq \R^m \) are open  and \( F:U_1 \to
U_2 \) is a differentiable map. For \( p \in U_1 \), define \emph{linear pushforward} \( F_{p*}(v_1 \wedge \cdots \wedge
v_k) := DF_p(v_1) \wedge \cdots \wedge DF_p (v_k) \) where \( DF_p \) is the total
derivative of \( F \) at \( p \). Define \emph{pushforward} \( F_*(p; \a) := (F(p), F_{p*}\a) \) for all simple \( k \)-elements \( (p;\a) \) and extend  to a linear map \( F_*:\A_k(U_1) \to \A_k(U_2) \).   Observe that we permit maps which do not extend to differentiable maps in a neighborhood of \( \overline{U_1} \) such as the map \( F \) depicted in Figure \ref{fig:Mapping} (see Example   \ref{ex:Qpic} for an explanation of the notation in Figure \ref{fig:Mapping}).  
                                                                         \begin{figure}[ht] 
	\centering 
	\includegraphics[height=2in]{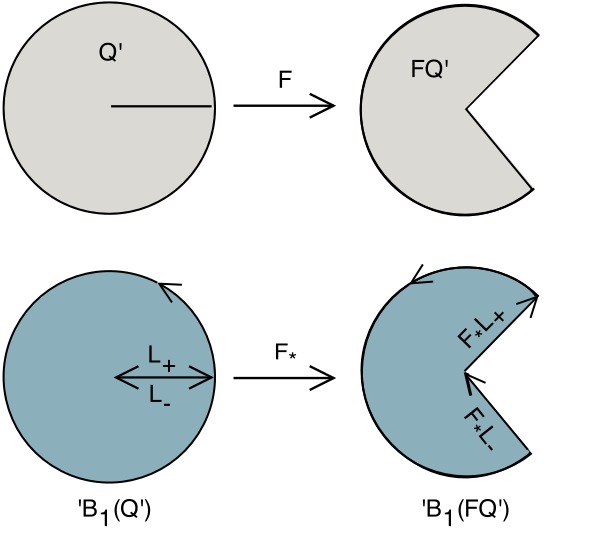}
\caption{A smooth map \( F \) that is not extendable to a neighborhood of \( \overline{Q'} \) nevertheless determines a well-defined pushforward map \( F_*:\pB_1(Q') \to \pB_1(F(Q')) \)} \label{fig:Mapping} \end{figure}

 The classical definition of pullback \( F^* \) satisfies the relation \( F^*\o (p;\a) =
\o(F(p); F_* \a)= \o F_* (p;\a) \) for all differentiable maps \( F:U_1 \to U_2 \),
exterior forms \( \o \in \A_k(U_2)^* \), and simple \( k \)-elements \( (p;\a) \) with
\( p \in U_1 \). 

\begin{defn}\label{def:Mr} \mbox{} \\ Let \( {\cal M}^r(U,\R^m) \) be
the vector space of \( C^r \) differentiable maps \( F:U \to \R^m \)  such that the  
 directional derivatives of its coordinate functions  \( F_{e_i}(p) := \< F(p),e_i\> \)  
  are elements of  \( \B_0^{r-1}(U) \).  
For \( F \in {\cal M}^r(U,\R^m) \), define the seminorm  
  \( |F|_{D^{r,U}} := \max_{i,j}
	   \{\|L_{e_j} F_{e_i}\|_{B^{r-1, U}}\}
	 \).  
	 \end{defn}	
	
	The \( D^{r,U} \) seminorms on maps are increasing as \( r \to \i \).  For \( r \ge s \ge 2 \), let \( g_{r,s}: {\cal M}^r(U,\R^m) \hookrightarrow {\cal
	M}^s(U,\R^m) \) be the natural inclusions.   
	 The linear maps \( \{g_{r,s}\} \) form an
	inverse system. We may therefore define \({\cal M}(U,\R^m) =  {\cal M}^\i(U,\R^m) := \varprojlim {\cal
	M}^r(U,\R^m) \) and endow it with the projective limit topology.  Let \( {\cal M}^r(U_1, U_2) \) denote the subset   \( \{F \in {\cal M}^r(U_1, \R^m) : F(U_1) \subseteq U_2 \subseteq \R^m\} \) where \( U_2 \) is open for all \(r \ge 0 \).  
	
\begin{remarks} \mbox{}
	\begin{itemize}
		\item Since pushforward \( F_* \) relies on the total derivative, the matrix of partial derivatives, this is the information we need to consider in the definition of the seminorms \(  |F|_{D^{r,U}} \).
		\item \(  |\cdot|_{D^{r,U}} \) is not a norm since all constant  maps with \( F(x) = p \) for fixed \( p \in \R^m \)   are elements of \( {\cal M}^r(U,\R^m) \).
		\item  \( {\cal M}^r(U_1,U_2) \) is not generally a vector space unless \( U_2 = \R^m  \).
		\item  A \( 0 \)-form \( f_0 \) in \( \B_0^r(U) \) can be used to create a map \( f(x):= f_0(x;1) \) in \( {\cal M}^r(U,\R) \),  a multiplication operator \( m_f(p;\a) := (p; f(p)\a)  \) in the ring \( {\cal R}^r(U) \), or a  linear map (pushforward) \( f_*: \hB_0^r(U)	 \to \hB_0^r(\R)) \) with \( f_*(p;\a) = (f(p); f_* \a) \).
		\item 	The map \( f(x) = x \) is an
				element of \( {\cal M}^r(\R,\R) \), although the \( 0 \)-form \( f_0(x;1) = x \) is not in \(	\B_0^r(\R) \). 

	\end{itemize}
\end{remarks}  
 Need \( |F|_{Lip} \le |F|_{D^1} \)   need \( |L_v F|_{B^0} \le |F|_{D^1} \)   \( \|F_*v\|\le |F|_{D^1} \).
\begin{lem}\label{lem:Y}  
If \( F \in {\cal M}^r(U,\R^m) \)  and \( 1 \le r < \i \), then   \[  \|F^* \o\|_{B^{r,U}}  \le  \max \{ 1, r |F|_{D^{r,U}}\}  \|\o\|_{B^r}   \] 	for all   \( \o \in \B_k^r(\R^m)
	\).
\end{lem}
 
\begin{proof}    For \( k = 0 \), the proof follows using the product and chain rules.  In general, interior product can be used to reduce the result to the zero dimensional case.

\end{proof}

\begin{thm}\label{thm:prop2}  Let \( U \subseteq \R^n \) be open and \( F \in {\cal M}^r(U,\R^m) \).  The  linear maps  determined by
\begin{align*} F_*:\A_0(U) &\to \A_0(\R^m) \\  (p; 1)  &\mapsto
(F(p); 1)
\end{align*} and, for \( k > 0 \),
\begin{align*} F_*:   \A_k(U) &\to \A_k(\R^m) \\  (p; \a)  &\mapsto (F(p); F_*\a)
\end{align*} where \( (p;\a) \) is a simple \( k \)-element with \( p \in U \), satisfy \(
\|F_*(A)\|_{B^r} \le  \max \{ 1, r |F|_{D^{r,U}}\} \|A\|_{B^{r,U}} \) for all
\( A \in \A_0(U) \) and all \( r \ge 0 \). 
\end{thm}

\begin{proof}  Let \( A \in \A_k(U) \).   By Theorem \eqref{basic} and Lemma \ref{lem:Y},
\begin{align*}   
	 \|F_*A\|_{B^r} = \sup_{0 \ne \o \in \B_k^r} \frac{\left| \cint_A F^* \o \right| }{\|\o\|_{B^r}} \le \sup_{0 \ne \o \in \B_k^r}\frac{ \|F^* \o\|_{B^{r,U}}}{\|\o\|_{B^r}} \|A\|_{B^{r,U}}  \le \max \{ 1, r |F|_{D^{r,U}}\} \|A\|_{B^{r,U}}.
\end{align*} 
The remaining properties
hold using general methods as in \S\ref{ssub:creation_and_annihilation_operators}.
\end{proof}
 
\begin{cor}\label{cor:pull} [Change of variables] If \( F \in {\cal M}^r(U_1,U_2) \) is
a \( B^r \) smooth map, then \( F_*: \hB_k^r(U_1) \to \hB_k^r(U_2) \) and \(
F^*: \B_k^r(U_2) \to \B_k^r(U_1) \) are continuous linear maps for all \(r \ge 1\).  These maps extend to continuous graded linear maps \(F_*:\pB(U_1) \to \pB(U_2) \) and \( F^*: \B(U_2) \to \B(U_1) \).  Furthermore, 
 \begin{equation}\label{eq:pushf} \cint_{F_*J} \o = \cint_J F^* \o \end{equation}
for all maps \( F \in {\cal M}^r(U_1,U_2) \), differential chains \( J \in \hB_k^r(U_1) \), and differential forms \( \o \in \B_k^r(U_2) \) or \( J \in \pB_k(U_1) \) and \( \o \in \B_k(U_2) \). 
\end{cor}

\begin{proof} This follows from Theorems \ref{thm:prop2} and  \ref{lem:seminorm}.
\end{proof}
  
\begin{prop}\label{prop:whitneybd} \( F_* \p = \p F_* \) 
\end{prop}

\begin{proof}  This follows easily from the dual result on differential forms \( F^* d = F^* d \).      
\end{proof}

 Compare \cite{lax1} \cite{lax2} which is widely considered to be the natural change of
variables for multivariables. We offer a coordinate free change of variables theorem in
arbitrary dimension and codimension\footnote{Whitney defined the pushforward operator \( F_* \) on polyhedral chains and extended it
to sharp chains in \cite{whitney}.  He proved a change of variables formula \eqref{eq:pushf}
for Lipschitz forms. However, the important relation \( F_* \p = \p F_* \) does not hold
for sharp chains since \( \p \) is not defined for the sharp norm. (See Corollary
\ref{cor:FBD} below.)  The flat norm of Whitney does have a continuous boundary operator, but flat forms are highly unstable. The following
example modifies an example of Whitney found on p. 270 of \cite{whitney} which he used  to show that components of flat forms may not be flat. But the same example shows that the flat
norm has other problems.  The author includes mention of these problems of the
flat norm since they are not widely known, and she has encountered more than one person
who has devoted years trying to develop calculus on fractals using the flat norm.
Whitney's great contributions to analysis and topology are not in question, and his basic
idea that ``chains come first'' was right. 
\textbf{Example}: In \( \R^2 \), let \( \o_t(x,y) = \begin{cases} e_1 + e_2 + tu, &x+y < 0\\ 0, &
x+y > 0  \end{cases} \)  where \( t \ge 0 \) and \( u \in \R^2 \) is nonzero.
  Then \( \o_0 \) is flat, but \( \o_t \) is not flat for any \( t > 0 \).  In particular, setting \( t = 2, u = -e_2 \), we see that \( \star \o_0 \) is not
flat.}.

\begin{prop}\label{pro:support2} \mbox{}

 If \( F \in {\cal M}^r(U_1,U_2)
\), then \( \supp (F_*J) \subseteq F(\supp J) \subseteq U_2 \) for
all  \( J \in \pB_k(U_1) \) supported in \( U_1 \),  and \( \supp (F_*J) \subseteq F(\supp J) \subseteq \overline{U_2} \) for
all  \( J \in \pB_k(U_1) \). If \( G \in {\cal
M}(U,U_1) \), then \( F_* \circ G_* = (F \circ G)_* \). 
\end{prop}

\begin{proof} We show that \( F_*J \) is supported in \( F(\supp(J)) \) if \( J \) is supported in \( U_1 \) . Let \( N  \subseteq U_2\) be
a neighborhood of \( F(\supp (J)) \). Then \( F^{-1}N  \subseteq U_1 \) is a neighborhood of \( \supp(J)
\). Hence there exists \( A_i \to J \) supported in \( F^{-1}N \) and therefore \( F_*A_i
\to F_* J \), showing that \( F_*J \) is supported in \( N \).  
\end{proof} 

\subsection{Integral theorems in open sets} % (fold)
\label{sub:integral_theorems_in_open_sets}

% subsection integral_theorems_in_open_sets (end)
The integral theorems of \S\ref{ssub:creation_and_annihilation_operators} and \S\ref{sub:boundary}, e.g., the general Stokes' Theorem, extend to suitable pairings of elements of \( \pB(U) \) and \( \B(U) \).  They  apply to differential chains supported in the boundary of \( U \), even for non-regular open sets \( U \), paired with smooth differential forms defined on \( U \) which may not be extendable to a smooth form defined on a neighborhood of \( U \).  This is not simply a restriction of the earlier results to open sets, and the venerable  Stokes' theorem, in particular, reaches one more nontrivial level of generality:  

\begin{thm}[General Stokes' Theorem for Open Sets]\label{cor:stokesopen}  \[
	\cint_{\p J} \o = \cint_J d \o
\]  for all \( J \in \hB_{k+1}^{r-1}(U) \)  and  \( \o \in \B_k^r(U) \), or \( J \in \pB_{k+1}(U) \) and \( \o \in \B_k(U) \),  where \( r \ge 1 \) and \( 0 \le k \le n-1 \).
\end{thm}  

This result applies to open sets \( U \) which are not regular, and thus to differential forms which are not extendable to smooth forms defined in a neighborhood of \( U \).

\section{Algebraic chains, submanifolds, soap films and fractals}
\label{sub:algebraic_chains}
\subsection{Algebraic chains}

An \emph{algebraic \( k \)-cell in an open set} \( U_2 \subseteq \R^n \) is a
differential \( k \)-chain \( F_* \widetilde{Q} \) where \( Q \) is an affine \( k \)-cell
contained in \( U_1 \subseteq \R^n\) and the map \( F:U_1 \to U_2 \) is an element of \( {\cal M}^r(U_1,U_2) \). We
say that \( F_* \widetilde{Q} \) is \emph{non-degenerate} if \( F:Q \to U_2 \) is a
diffeomorphism onto its image. An \emph{algebraic \( k \)-chain \( A \) in \( U_2 \)} is a
finite sum of algebraic \( k \)-cells \( A = \sum_{i=1}^N a_i F_{i*} \widetilde{Q_i} \) where \( a_i \in \R \), although it is often natural to assume \( a_i \in \Z \) for an ``integral'' theory. According to Proposition \ref{pro:support2} it
follows that \( \supp(A) \subseteq U_2 \). Algebraic \( k \)-chains offer us relatively simple ways to represent familiar mathematical ``objects'' (see the figures below), and all of our integral theorems hold form the.  For example,  the
change of variables equation for algebraic chains (see Corollary \ref{cor:pull}) takes the form \[ \cint_{ \sum_{i=1}^N
a_i F_{i*} \widetilde{Q_i}} \o = \sum_{i=1}^N \int_{a_i Q_i} F_i^*\o. \] The integral
on the right hand side is the Riemann integral for which there are classical methods of
evaluation.

 Singular cells are quite different from algebraic cells. A \emph{singular cell} is
defined to be a map of a closed cell \( G: Q \to U \) and \( G \) might only be continuous.   
For example, let \( G: [-1,1] \to \R \) be given by \( G(x) = x \) if \( x \ge 0 \) and
\( G(x) = -x \) if \( x \le 0 \). This singular cell is nonzero, but the algebraic cell
\( G_*(\widetilde{[0,1]}) = 0 \). This problem of singular cells vanishes in homology,
but the algebra inherent in algebraic chains is present before passing to homology.

\subsection{Submanifolds of \( \R^n \)}  \label{sub:submanifolds}

 We next show that smooth \( k \)-submanifolds in an open subset \( U \subseteq \R^n \) are uniquely represented by algebraic \( k \)-chains.   

 Two non-degenerate algebraic \( k \)-cells \( F_{1*}\widetilde{Q}_1 \) and \(
F_{2*}\widetilde{Q}_2 \) are \emph{non-overlapping} if \[
\mathrm{supp}(F_{1*}\widetilde{Q}_1) \cap \mathrm{supp}(F_{2*}\widetilde{Q_2})
\subseteq \mathrm{supp}(\partial F_{1*}\widetilde{Q_1}) \cup \mathrm{supp}(\partial
F_{2*}\widetilde{Q_2})\] That is, \( F_1(Q_1) \) and \( F_2(Q_2) \) intersect at
most at their boundaries if they are non-overlapping. An algebraic chain \( A = \sum F_{i*}\widetilde{Q_i} \) is \emph{non-overlapping }if each pair in the set \(\{F_{i*}\widetilde{Q_i} \} \) is non-overlapping.  
We say that algebraic \( k \)-chains \( A \) and \( A' \) are \emph{equivalent} and write \( A \sim A' \) if \( A = A' \) as differential \( k \)-chains.   

A \( k \)-chain \( J \in \hB_k^r(U) \) \emph{represents} a \( k \)-submanifold \( M \) of class \( C^{r-1+Lip} \) in \( U \) if \( \cint_J \o = \int_M \o \) for all forms \( \o \in \B_k^r(U) \).  If there exists a \( k \)-chain \( J \) representing \( M \), then it is unique.      Even though two algebraic chains \( A \) and \( A' \) may have different summands,  they are still identical as chains if they both represent \( M \). 

\begin{thm}\label{thm:subm} Smoothly embedded compact \( k \)-submanifolds \( M \) of \( U  \)  of class \( C^{r-1+Lip} \)	are in one-to-one correspondence with   equivalence classes  of algebraic \( k \)-chains \(  A \in \hB_k^r(U) \)   satisfying the following two properties: 
	 \begin{enumerate}
		\item  For each \( p \in \supp(A) \)  there exists a   non-overlapping  algebraic \( k \)-chain \( A_p = \sum_{i=1}^N a_i F_{i*} \widetilde{Q_i} \)  with \(  A_p \sim  A \) and     	\( p \notin   \supp(\p F_{i*}\widetilde{Q_i}) \) for any \( i \). 
		\item \( \partial A = \{0\} \) for any \( A \) representing \( [A] \).
	\end{enumerate}
Furthermore, the correspondence is natural.  That is, \( A \leftrightarrow M  \) if and only if \( A \) represents \( M \).
\end{thm}

\begin{proof} 
Let \( M \) be a smoothly embedded compact \( k \)-submanifold of class \( C^{r-1+Lip} \).  Cover \( M \) with finitely many locally embedded charts diffeomorphic to simplices. The embeddings are of class \( B^r \)  by Lemma \ref{lem:oncemore}. Let \( p\in M \).     Choose the cover  so that \( p \) does not meet any of the chart boundaries.  Take a simplicial subdivision of the cover so that \( M = \cup_i F_i(\s^i) \). Then \( A = \sum F_{i*} \widetilde{\s^i} \)   is a   non-overlapping algebraic \( k \)-chain. We show that \( A \) represents \( M \).  Suppose \( \o \in \B_k^r(U) \).  Choose exist simplicial neighborhoods \( \t_i \) of \(  \s^i  \) such that  \( F_i \) extends to an embedding of \( \t_i \) and \( M = \cup_i F_i(\t_i) \).  Choose a partition of unity \( \{\phi_i\} \) subordinate to \( \{F_i \t_i\} \).    Using  Theorem \ref{thm:opensets}, Corollary \ref{cor:pull},  Corollary \ref{cor:changeofden}, and Corollary \ref{thm:pou} we deduce 
\begin{align*} \int_M \o = \sum \int_{F_i \t_i} \phi_i \o &= \int_{\t_i} F_i^* \phi_i \o = \cint_{\widetilde{\t_i}} F_i^* \phi_i \o \\&= \sum \cint_{F_{i*} \widetilde{\t_i}} \phi_i \o  =\sum \cint_{m_{\phi_i} F_{i*} \widetilde{\t_i}}\o  =  \sum \cint_{ F_{i*} \widetilde{\s^i}}   \o = \cint_A \o.  
	\end{align*} This shows that \( A \) satisfies condition (a) and represents \( M \).   Since \( M \) has no boundary, it follows that \( \p A = 0 \). 

Conversely, let \( A= \sum_i F_{i*}\widetilde{Q_i}  \) be a non-overlapping algebraic \( k \)-chain satisfying  (a) and (b).   We show that \( M = \supp(A) \) is a smoothly embedded compact \( k \)-submanifold.  For each \( p \in  M \) we may choose a non-overlapping representative \( A_p =  \sum G_{i*}\widetilde{Q_i'} \) of the equivalence class \( [A] \) such that  	\( p \notin   \supp(\p G_{i*}\widetilde{Q_i'}) \) for any \( i \) by (a).   Then  there exists \( Q_{i_p} \) such that \( p \in G_{i_p} (Q_{i_p}) \).  It follows that  \( \{Q_{i_p}\} \) determines an atlas of \( M \):  If \( x \in Q_{i_p} \cap Q_{i_q} \), then \( Q_{i_x} \cap Q_{i_p} \cap Q_{i_q}  \) is smoothly embedded.   It follows that  the overlap maps are of class \( B^r \) since all of the embeddings are of class \( B^r \).  Therefore, by Lemma \ref{lem:oncemore},  \( M \) is a smoothly embedded submanifold of \( U \) of class \( C^{r-1+Lip} \).  
\end{proof}

 If we replace condition (b) with \( \p A  \) is nonzero and represents a smooth \( (k-1) \)-submanifold with smooth boundary, then   \( A \) represents a
 \( k \)-submanifold with smooth boundary. If \( A = \sum A_j \), and
each \( A_j \) represents an embedded submanifold with boundary, then  \( A \) represents a piecewise smooth immersed  submanifold.

\begin{examples} \mbox{} 
	\begin{itemize}
		 \item The differential chain representative of an oriented \( 2 \)-sphere in \( \R^3 \) can be
written as the sum of representatives of any two hemispheres, for example, or as the sum
of representatives of puzzle pieces as seen in wikipedia's symbol. 
      
 		 \item   In classical topology a torus can be obtained by gluing  a cylinder
  to the sphere with two small disks removed.  Instead, we subtract representatives of the small disks from a representative of the sphere, and add a representative of the cylinder to get a representative of the torus.   Similarly, any orientable surface may be represented by a chain by adding  representatives of ``handles'' to a
sphere representative. As with the Cantor set or Sierpinski triangle below, algebraic sums  and pushforward can replace 
cutting and pasting.
 		 \item The quadrifolium and Boys surface can be represented by algebraic chains. 
 \begin{figure}[ht] \centering \includegraphics[height=1.25in]{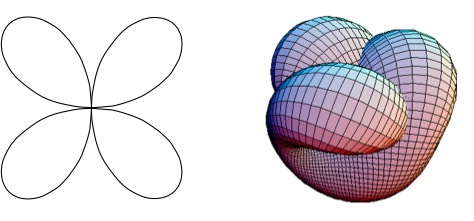} 
\caption{ } \label{fig:Quadrifolium} 
\end{figure}
\item   Whitney stratified sets can be represented by algebraic chains since stratified sets can be triangulated \cite{goresky} (see Figure \ref{fig:WhitneyUmbrella}). 

				\begin{figure}[ht]
					\centering
						\includegraphics[height=1.5in]{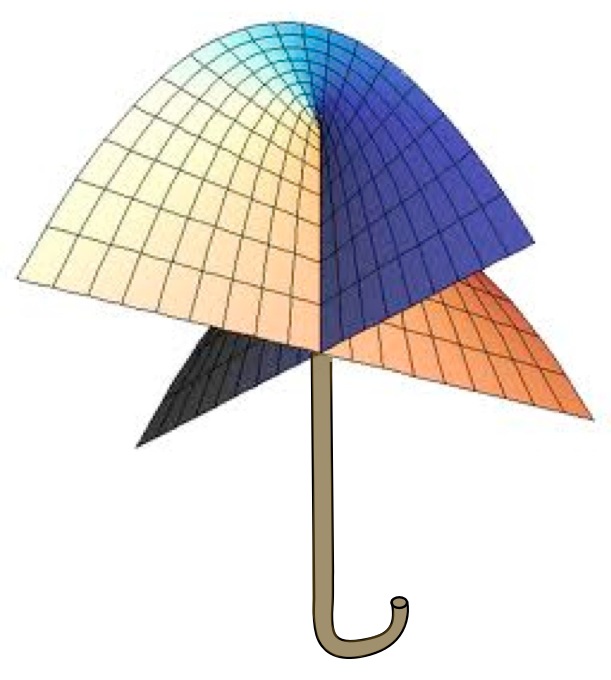}
					\caption{Whitney umbrella}
					\label{fig:WhitneyUmbrella}
				\end{figure}
     
 				 \end{itemize} 
				\end{examples}
\subsection{Dipole chains} \label{sub:higher_order_algebraic_chains} \( P_v
\widetilde{\s} \) is a \emph{dipole} \emph{\( k \)-cell}, and \( \sum F_{*i} P_{v_i} \widetilde{s_i}
\) is a\emph{ dipole algebraic \( k \)-chain}, or, more simply, a \emph{dipole \( k \)-chain}. This idea
may be extended to define \( k \)-cells of order \( r \), but we do not do so here. Dipole
\( k \)-chains are useful for representing soap films, Moebius strips, and soap bubbles.   (See
\cite{soap} and \cite{plateau10} for more details).

\subsection{Representatives of fractals} \label{sub:self_similar_fractals} The interior
of the Sierpinski triangle \( T \) can be represented by a \( 2 \)-chain \( \tilde{T} =
\lim_{k \to \i} (4/3)^k \widetilde{S_{k_i}} \) in the \( B^1 \) norm where the \( S_{k_i} \) are
 oriented simplices filling up the interior of \( T \) in the standard
construction\footnote{If we had started with polyhedral chains instead of Dirac chains,
convergence would be in the \( B^0 \) norm.}. The support of its boundary \( \p \tilde{T}
\) is \( T \). We may integrate smooth forms and apply the integral theorems to calculate
flux, etc. Other applications to fractals may be found in \cite{earlyhodge, continuity}.
 \begin{figure}[ht]
 	\centering
 		\includegraphics[height=1in]{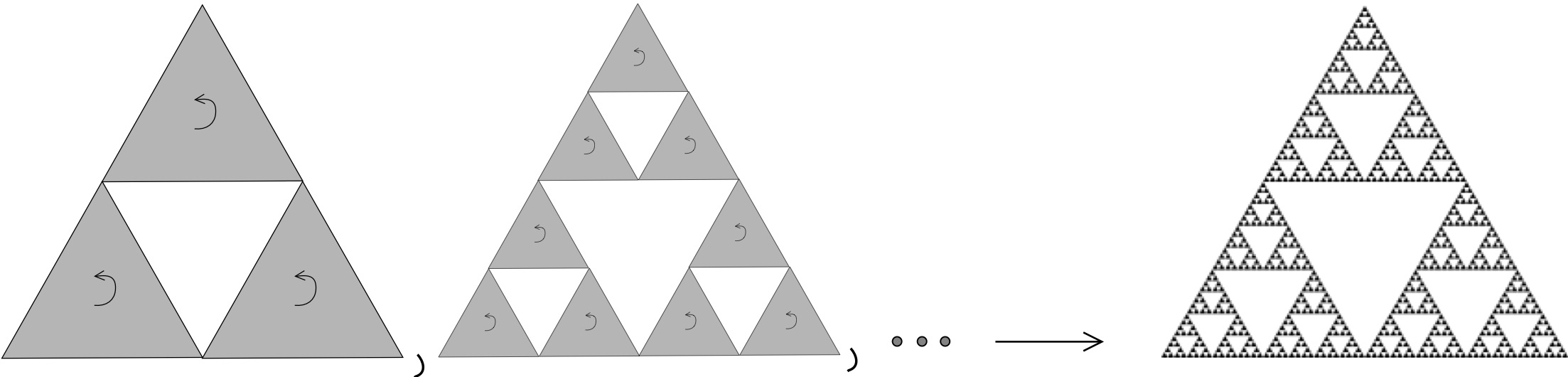}
 	\caption{Sierpinski triangle}
 	\label{fig:Sierpinski}
 \end{figure}  

  We revisit the middle third Cantor set \( \G \) from \S\ref{sub:poly}.
Recall the chain representatives of approximating open sets obtained by removing middle
thirds forms a Cauchy sequence in \( \hB_1^1 \).  These are algebraic \( 1 \)-chains. The limit \( \widetilde{\G} \)
is a differential \( 1 \)-chain that represents \( \G \). Its boundary \( \p \widetilde{\G}
\) is well-defined and is supported in the classical middle third Cantor set. We may
therefore integrate differential forms over \( \G \) and state the fundamental theorem of
calculus where \( \G \) is a domain and \( f\in \B_1^1 \): \[ \cint_{\p \G} f =
\cint_{\G} df. \]

\section{Vector fields and the primitive operators} \label{sub:vector_fields_ops}
\subsection{The space \( {\cal V}^r(U)  \)}
\label{vectorfields}

Let \( U \) be open in \( \R^n \). Recall \( {\cal M}^r(U,\R^n) \) is
the seminormed  space of \( C^r \) differentiable maps \( F:U \to \R^n \) whose coordinate functions  have
directional derivatives in \( \B_0^{r-1}(U) \) for all \( r \ge 1 \).  
 
Let \( {\cal V}^r(U)  \)  be the subspace of elements of \( {\cal M}^r(U,\R^n)  \) whose images are bounded subsets of \( \R^n \). For \( V \in {\cal V}^r(U) \), let \( V_{e_i} \) be the coordinate form determined by \( V_{e_i}(p) := \<V(p),e_i\>   \).  Define \( \|V\|_{B^{0,U}} := \sup\{\|V(p)\|: p \in U\} \) and \( \|V\|_{B^{r,U}} = \max_{i,j} \{ \|V\|_{B^{0,U}}, \|L_{e_j} V_{e_i}\|_{B^{r-1,U}} \} \) for each \( r \ge 1 \).  

\begin{prop}\label{prop:Vnorm}
 Let \( U \) be open in \( \R^n \).  Then \( \|\cdot\|_{B^{r,U}} \) is a  norm on the vector space \( {\cal V}^r(U)  \) for each  \( r \ge 1 \). 
\end{prop}

\begin{proof} 
	Suppose \( X, Y \in {\cal V}^r(U) \).  Then \( \|X + Y\|_{B^{0,U}} = \sup\{\|(X + Y)(p)\|: p \in U\} \le \sup\{\|X(p)\|: p \in U\} + \sup\{\|Y(p)\|: p \in U\} = \|X\|_{B^{0,U}} + \|Y\|_{B^{0,U}} \).   
	Since \( \|\cdot\|_{B^{r,U}}  \) is a seminorm on \( {\cal M}^r(U, \R^n) \), the triangle inequality is satisfied:     \( \|X + Y\|_{B^{r,U}} = \max\{ \|X+Y\|_{B^{0,U}}, \|L_{e_j} (X_i + Y_i)\|_{B^{r-1,U}} \}  \le \max\{ \|X\|_{B^{0,U}}, \|L_{e_j} X_i\|_{B^{r-1,U}} \}  + \max\{ \|Y\|_{B^{0,U}}, \|L_{e_j} Y_i\|_{B^{r-1,U}} \} = \|X\|_{B^{r,U}} + \|Y\|_{B^{r,U}}  \).  Homogeneity is similar.  Suppose \( X \ne 0 \).  Then there exists \( p \in U \) with \( X(p) \ne 0 \).  Hence \( \|X\|_{B^{0,U}} > 0 \) and therefore \( \|X\|_{B^{r,U}} > 0 \). 
\end{proof}

The norms are increasing with \( r \).  It follows that the projective limit \( {\cal V}(U) = {\cal V}^\i(U) := \varprojlim {\cal V}^r(U) \) endowed with the projective limit topology is a well-defined Fr\'echet space.  

\begin{lem}\label{prop:vectornorm} Let \( u \in \R^n \) and \( V \in  {\cal V}^r(U)   \) a vector field.    Then each coordinate form \( V_{e_i} \) is an element of \( \B_0^r(U) \) with \( \|V_{e_i}\|_{B^{r,U}} \le \|V\|_{B^{r,U}} \) for all \( 0 \le r <\i \). 
\end{lem} 

\begin{proof} 
The proof is immediate from the definitions.   
\end{proof}

  In this section we extend the primitive operators  to \( E_V, E_V^\dagger, \mbox{ and } P_V \) where
\( V \in {\cal V}^r(U) \) is    a vector field  and \( U \) is an open subset of \( \R^n \).

\subsection{Extrusion} \label{sub:extrusion}  

Let \( V \in {\cal V}^r(U) \). Define the bilinear map \( E_V \) on Dirac chains by
\begin{align*}
E_V:  \A_k(U) &\to \A_{k+1}(U) \\  (p;\a)  &\mapsto (p; V(p) \wedge \a)
\end{align*}.   It follows that  \( i_V \o (p;\a) = \o E_V (p;\a) \) where \( i_V \) is classical interior product.

\begin{lem}\label{lem:extX} \( E_{fv} = m_f E_v \) for all \( v \in \R^n \) and
functions \( f \in \B_0^r \). 
\end{lem}

\begin{proof} This follows directly from the definitions: \[ E_{fv} (p;\a) = (p; f(p)v \wedge \a ) = m_f(p; v \wedge \a) = m_f E_v (p;\a). \]
\end{proof} 

The next two results follow by applying Lemma \ref{lem:extX} to Theorem \ref{pro:IIAV} and Corollary \ref{cor:extV}.  We define the subspace \( H_k(U)  \subset \pB_k(U) \) much as we defined \( H_k(\R^n) \) on page \pageref{thm:continuousoperators}.

\begin{thm}\label{pro:IIAV}  Let \( U \) be open in \( \R^n \).  The linear map \( E_V: \A_k \to \A_{k+1} \) extends to  continuous linear maps \( E_V: \hB_k^r(U) \to \hB_{k+1}^r(U) \) with  \[ \|E_V(J) \|_{B^{{r,U}}} \le n^2r\|V\|_{B^{r,U}}\|J\|_{B^{{r,U}}} \mbox{ for all }  J \in  \hB_k^r(U) \mbox{ and }\]
  \[ E_V(H_k(U)) \subset H_{k+1}(U) \].
\end{thm}

\begin{proof} We establish the inequality.   Since \( V \in {\cal V}^r(U) \) we know \( V = \sum V_{e_i} e_i \) where \(
V_{e_i} \in \B_0^r(U) \). By  Lemma \ref{lem:extX} \( E_V = \sum_{i=1}^n E_{V_{e_i} e_i} = \sum_{i=1}^n
m_{V_{e_i}} E_{e_i} \).

By Theorem \ref{thm:continfuncII} \( \|m_{V_{e_i}} A\|_{B^{r,U}} \le nr \|V_{e_i}\|_{B^{r,U}} \|A\|_{B^{r,U}} \). Therefore, using Corollary \ref{cor:ext} and Lemma \ref{prop:vectornorm}, 
\begin{align*} \|E_V A\|_{B^{r,U}} \le \sum_{i=1}^{n}
\|m_{V_{e_i}} E_{e_i} A\|_{B^{r,U}} \le n r \sum_{i=1}^{n} \|V_{e_i}\|_{B^{r,U}} \|E_{e_i} A\|_{B^{r,U}}&\le nr \sum_{i=1}^{n} \|V_{e_i}\|_{B^{r,U}} \|A\|_{B^{r,U}} \\&\le n^2 r
\|V\|_{B^{r,U}}\|A\|_{B^{r,U}}. 
\end{align*} Define \( E_V(J) := \lim_{i \to \i} E_V(A_i) \) for \( J \in \pB(U) \).  

The second part is immediate, as it was in Lemma \ref{lem:IIA}.
\end{proof}                      

  Let \( i_V \o := \o E_V \) the dual operator on forms.

\begin{cor}[Change of dimension I]\label{cor:extV}  Let \( V \in {\cal V}^r(U) \). Then \( E_V \in
{\cal L}(\pB) \) and \( i_V \in {\cal L}( \B) \) are continuous
graded operators. Furthermore, 
\begin{equation*}  \cint_{E_V J} \o = \cint_J i_V \o
\end{equation*} for all matching pairs, i.e.,  \( J \in \hB_k^r(U) \) and \( \o \in \B_{k+1}^r(U) \),
or \( J \in \pB_k(U) \) and \( \o \in \B_{k+1}(U) \).  Furthermore,  the bigraded multilinear maps  \( E \in {\cal L}( {\cal V} \times \pB, \pB) \), given by \( (V, J) \mapsto  E_V(J) \),  and      \( \phi_E \in {\cal L} ( {\cal V} \times \pB \times \B, \R)  \),   given by  \( (V, J, \o) \mapsto \cint_{E_V  J} \o  \), are both separately continuous.
\end{cor}

 \begin{proof} 
 Since \( \o(E_V (p;\a)) = i_V \o (p;\a)) \) the integral relation holds for  Dirac chains, and thus to pairs of all chains and forms of matching class by continuity of the integral. The extension to \( r = \i \) and to \( \pB \) follow as they did for \ref{cor:ext}.  Interior product \( i_V \) is continuous since it is dual to the continuous operator \( E_V \).
\end{proof}

\subsection{Retraction} \label{sub:retraction} For a vector field \( V \in {\cal V}^r(U)
\), define the linear map \( E_V^\dagger \) on Dirac chains
\begin{align*}
E_V^\dagger:\A_{k+1}(U) &\to \A_k(U) \\ (p; v_1 \wedge \cdots \wedge v_{k+1})
&\mapsto \sum_{i=1}^{k+1} (-1)^{i+1} \<V(p),v_i\> (p; v_1 \wedge \cdots \wedge
\widehat{v_i} \wedge \cdots \wedge v_{k+1}) 
\end{align*} 
or, in our more compact notation, \[ E_V^\dagger (p; \a) = \sum_{i=1}^{k+1} (-1)^{i+1}\<V(p),v_i\> (p; \hat{\a}_i). \]

\begin{lem}\label{lem:extXd}   If \( v
\in V \) and  \( f \in \B_0^r(U), r \ge 1\), then \(
E_{fv}^\dagger = m_f E_v^\dagger \). 
\end{lem}

\begin{proof} This follows directly from the definitions:
\begin{align*} E_{fv}^\dagger
(p; \a) = \sum_{i=1}^{k+1} (-1)^{i+1} \<f(p)v,v_i\> (p;
\widehat{\a_i})  = f(p)
\sum_{i=1}^{k+1} (-1)^{i+1} \<v,v_i\> (p; \hat{\a_i})  = m_f E_v^\dagger (p;\a). 
\end{align*} 
\end{proof} 

\begin{thm}\label{thm:retX} The  linear map \( E_V^\dagger \) extends to   continuous linear maps \( E_V^\dagger:\pB_k(U)  \to \pB_{k-1}(U) \)    with \[ \|E_V^\dagger(J) \|_{B^{{r,U}}}
\le kn^2 r \|V\|_{B^{r+1,U}}\|J\|_{B^{{r,U}}}, \] and \( E_V^\dagger:  \pB_k(U) \to
\pB_{k-1}(U) \).     
\end{thm} 

The proof is similar to that of Corollary \ref{cor:daggerE}.

\begin{lem}\label{lem:daggerE} Let \( V \in {\cal V}^r(U) \) be a vector field. Then \(
V^\flat \wedge (\cdot): \B_k^r \to \B_{k+1}^r \) is continuous and \( (V^\flat
\wedge \o) (p;\a) = \o E_V^\dagger (p;\a) \). 
\end{lem}

\begin{proof} If \( V \in {\cal V}^r(U) \), then   \( v^\flat \wedge \o \in \B_{k+1}^r(U) \) for all \( \o \in \B_k^r(U) \) and \( V^\flat \in \B_1^r(U) \). Finally,   
\begin{align} \o E_V^\dagger (p;\a) = \sum_{i=1}^k (-1)^{i+1} \<V(p),v_i\> \o
(p; \hat{\a_i}) = (V^\flat \wedge \o)(p;\a)  
\end{align}   by standard formula of wedge product of forms (see \cite{federer}, for example.)
\end{proof}

\begin{cor}[Change of dimension II]\label{thm:retXint} Let \( V \in {\cal V}^r(U) \).  Then \( E_V^\dagger:
\hB_{k+1}^r(U) \to \hB_k^r(U) \) 	and its dual operator \( V^\flat \wedge :  \B_k^r(U) \to \B_{k+1}^r(U) \) are continuous.  Furthermore, \[ \cint_{E_V^\dagger J} \o =
\cint_J V^\flat \wedge \o \] for all \( J \in \hB_{k+1}^r(U) \), \( \o \in {\cal
B}_k^r(U) \), and \(r \ge 1 \) or  \( J \in\pB_{k+1}(U) \) and \( \o \in \B_k(U) \). Furthermore,  the bigraded multilinear maps  \( E^\dagger \in {\cal L}( {\cal V} \times \pB, \pB) \), given by \( (V, J) \mapsto  E_V^\dagger(J) \),  and      \( \phi_{E^\dagger} \in {\cal L} ( {\cal V} \times \pB \times \B, \R)  \),   given by  \( (V, J, \o) \mapsto \cint_{E_V^\dagger  J} \o  \), are both separately continuous.
\end{cor} 

The proof is similar to that of Corollary \ref{cor:daggerE}.
 
\subsection{Prederivative} \label{sub:prederivative} Define the bilinear map \( P: {\cal V}^r(U) \times \hB_k^r(U) \to \hat{\cal B}_k^{r+1}(U) \) by \( P(V,J) = P_V(J) := \p E_V(J) + E_V \p(J) \). Since both \( E_V \) and \( \p \) are
continuous, we know that \( P_V \) is continuous.

\begin{figure}[ht]
	\centering
		\includegraphics[height=1.25in]{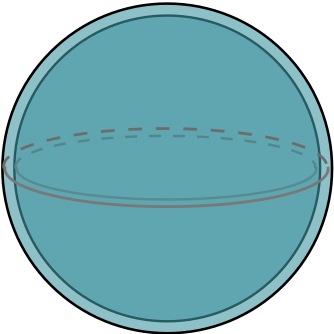}
	\caption{Prederivative of the sphere with respect to the radial vector field}
	\label{fig:Dipolesphere}
\end{figure}

Its dual operator \( L_V \) is the classically defined Lie derivative    since \( L_V = i_V d + d i_V \) uniquely determines it.    
  
\begin{thm}\label{thm:predfirst} The bilinear map determined by 
\begin{align*} P: {\cal V}^r(U) \times
\hB_k^r(U) &\to \hB_k^{r+1}(U) \\ (V, J) &\mapsto P_V(J) 
\end{align*} is
well-defined, and \( \|P(V, J)\|_{B^{r+1,U}} \le  2 kn^3 r \|V\|_{B^{r+1,U}} \|A\|_{B^{r,U}} \) for all \( r \ge 1 \). Therefore, \( P \) extends to  continuous
bilinear maps   \( P:{\cal V}^\i(U) \times  \pB_k(U) \to \pB_k(U) \) and \( P: {\cal V}^\i(U) \times \pB_k(U) \to \pB_k(U) \).   
\end{thm} 
  
\begin{proof} By Theorems \ref{thm:bod} and \ref{pro:IIAV}
\begin{align*} 
      \|P_V J\|_{B^{r+1,U}} \le  \|\p E_V J\|_{B^{r+1,U}} +  \| E_V \p J\|_{B^{r+1,U}} &\le kn \|E_V J\|_{B^{r,U}}+ n^2 r\|V\|_{B^{r+1,U}} \|\p J\|_{B^{r+1,U}} \\&\le  k n^3r\|V\|_{B^{r,U}} \|  J\|_{B^{r,U}}+ kn n^2 r \|V\|_{B^{r+1,U}} \|  J\|_{B^{r,U}}  \\&\le 2 kn^3 r \|V\|_{B^{r+1,U}}\|J\|_{B^{r,U}}.  
   \end{align*}               
\end{proof}

\begin{cor}\label{thm:preV} [Change of order I] Let \( V \in {\cal V}^r(U) \). Then \( P_V \in
{\cal L}(\pB(U)) \) and \( L_V \in {\cal L} (\B(U)) \) are continuous
bigraded operators satisfying
\begin{equation}\label{primitiveprederivativeV} \cint_{P_V J} \o = \cint_J L_V \o
\end{equation} for all matching pairs, i.e., for all   \( J \in \hB_k^{r-1}(U) \) and \( \o \in \B_k^r(U) \) or \( J \in \pB_k(U) \) and \( \o \in \B_k(U) \). 
\end{cor}

\begin{thm}\label{thm:preder} If \( V \in {\cal V}^r(U) \) is a vector field and \(
(p;\a) \) is a \( k \)-element with \( p \in U \), then \[ P_V (p;\a) = \lim_{t \to 0}
(\phi_t(p); \phi_{t*} \a/t) - (p;\a/t) \] where \( \phi_t \) is the time-\( t \) map of
the flow of \( V \). 
\end{thm}

\begin{proof} The right hand side equals \( \dot{\phi}_* (p;\a) \) where \( \dot{\phi} := \lim_{t \to 0} \frac{\phi_{t}- I}{t} \) which is a well defined
element of \( \hB_k^r \) since \( \phi_t \in {\cal M}^{r+1}(U, U) \).  (\( \phi_t \) is always defined for sufficiently small \( t \) and \( p\in U \).) The left hand side is also an element of \( \hB_k^r \). We show
they are equal as currents via evaluation. Since both are differential chains, and chains are naturally included in currents,  this will imply they are equal as chains.   Let \( \o \in \B_k^r \). By Theorem \ref{thm:predfirst}  the definition of   Lie derivative of forms   \( \cint_{P_V (p;\a)} \o  =  \cint_{(p;\a)} L_V \o= \cint_{ (p;\a)} \dot{\phi}^*\o = \cint_{\dot{\phi}_* (p;\a)} \o\) and we
are finished. 
\end{proof} 

Prederivative gives us a way to ``geometrically differentiate'' a differential chain in the infinitesimal directions determined by a vector field,  even when the support of the differential chain is highly nonsmooth, and without using any functions or forms.

 The next result follows directly from the definitions.

\begin{thm}[Commutation   Relations] \label{lem:x} If \( V_1,V_2 \in {\cal V}^r(U) \), then 	\(  [P_{V_1}, P_{V_2}]  =
P_{[{V_1},{V_2}]}, \quad \) \( \{E_{V_1}^\dagger, E_{V_2}^\dagger\}  = \{E_{V_1}, E_{V_2}\} =0, \quad \) \( 	\{E_{V_1},
	E_{V_2}^\dagger\}   = \<{V_1}(p),{V_2}(p)\> I,  \quad\) \(  [E_{V_2}^\dagger, P_{V_1}]  = E_{[{V_1},{V_2}]}^\dagger, \quad \) and \( \quad [E_{V_2}, P_{V_1}]  =
	E_{[{V_1},{V_2}]} \).
\end{thm}

 Let \( {\cal L}(\pB)(U) \) be the algebra of operators on \( \pB(U) \).   This includes
  \( \{  m_f, F_*,\p, \perp \} \), as well as \( E_V, E_V^\dagger, P_V,
 \) for vector fields \( V \in {\cal V}^r(U) \).

\subsection{Naturality of the operators} \label{ssub:naturality_of_the_operators}  
\begin{thm}\label{pro:mF} [Naturality of the operators] Suppose \( V \in {\cal V}^r(U_1) \)
is a vector field and \( F \in {\cal M}^r(U_1, U_2) \) is a map. Then 

\begin{enumerate} 
	\item    \( F_* m_{f \circ F} =
m_f F_* \) for all \( f \in \B_0^r(U_2) \).   \\If \( F \) is a diffeomorphism  onto its image, then 
  \item \( F_* E_V = E_{FV} F_* \); 
	\item \( F_* E_V^\dagger = E_{FV}^\dagger F_* \) if \( F \) preserves the metric; 
	\item \( F_* P_V = P_{FV} F_*\); 
\end{enumerate} 
\end{thm}

\begin{proof} (a): \(F_* m_{f \circ F} (p;\a) = (F(p);
f(F(p)) F_*\a) = m_f F_* (p;\a) \).  (b): is omitted since it is much like (c).  (c): Since \( \<F_*V, F_* v_i\> = \<V,v_i\> \) we have 
\begin{align*}
 	F_*(E_V^\dagger((p;\a))) = \sum_{i=1}^k(-1)^{i+1} \<V,v_i\> (F(p); F_* \hat{\a}_i)&= \sum_{i=1}^k (-1)^{i+1}
	\<F_*V, F_* v_i\> (F(p); F_* \hat{\a}_i) \\&=E_{FV}^\dagger (F(p); F_* \a) \\&=    E_{FV}^\dagger F_* (p;\a)
\end{align*} 
 (d): Observe that if \( V_t \) is the flow of \( V \), then \( F V_t F^{-1} \) is
the flow of \( FV \). It follows that
\begin{align*} F_* P_v (p;\a) = F_* \lim_{t \to 0}
(V_{t}(p); V_{t*}\a/t) - (p;\a/t) &= \lim_{t \to 0} (F(V_{t}(p)); F_*V_{t*}\a/t) -
(F(p);F_*\a/t) \\&= \lim_{t \to 0} ((FV_{t}F^{-1})F(p); F_*\a/t) - (F(p);F_*\a/t) \\&=
P_{FV}(F(p); F_*\a) \\&= P_{FV}F_* (p;\a). 
\end{align*}
\end{proof}

\begin{cor}\label{cor:FBD} \( [F_*, \p]   = 0 \) for maps \( F \in
{\cal M}(U_1, U_2) \). 
\end{cor}     

\begin{proof} This follows from Theorem \ref{pro:mF} and the definition of \( \p \).
\end{proof} 

Similarly, our results regarding the operators \( m_f \in {\cal L}(\pB(\R^n)) \) extend to similar operators  \( m_f \in {\cal L}(\pB(U)) \).  

\section{Cartesian wedge product} \label{sub:tensor_product} 

\subsection{Definition of \( \times \)}
\label{Cartesian2}

Suppose \( U_1  \subseteq \R^n \) and \( U_2 \subseteq \R^m \) are open.  Let \( \iota_1:  U_1 \to U_1 \times U_2  \) and \( \iota_2: U_2 \to U_1 \times U_2   \) be the inclusions \(  \iota_1(p) = (p,0)  \) and \(  \iota_2(q) = (0,q) \).   Let \( \pi_1:U_1 \otimes U_2 \to U_1 \) and \( \pi_2:U_1 \otimes U_2 \to U_2 \)  be the projections \( \pi_i(p_1,p_2) = p_i \), \( i = 1, 2 \).
Let \( (p;\a) \in \A_k(U_1) \) and \( (q;\b) \in \A_\ell(U_2) \). Define \( \times: \A_k(U_1)
\times \A_\ell(U_2) \to \A_{k+\ell}(U_1 \times U_2) \) by \( \times((p; \a), (q;
\b)) := ((p,q); \iota_{1*}\a \wedge \iota_{2_*}\b) \) where \( (p;\a) \) and \( (q; \b) \) are \( k \)- and
\( \ell \)-elements, respectively, and extend bilinearly. We call \( P \times Q = \times(P,Q)
\) the \emph{Cartesian wedge product\footnote{By the universal property of tensor product,  \( \times \) factors through a continuous linear map  \emph{cross product}  \(   \tilde{\times}: \pB_j(U_1) \otimes \pB_k(U_2) \to  \pB_{j+k}(U_1 \times U_2)\).  This is closely related to the classical definition of \emph{cross product} on simplicial chains (\cite{hatcher}, p. 278) }} of \( P \) and \( Q \). Cartesian wedge product of
Dirac chains is associative since wedge product is associative, but it is not graded
commutative since Cartesian product is not graded commutative.  We next show that Cartesian wedge product is continuous.

\begin{lem} \label{products} If \( D^i \) is an \( i \)-difference \( k \)-chain in \( U_1 \) and \( E^j \) is a \( j \)-difference
\( \ell \)-chain in \( U_2 \), then \( D^i \times E^j \) is an \( (i+j) \)-difference \( (k+\ell) \)-chain in \( U_1 \times U_2 \) with \[ \|D^i \times E^j\|_{B^{i+j, U_1 \times U_2}} \le |D^i |_{B^{i,U_1}} |E^j|_{B^{j,U_2}}. \]      
\end{lem}
                                                             
\begin{proof} This follows since \( \D_{\s^i} (p;\a) \times \D_{\t^j}(q;\b) =
\D_{\s^i,\t^j}  ((p,q); \a \wedge \b) \) and \( |\D_{\s^i,\t^j}  ((p,q);\a \wedge
\b)|_{i+j, U_1 \times U_2} \le \|\s\|\|\t\|\|\a\|\|\b\| = |\D_{\s^i} (p;\a)|_{B^{i,U_1}} |\D_{\t^j}
(q;\b)|_{B^{j,U_2}} \).  Furthermore, if  \( \D_{\s^i} (p;\a) \) is in \( U_1 \) and \( \D_{\t^j}(q;\b) \) is in \( U_2 \), then  the paths \(  \ell(p,\s) \) (see \S\ref{sub:definitions})  are contained in \( U_1 \)  and the paths  \( \ell(q; \t) \) are contained in \( U_2 \).  It follows from the definition of Cartesian product that  the paths \( \ell((p,q), \s \circ \t) \)   are contained in \( U_1 \times U_2 \).   Therefore,   \( D^i \times E^j \) is an \( (i+j) \)-difference \( (k+\ell) \)-chain in \( U_1 \times U_2 \).
\end{proof}

\begin{prop} \label{Cartesian} Suppose \( P \in \A_k(U_1) \)
and \( Q \in \A_{\ell}(U_2) \) are Dirac chains where \( U_1 \subseteq \R^n, U_2 \subseteq
\R^m \) are open. Then \( P \times Q \in \A_{k +\ell}(U_1 \times U_2) \) with \[ \|P
\times Q\|_{B^{r+s, U_1 \times U_2}} \le \|P\|_{B^{r,U_1}}\|Q\|_{B^{s,U_2}}. \] Furthermore, \[ \|P \times
\widetilde{(a,b)}\|_{B^{r,U_1 \times \R}} \le  |b-a|\|P\|_{B^{r,U_1}} \] where \( \widetilde{(a,b)} \) is a \( 1 \)-chain
representing the interval \( (a,b) \). 
\end{prop}

\begin{proof}  The proof of the first inequality proceeds by induction. The result holds for \( r =s = 0 \). Assume it
holds for \( r \) and \( s \). We show it holds for \( r+1 \) and \( s+1 \).

 Choose \( \e > 0 \) and let \( \e' = \e/(\|P\|_{B^{r,U_1}} + \|Q\|_{B^{s,U_2}} +1) < 1 \). There exist
decompositions \( P = \sum_{i=0}^r D^i \) and \( Q = \sum_{j=0}^s E^j \) such that \( \|P\|_{B^r}
> \sum_{i=0}^r |D^i|_{B^{i,U_1}} -\e \) and \( \|Q\|_{B^{s,U_1}} > \sum_{j=0}^s |E^j |_{B^{j,U_1}} - \e \).
Since \( P \times Q = \sum_{i=0}^r\sum_{j=0}^s D^i \times E^j \)
 and by Lemma \ref{products}
\begin{align*} \|P\times Q\|_{B^{r+s,U_1\times U_2}} &\le
\sum_{i=0}^r\sum_{j=0}^s \|D^i \times E^j\|_{B^{i+j, U_1 \times U_2}} \\ &\le \sum_{i=0}^r\sum_{j=0}^s
|D^i |_{B^{i,U_1}} |E^j |_{B^{j,U_1}} \\&\le (\|P\|_{B^{r,U_1}} +\e')(\|Q\|_{B^{s,U_2}} +\e').
\end{align*} Since this holds for all \( \e > 0 \), the result
follows.

 The second inequality is similar, except we use \( \| D^i \times \widetilde{(a,b)}
\|_{B^{i,U_1 \times \R}} \le |b-a| |D^i |_{B^{i, U_1 \times \R}}. \) 
\end{proof}  

Let \( J \in \hB_k^r (U_1) \) and \( K \in \hB_\ell^s(U_2) \). Choose
Dirac chains \( P_i \to J \) converging in \( \hB_k^r (U_1) \), and \( Q_i \to K \) converging in \( \hB_\ell^s(U_2) \).
Proposition \ref{Cartesian} implies \( \{P_i \times Q_i\} \) is Cauchy. Define \[ J \times K
:= \lim_{i \to \i} P_i \times Q_i. \]

\begin{thm}\label{cart} Cartesian wedge product \( \times: \hB_k^r(U_1) \times
\hB_\ell^s(U_2) \to \hB_{k+\ell}^{r+s}(U_1 \times U_2) \) is associative,
bilinear and continuous for all open sets \( U_1 \subseteq \R^n, U_2 \subseteq \R^m \) and
satisfies
 \begin{enumerate} 
		\item \( \|J\times K\|_{B^{r+s, U_1 \times U_2}} \le \|J\|_{B^{r,U_1}} \|K\|_{B^{s,U_2}} \)
		\item \begin{align*} \p(J \times K) = \begin{cases} ( \partial J) \times K + (-1)^k J
\times ( \partial K), &k > 0, \ell >0 \\ ( \partial J) \times K, &k > 0, \ell = 0\\ J
\times ( \partial K), &k = 0, \ell > 0 \end{cases} 
\end{align*} 
		\item \( J \times K = 0
\) implies \( J = 0 \) or \( K = 0 \);  
		\item \( (p; \s \otimes \a) \times (q; \t \otimes \b) = ((p,q); \s
\circ \t \otimes \iota_{1*}\a \wedge \iota_{2*}\b) \);
		\item  \( (\pi_{1*}\o \wedge \pi_{2*}\eta)(  J \times   K) = \o(J)\eta(K) \)  for \( \o \in \B_k^r(U_1), \eta \in \B_\ell^s(U_2) \);
		\item \( \supp(J \times K) = \supp(J) \times \supp(K) \).
\end{enumerate}
\end{thm}

\begin{proof} (a) is a consequence of Proposition \ref{Cartesian}. 
(b):  This follows from   Corollary \ref{thm:bod}:
 \( \p ((p,q); \a \wedge \b) = (\p ((p,q);\a)) \cdot ((p,q); \b) +
(-1)^k ((p,q);\a) \cdot (\p ((p,q);\b))  =(\p (p;\a)) \times (q; \b) +
(-1)^k (p;\a) \times (\p (q;\b)) \), for all \( (p;\a) \in \A_k(U_1), \text{ and } (p;\b)\in {\cal P}_\ell(U_2) \). The boundary relations for Dirac chains follow by linearity. We know  \(\times \) is continuous by (a) and \( \p \) is continuous by Theorem \ref{thm:bod}, and therefore the relations extend to \( J \in
\pB_k(U_1) \) and \( K \in \pB_\ell(U_2) \).

 (c): Suppose \( J \times K = 0 \), \( J \ne 0 \) and \( K \ne 0 \). Let \( \{e_i\} \)
be an orthonormal basis of \( \R^{n+m} \), respecting the Cartesian wedge product.
Suppose \( e_I \) is a \( k \)-vector in \( \L_k(\R^n) \) and \( e_L \) is an \( \ell
\)-vector in \( \L_\ell(\R^m) \). Then
\begin{equation*} \cint_Q \left( \cint_P f de_I
\right) g de_L = \cint_{P \times Q} h de_I de_L 
\end{equation*} where \( h(x,y) = f(x)
g(y) \), \( P \in \A_k(U_1), Q \in \A_\ell(U_2) \). The proof follows easily by
writing \( P = \sum_{i=1}^r (p_i; \a_i) \) and \( Q = \sum_{j=1}^s(q_i; \b_i) \) and
expanding. By continuity of Cartesian wedge product \( \times \) and the integral, we
deduce
\begin{equation*} \cint_K \left( \cint_J f de_I \right) g de_L = \cint_{J \times K}
h de_I de_L 
\end{equation*} where \( h(x,y) = f(x) g(y) \), \( J \in \pB_k(U_1), K
\in \pB_\ell(U_2) \). Since \( J\ne 0 \) and \( K \ne 0 \), there exist \( f \in
\B_0^(U_1), g \in \B_0(U_2) \) such that \( \cint_J fde_I \ne 0 \), \( \cint_K
gde_L \ne 0 \). Therefore, \( \cint_K \left( \cint_J f de_I \right) g de_L = \cint_J f
de_I \cint_K g de_L \ne 0 \), which implies \( \cint_{J \times K} h de_I de_L \ne 0 \),
contradicting the assumption that \( J \times K = 0 \).

 (d):  We first show that \( (p; \s \otimes \a) \times (q; \b) = ((p;q); \s \otimes
\a \wedge \b) \). The proof is by induction on the order \( j \) of \( \s \). This holds
if \( j = 0 \), by definition of \( \times \). Assume it holds for order \( j-1 \).
Suppose \( \s \) has order \( j \). Let \( \s = u \circ \s' \).   By Proposition \ref{Cartesian}
\begin{align*} (p; \s \otimes \a) \times (q; \b) &= (\lim_{t \to 0} (p +tu; \s' \circ
\a/t) - (p; \s' \circ\a/t)) \times (q;\b) \\&= \lim_{t \to 0} (p +tu; \s' \circ \a/t)\times
(q;\b) - (p; \s' \circ \a/t) \times (q;\b) \\&= \lim_{t \to 0} ((p+tu,q);\s' \circ \a
\wedge \b/t) - ((p,q); \s' \circ \a \wedge \b/t) \\&= P_u((p,q); \s' \circ \a \wedge \b) =
((p,q); \s' \circ u \otimes \a \wedge \b) = ((p,q); \s \circ \a\wedge \b). 
\end{align*} 
In a similar way, one can show \( (p; \s \otimes \a) \times (q; \t \otimes \b) = ((p;q); \s \circ \t \otimes \a \wedge \b) \) using induction on the order \( j \) of \( \s \). 

(e):  This follows from the definition of \( \times \) for simple elements:
\begin{align*} (\pi^{1*}\o \wedge \pi^{2*}\eta)(  (p;\a) \times   (q;\b)) &=  (\pi^{1*}\o \wedge \pi^{2*}\eta)((p,q); \iota_{1*}\a \wedge \iota_{2_*}\b)   \\&=\o(p;\a)\eta(q;\b). 
 \end{align*}
 It extends by linearity to Dirac chains, and by continuity to chains.
\end{proof}   

Cartesian wedge product extends to a continuous bilinear map \( \times: \pB_k(U_1) \times \pB_\ell(U_2) \to \pB_{k+\ell}(U_1 \times U_2)\)  and the relations (b)-(f) continue to hold.

\begin{example}\label{carte} Recall that if \( A \) is affine \( k \)-cell in \( U_1 \), then \( A \) is
represented by an element \( \widetilde{A} \in \pB_k(U_1) \). If \( A \) and \( B
\) are affine \( k \)- and \( \ell \)-cells in \( \R^n \) and \( \R^m \), respectively, then the
classical Cartesian wedge product \( A \times B \) is an affine \( (k+\ell) \)-cell in \(
\R^{n+m} \). The chain \( \widetilde{A \times B} \) representing \( A \times B \)
satisfies \( \widetilde{A \times B} = \widetilde{A} \times \widetilde{B} \).
\end{example}

\begin{remarks} \mbox{}
	 \begin{itemize} 
			\item The boundary relations hold if we replace
boundary \( \p \) with the directional boundary \( \p_v = P_v E_v^\dagger \) for \( v \in
\R^{n+m} \) where \( v = (v_1, 0), v_1 \in \R^n, 0 \in \R^m \) or \( v = (0, v_2), 0 \in
\R^n, v_2 \in \R^m \). 
			\item Cartesian wedge product is used to define a continuous convolution
product on differential chains in a sequel \cite{algebraic};
\item  According to Fleming (\cite{fleming}, \S 6), ``It is not possible to give a satisfactory definition of the cartesian product \( A \times B \) of two arbitrary flat chains.''  Fleming defines   ``Cartesian product'' on polyhedral chains, but this coincides with our Cartesian wedge product according to Example \ref{carte}. Cartesian wedge product is  well-defined   on currents \( {\cal D}' \) \cite{federer}.   
	 \end{itemize} 
\end{remarks}

\section{Fundamental theorems of calculus for chains in a flow} \label{sub:a_stokes}
                     
\subsection{Evolving chains} 
\label{sub:moving_pictures}

Let \( J_0 \in \hB_k^r(U) \) have compact support in \( U \) and \( V \in {\cal V}^r(U)\) where \( U \) is open in \( \R^n \). Let \( V_t \) be the time-\( t \) map of the flow of \( V \) and \( J_t:= V_{t*}J_0 \). For each \( p \in U \), the image \( V_t(p) \) is well-defined in \( U \) for sufficiently small \( t \). Since \( J \) has compact support, there exists \( t_0> 0 \) such that pushforward \( V_{t*} \) is defined on \( J \) for all \(0 \le t < t_0 \). Let \( \theta: U \times [0,t_0) \to U \) be given by \( \theta(p,t) := V_t(p) \). Then  \( \theta \in {\cal M}^{r+1}(U \times [0,t_0), U) \). By Theorem \ref{thm:opensets}, Corollary \ref{cor:pull}, Theorem \ref{sub:retraction}, and Theorem \ref{cart} we infer 
\begin{equation}\label{FTCPlus} \{J_t \}_a^b := \theta_* E_{e_{n+1}}^\dagger ( J_0 \times \widetilde{(a,b)} ) 
\end{equation}       
is a well-defined element of \( \hB_k^r \) where \( \widetilde{(a,b)} \) is the chain representing \( (a,b) \),  \( 0 \le a \le b < t_0 \), and \( e_{n+1} \in \R^{n+1} \) is unit. We know of no way to represent the domain \( \{J_t\}_a^b \) using classical methods, and yet this is a natural structure that often arises in mathematics and its applications -- a ``moving picture'' of \( J \) in the flow of the vector field \( V \).

  \begin{figure}[ht]
  	\centering
  		\includegraphics[height=1.75in]{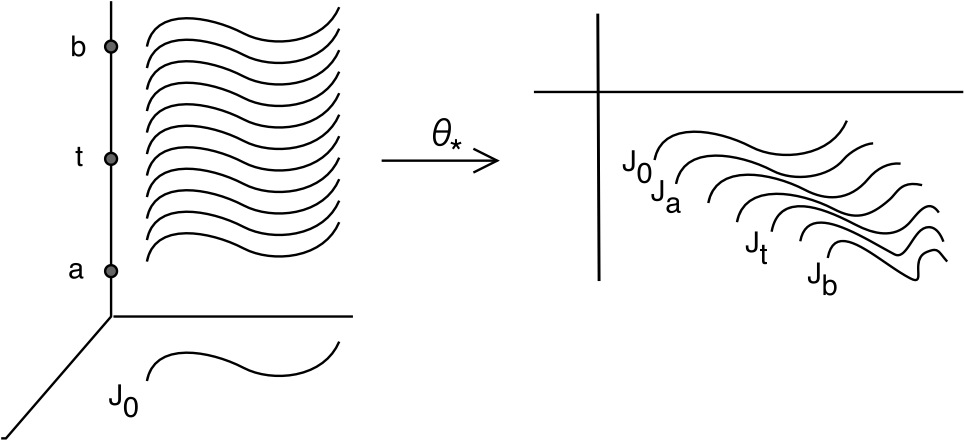}
  	\caption{A chain in a flow}
  	\label{fig:MovingChains}
  \end{figure}

 The construction of an evolving chain and the results below extend to \( J \in \pB_k(U) \)
with compact support in \( U \), and \( V \in {\cal V}(U) \) of class \( \B^\i \).

 \begin{lem}\label{lem:movingchain}  \mbox{}
 \begin{enumerate}
 	\item \( \theta_*(J_0 \times (t;1)) = V_{t*} J_0 \);
 	\item  \( P_V \theta_* = \theta_* P_{e_{n+1}} \);
 	\item  \( E_V^\dagger \p = \p E_V^\dagger \);
 	\item  \( \p\{J_t\}_a^b = \{ \p J_t\}_{a^b} \).
 \end{enumerate}
 \end{lem}

\begin{proof}\mbox{} 
	(a):  It suffices to prove this for \( J_0 = (p;\a) \), a simple \( k \)-element with \( p \in U \).  But \( \theta_*((p;\a) \times (t;1))  = \theta_*((p,t); \iota_{1*}\a) = V_{t*} (p;\a) \) since \( \theta(p,t) = V_t(p) \). 
	
	(b): By Theorem \ref{pro:mF} (c), \( \theta_* P_{e_{n+1}} = P_{\theta e_{n+1}} \theta_* = P_V \theta_* \) 
	
	(c): \vspace{-.2in} 
		\begin{align*}
			 E_V^\dagger \p (p;\a) = \sum _{i=1}^k(-1)^{i+1} E_V^\dagger (p; v_i \otimes \a_i) 
			&=  \sum_{i=1}^k  \sum_{i_j=1}^{k-1} (-1)^{i+1}(-1)^{i_j+1}\<V(p), v_{i_j} \> (p; v_{i_j} \otimes \a_{i_j}) \\&= \sum_{i=1}^k (-1)^{i+1}  \<V(p),v_i\> \p (p;\a_i) = \p E_V^\dagger (p;\a)
		\end{align*}
	
(d):  \vspace{-.4in}

\begin{align*}
   \p\{J_t\}_a^b &= \p \theta_* E_{e_{n+1}}^\dagger(J_0 \times \widetilde{(a,b)}) && \text{by \eqref{FTCPlus} }  \\&
 =     \theta_*\p E_{e_{n+1}}^\dagger(J_0 \times \widetilde{(a,b)})  && \text{by Proposition \ref{prop:whitneybd}}  \\&
    =   \theta_* E_{e_{n+1}}^\dagger \p (J_0 \times \widetilde{(a,b)})  && \text{by (c)}   
 \\& =   \theta_* E_{e_{n+1}}^\dagger (\p J_0 \times \widetilde{(a,b)} - J_0 \times ((b;1) -  (a;1)) ) && \text{by Theorem \ref{cart}(b)}  
 \\& = \{ \p J_t\}_{a^b}  &&\text{by \eqref{FTCPlus} and  }\\&\quad &&     \text{since } E_{e_{n+1}}^\dagger( J_0 \times ((b;1) -  (a;1))) = 0. 
\end{align*}   
\end{proof}
\vspace{-.4in}
\subsection{Fundamental theorem} 
\label{sub:fundamental_theorem_of_calculus_for_moving_chains}

\begin{thm}[Fundamental theorem for chains in a flow]\label{thm:Lieder} Suppose \( J \in \hB_k^r(U) \) is a differential chain with compact support in \( U \), \( V \in {\cal V}^r(U)\) is a vector field, and \( \o \in \B_k^{r+1}(U) \) is a differential form where \( U \) is open in \( \R^n \). Then 
\[ \cint_{J_b} \o - \cint_{J_a} \o = \cint_{\{J_t\}_a^b} L_V \o. \] 
\end{thm}

\begin{proof} 
 
 Since Dirac chains are dense and both \( \p_v \) and Cartesian wedge product are
continuous, we conclude \( \p_v( \widetilde{(a,b)} \times J) = \{b\} \times J - \{a\}
\times J \). Finally we change coordinates to consider a local flow chart of \( V \): By
Proposition \ref{pro:mF} we know \( \p_{\theta e_1} \theta_* = \theta_* \p_{e_1} \).
Using Corollary \ref{cor:daggerE} we deduce
\begin{align*} 
	  \cint_{\{J_t\}_a^b} L_V \o  &= \cint_{P_V \{J_t\}_a^b} \o &&\text{by Theorem \ref{thm:predfirst}} 
	 \\& = \cint_{P_V\theta_* E_{e_{n+1}}^\dagger ( J_0 \times \widetilde{(a,b)})} \o  &&\text{by Definition \ref{FTCPlus}}
\\&   = \cint_{\theta_* P_{e_{n+1}}E_{e_{n+1}}^\dagger (J_0 \times \widetilde{(a,b)})} \o &&\text{by Proposition \ref{pro:mF}} 
	\\&
	= \cint_{ \theta_* \p_{e_{n+1}}(J_0 \times \widetilde{(a,b)}) } \o &&\text{by the definition of directional boundary p.\pageref{cor:stokes}} \\&
	= \cint_{\theta_*(J_0 \times (b;1) )- \theta_*(J_0 \times (a;1) ) } \o &&\text{since } \p(\widetilde{(a,b)}) = (b;1)-(a;1) \\&
	= \cint_{V_{b*}J} \o -\cint_{V_{a*}J} \o && \quad. 
\end{align*} 
\end{proof}

\begin{remarks}\mbox{} 
	\begin{itemize}
		\item   If \( J_a \) is a \( 0 \)-element, then this is the fundamental theorem of calculus for integral curves.  
		\item   If \( J \) is a smooth submanifold, the integral on the left hand side of \eqref{thm:Lieder} can be  stated in the classical theory.  
		\item The integral on the right hand side of \eqref{thm:Lieder} looks to be entirely new, and requires a large fraction of this paper to work as it does.  
	\end{itemize}
      
\end{remarks}

\begin{thm}[Stokes' theorem for evolving chains]\label{cor:st} Let \( \o \in \B_{k-1}^{r+1}(U) \) be
a differential form, \( J \in \hB_k^r(U) \) a differential chain, and \( V \in
{\cal V}^r(U) \) a vector field. Then \[ \cint_{\{J_t\}_a^b} d L_V \o = \cint_{\p J_b}
\o - \cint_{\p J_a} \o. \] 
\end{thm} 

\begin{proof} 
This follows directly from Stokes' Theorem \ref{cor:stokes},  Theorem \ref{thm:Lieder}, and Lemma  \ref{lem:movingchain} (d).
\end{proof}

\begin{cor}\label{cor:exactL} Let \( \o \in \B_{k-1}^{r+1}(U) \) be a differential
form, \( J \in \hB_k^r(U) \) a differential chain, and \( V \in {\cal V}^r(U)
\) a vector field. If \( L_V \o \) is closed, then \( \cint_{\p J_b} \o = \cint_{\p J_a}
\o \). 
\end{cor} 

\subsection{Leibniz integral rule} % (fold)
\label{sub:leibniz_integral_rule}

We say that a one-parameter family of differential $k$-chains \( J_t \) is \emph{differentiable in time} if \( \frac{\p}{\p t} J_t  = \lim_{h \to 0} (J_{t+h} - J_t) /h  \) is well-defined in \(  \in \pB_k \).  Examples abound.  If \( J_t \) is an \emph{evolving differential $k$-chain}, that is,  the pushforward of a $k$-chain \( J_0 \) via the time-\( t \) map of a vector field \( V \),  flow,  it follows directly from the definitions   that 
\begin{equation}\label{ptJ}
	 \frac{\p}{\p t} J_t = P_V J_t.
\end{equation}

 Our next result extends the classical Leibniz Integral Rule to a coordinate free version which holds for rough and smooth domains, alike, none of which need to be parametrized.  
\begin{thm}[Generalized Leibniz Integral Rule] \label{thm:differenint}   If \( \{J_t\} \) is differentiable in time, then
 \[ \frac{\p}{\p t} \cint_{J_t} \o_t = \cint_{\frac{\p}{\p t} J_t} \o_t + \cint_{J_t} \frac{\p}{\p t} \o_t. \]
\end{thm} 

\begin{proof} 
\begin{align*}
  \frac{\p}{\p t} \cint_{J_t} \o_t = \lim_{h \to 0}  \cint_{\frac{J_{t+h} - J_{t}}{h}} \o_t + \cint_{J_t} \frac{\o_{t+h} - \o_t}{h} =  \cint_{\frac{\p}{\p t} J_t} \o_t + \cint_{J_t} \frac{\p}{\p t} \o_t.
\end{align*} 

\end{proof} If \( J_t \) is constant with respect to time, we get \(  \frac{\p}{\p t} \cint_{J_t} \o_t =   \cint_{J_t} \frac{\p}{\p t} \o_t   \), as expected.

 The simplicity of this foundational result belies its power.  The definition of \(  \frac{\p}{\p t} J_t \) is simple enough, but without the copious supply of examples to which it applies (given by prederivative \( P_V \)), it would mean little.  

We immediately deduce from Equation \ref{ptJ}, Corollary \ref{thm:differenint} and the duality \( L_v \o = \o P_V \) (see \S\ref{sub:prederivative}):
\begin{cor}[Differentiating the Integral]\label{cor:differintegral}  Suppose \( J_t \) is an evolving differential chain under the flow of a vector field \( V \).   Then
\[
	\frac{\p}{\p t} \cint_{J_t} \o_t =     \cint_{J_t}\left( \frac{\p}{\p t} \o_t + L_V \o_t\right).
\]
\end{cor}
The classical method of differentiating the integral has been extremely useful in engineering and physics\footnote{Feynman wrote in his autobiography \cite{joking}, ``That book [Advanced Calculus, by Wood.] also showed how to differentiate parameters under the integral sign —- it's a certain operation. It turns out that's not taught very much in the universities; they don't emphasize it. But I caught on how to use that method, and I used that one damn tool again and again. So because I was self-taught using that book, I had peculiar methods of doing integrals.
The result was, when guys at MIT or Princeton had trouble doing a certain integral, it was because they couldn't do it with the standard methods they had learned in school. If it was contour integration, they would have found it; if it was a simple series expansion, they would have found it. Then I come along and try differentiating under the integral sign, and often it worked. So I got a great reputation for doing integrals, only because my box of tools was different from everybody else's, and they had tried all their tools on it before giving the problem to me.''}.  
From  the duality  \( d\o = \o \p \) (see Theorem \ref{cor:stokes}) and  Cartan's Magic Formula \( L_V = d i_V + i_V d \)  we  immediately deduce:
\begin{cor}\label{cor:}
	If \( \o_t \) is closed, then  
		\[
			\frac{\p}{\p t} \cint_{J_t} \o_t =     \cint_{J_t} \frac{\p}{\p t} \o_t +  \cint_{\p J_t} i_V \o_t.
		\] If \( J_t \) is a cycle, then 
		\[
			\frac{\p}{\p t} \cint_{J_t} \o_t =     \cint_{J_t} \frac{\p}{\p t} \o_t +  \cint_{  J_t} i_V d \o_t.
		\]	
\end{cor}

\begin{cor}[Generalization of the Reynold's Transport Theorem]\label{cor:Reynolds} If \( \o_t \) and \( J_t \) are top dimensional, then 
	\[
		\frac{\p}{\p t} \cint_{J_t} \o_t =   \cint_{J_t} \frac{\p}{\p t} \o_t + \cint_{\p J_t}i_V \o_t. 
	\]
\end{cor}  The classical version, a cornerstone of mechanics and fluid dynamics, is a straightforward consequence.  Our result applies to a wide range of domains of integration, many of which are discussed in this paper.

\section{Differential chains in manifolds} \label{sec:chains_in_manifolds}  

The Levi-Civita connection and metric can be used to extend this theory to open subsets of Riemannian manifolds  (see \cite{thesis}, as well).  We sketch some of the main ideas here.  Let  \( {\cal R}_{U,M} \) be the category whose objects are pairs \( (U, M) \) of open subsets \( U \) of  Riemannian manifolds \( M \), and morphisms are smooth maps.  Let \( TVS \) be the category whose objects are locally convex topological vector spaces, and morphisms are continuous linear maps between them.  Then  \( \pB \) is a functor from  \( {\cal R}_{U,M} \) to \( TVS \). We define \( \pB(U) =  \pB(U, M) \) much as we did for \( M = \R^n \), but with a few changes:   Norms of tangent vectors and masses of \( k \)-vectors are defined using the metric.  The connection can be used to define \( B^r \) norms of vector fields.  Dirac \( k \)-chains are defined as formal sums \( \sum (p_i; \a_i) \) where \( p_i \in U \), \( \a_i \in \L_k(T_{(p_i)}(M)) \).  The vector space of Dirac \( k \)-chains is \( \A_k(U, M) \).  Difference chains are defined using pushforward along the flow of   locally defined, unit vector fields.    The symmetric algebra is replaced by the universal enveloping algebra.  From here, we can define the \( B^r \) norms on \( \A_k(U, M) \), and the \( B^r \) norms on forms, and thus the topological vector spaces  \( \pB(U, M) \) and \( \B_k(U,M) \).  Support of a chain is a subset of \( M \) together with the attainable boundary points of \( M \). Useful tools include naturality of the operators    \S\ref{ssub:naturality_of_the_operators}, partitions of unity \S\ref{ssub:partitions_of_unity}, and cosheaves.

\addcontentsline{toc}{section}{References} 
\bibliography{Jennybib.bib, mybib.bib}{}
\bibliographystyle{amsalpha}

\bibliographystyle{amsalpha} 
\end{document}